\documentclass[12pt,twoside, a4]{amsart}

\usepackage{mabliautoref}
\usepackage{amssymb,amsthm,amsmath}
\RequirePackage[dvipsnames, usenames]{xcolor}
\usepackage{hyperref}
\usepackage{mathtools}
\usepackage{minibox}
\usepackage{bbding}
\usepackage[abbrev,alphabetic]{amsrefs}
\usepackage[all]{xy}
\usepackage{tikz}
\usepackage{tikz-cd}
\usepackage{comment}
\usepackage[T1]{fontenc}
\usepackage[utf8]{inputenc}
\usepackage[obeyDraft]{todonotes}
\setuptodonotes{color=magenta!50}

\calclayout

\definecolor{blush}{rgb}{0.87, 0.36, 0.51}
\definecolor{jazzberryjam}{rgb}{0.65, 0.04, 0.37}
\definecolor{tiffanyblue}{rgb}{0.04, 0.73, 0.71}
\definecolor{darkcyan}{rgb}{0.0, 0.55, 0.55}

\hypersetup{
	bookmarks,
	bookmarksdepth=3,
	bookmarksopen,
	bookmarksnumbered,
	pdfstartview=FitH,
	colorlinks,backref,hyperindex,
	linkcolor=jazzberryjam,
	anchorcolor=BurntOrange,
	citecolor=MidnightBlue,
	citecolor=tiffanyblue,
	filecolor=darkcyan,
	menucolor=Yellow,
	urlcolor=tiffanyblue
}

\usetikzlibrary{decorations.markings}
\tikzset{double line with arrow/.style args={#1,#2}{decorate,decoration={markings,%
mark=at position 0 with {\coordinate (ta-base-1) at (0,1pt);
\coordinate (ta-base-2) at (0,-1pt);},
mark=at position 1 with {\draw[#1] (ta-base-1) -- (0,1pt);
\draw[#2] (ta-base-2) -- (0,-1pt);
}}}}
\tikzset{Equal/.style={-,double line with arrow={-,-}}}

\makeatletter
\let\save@mathaccent\mathaccent
\newcommand*\if@single[3]{%
  \setbox0\hbox{${\mathaccent"0362{#1}}^H$}%
  \setbox2\hbox{${\mathaccent"0362{\kern0pt#1}}^H$}%
  \ifdim\ht0=\ht2 #3\else #2\fi
  }
\newcommand*\rel@kern[1]{\kern#1\dimexpr\macc@kerna}
\newcommand*\widebar[1]{\@ifnextchar^{{\wide@bar{#1}{0}}}{\wide@bar{#1}{1}}}
\newcommand*\wide@bar[2]{\if@single{#1}{\wide@bar@{#1}{#2}{1}}{\wide@bar@{#1}{#2}{2}}}
\newcommand*\wide@bar@[3]{%
  \begingroup
  \def\mathaccent##1##2{%
    \let\mathaccent\save@mathaccent
    \if#32 \let\macc@nucleus\first@char \fi
    \setbox\z@\hbox{$\macc@style{\macc@nucleus}_{}$}%
    \setbox\tw@\hbox{$\macc@style{\macc@nucleus}{}_{}$}%
    \dimen@\wd\tw@
    \advance\dimen@-\wd\z@
    \divide\dimen@ 3
    \@tempdima\wd\tw@
    \advance\@tempdima-\scriptspace
    \divide\@tempdima 10
    \advance\dimen@-\@tempdima
    \ifdim\dimen@>\z@ \dimen@0pt\fi
    \rel@kern{0.6}\kern-\dimen@
    \if#31
      \overline{\rel@kern{-0.6}\kern\dimen@\macc@nucleus\rel@kern{0.4}\kern\dimen@}%
      \advance\dimen@0.4\dimexpr\macc@kerna
      \let\final@kern#2%
      \ifdim\dimen@<\z@ \let\final@kern1\fi
      \if\final@kern1 \kern-\dimen@\fi
    \else
      \overline{\rel@kern{-0.6}\kern\dimen@#1}%
    \fi
  }%
  \macc@depth\@ne
  \let\math@bgroup\@empty \let\math@egroup\macc@set@skewchar
  \mathsurround\z@ \frozen@everymath{\mathgroup\macc@group\relax}%
  \macc@set@skewchar\relax
  \let\mathaccentV\macc@nested@a
  \if#31
    \macc@nested@a\relax111{#1}%
  \else
    \def\gobble@till@marker##1\endmarker{}%
    \futurelet\first@char\gobble@till@marker#1\endmarker
    \ifcat\noexpand\first@char A\else
      \def\first@char{}%
    \fi
    \macc@nested@a\relax111{\first@char}%
  \fi
  \endgroup
}
\makeatother
\newcommand{\wb}{\widebar}

\numberwithin{equation}{section}

\makeatletter
\newcommand{\leqnomode}{\tagsleft@true\let\veqno\@@leqno}
\newcommand{\reqnomode}{\tagsleft@false\let\veqno\@@eqno}
\makeatother

\newcommand{\numberset}{\mathbb}
\newcommand{\N}{\numberset{N}}

\newcommand{\Q}{\numberset{Q}}
\newcommand{\R}{\numberset{R}}

\newcommand{\bQ}{\numberset{Q}}

\newcommand{\bk}{k}

\newcommand{\Pn}{\numberset{P}} 
\newcommand{\bP}{\numberset{P}}

\newcommand{\A}{\numberset{A}} 

\newcommand{\cO}{\mathcal{O}}

\newcommand{\cF}{\mathcal{F}}
\newcommand{\cG}{\mathcal{G}}

\newcommand{\cC}{\mathcal{C}}
\newcommand{\cB}{\mathcal{B}}

\newcommand{\cW}{\mathcal{W}}

\DeclareMathOperator{\Spec}{Spec}

\DeclareMathOperator{\Supp}{Supp}
\DeclareMathOperator{\supp}{Supp}
\DeclareMathOperator{\Exc}{Exc}

\DeclareMathOperator{\red}{red}
\DeclareMathOperator{\codim}{codim}

\DeclareMathOperator{\Bs}{Bs}

\begin{document}

\title{On the canonical bundle formula in positive characteristic}

\author{Marta Benozzo}
\email{marta.benozzo.20@alumni.ucl.ac.uk}
\address{University College London, London, WC1E 6BT, UK}
\curraddr{Université Paris-Saclay, Orsay, 91405, France}


\begin{abstract}
    Let $f \colon X \rightarrow Z$ be a fibration from a normal projective variety $X$ of dimension $n$ onto a normal curve $Z$ over a perfect field of characteristic $p>2$. 
    Let $(X, B)$ be a dlt pair such that the induced pair on a general fibre is log canonical.
    Assuming the LMMP and the existence of log resolutions in dimension $\leq n$, we prove that, when $K_X+B$ is $f$-nef, the moduli part is nef up to a birational map $Y \dashrightarrow X$.
    As a corollary, we prove positivity of the moduli part in the $K$-trivial case, i.e.\ when $K_X+B \sim_{\Q} f^*L$ for some $\Q$-Cartier $\Q$-divisor $L$ on $Z$.
    In particular, consider a dlt pair $(X, B)$ of dimension $3$ over an algebraically closed field of characteristic $p>5$ such that the induced pair on a general fibre is log canonical, then the canonical bundle formula holds unconditionally.
\end{abstract}

\maketitle

\setcounter{tocdepth}{1}
\tableofcontents

\section*{Introduction}

The classification of varieties is one of the guiding problems in algebraic geometry.
An important tool that birational geometers use to this end is the study of positivity properties of the canonical divisor.
A natural question in the field is whether we can meaningfully relate the canonical divisors of the source and the target of a fibration.
The canonical bundle formula tackles this problem.
Kodaira's result on elliptic fibrations is the first instance of a formula in this direction (see for example \cite[Theorem 8.2.1, Chapter 8]{KollarCBF}).
It states that, given an elliptic fibration $f\colon X \rightarrow Z$ from a normal projective surface over an algebraically closed field of characteristic $0$ such that $K_X \sim_{\Q} f^*L$ for some $\Q$-Cartier divisor $L$ on $Z$, then
\[
    K_X \sim_{\Q} f^*(K_Z+B_Z+M_Z),
\]
where $B_Z$ is an effective $\Q$-divisor which has an explicit description in terms of the singularities of the fibres, while $M_Z$ is a $\Q$-divisor defined via the $j$-invariant of the fibres, a number that classifies elliptic curves.
More precisely, if $j\colon Z \to \bP^1$ is the map that generically sends $z \in Z$ to the $j$-invariant of the elliptic curve $f^{-1}(z)$, then $M_Z$ is a positive rational multiple of $j^*\cO_{\bP^1}(1)$.
Over fields of positive characteristic, a similar formula holds with an additional term that takes into account the possible presence of \emph{wild fibres} (\cite[Theorem 2]{Enriques}).

Later, a similar formula was proven for \emph{$K$-trivial} fibrations in characteristic $0$.
If $(X, B)$ is a log canonical pair and $f\colon X \rightarrow Z$ is a fibration such that $K_X+B \sim_{\Q} f^*L$, for some $\Q$-Cartier $\Q$-divisor $L$ on $Z$, then we can write $L= K_Z+B_Z+M_Z$.
The divisor $B_Z$ is defined according to the singularities of $f$, whereas $M_Z$ measures how far the fibration is from being isotrivial, namely its \emph{variation}.
In general, it is difficult to construct a moduli space for the fibres, so $M_Z$ does not have an explicit description as in the elliptic curves case.
However, we can at least study whether it defines a meaningful map from $Z$.
The first step is looking at the positivity properties of $M_Z$.
These are essential to understand especially for inductive purposes: if we control $B_Z$ and $M_Z$, we can infer properties of $X$ from the study of $Z$, which is lower dimensional.
In this spirit, this formula has had many important applications in birational geometry, ranging from adjunction on log canonical centres (\cite{Kaw98,DS}), to finite generation of the log canonical ring (\cite[\S 5]{FujinoMori}).
Using variation of Hodge structures, it is possible to show that, under further smoothness assumptions, $f_*\omega_{X/Z}$ is \textit{semipositive} (see \cite[\S 8.9]{KollarCBF} for an account of the details).
Given this, in \cite{Kaw98,FujinoMori,Shokurov,modulibdiv,KollarCBF}, the authors proved that, up to a birational base change, $M_Z$ is nef. Moreover, we can control the singularities of the induced pair on $Z$ (\cite[Theorem 3.1]{Shokurov}).
Unfortunately, these techniques cannot be used over fields of positive characteristics. In fact, over any algebraically closed field of characteristic $p>0$, it is possible to construct fibrations $f\colon X \to Z$, where $X$ is a smooth surface and $Z:= \bP^1$, called Moret--Bailly families, which give counterexamples to the positivity of $f_*\omega_{X/Z}$ (\cite{MoretBailly}, \cite[Proposition 3.16]{SZ20}).

With the development of the theory of \textit{$F$-singularities}, which take into account both the geometry and the arithmetic of the varieties, it was possible to prove a canonical bundle formula using more algebraic techniques.
Indeed, if we ask the pair induced on the generic fibre to be \textit{globally $F$-split} (\cite[Definition 2.16]{DS}), the splitting allows to descend effectiveness on the base. More precisely, $K_X+B \sim_{\Q} f^*(K_Z+\Delta_Z)$, where $\Delta_Z$ is an effective, canonically defined $\Q$-divisor (\cite[Theorem 5.2]{DS}, \cite[Theorem 3.17]{Ejiri}).
Moreover, since moduli spaces of curves are well understood in any characteristic, when the fibration has relative dimension one we can exploit them to get the desired positivity (\cite[Lemma 6.6, Lemma 6.7]{BPF}, \cite[Theorem 3.2]{Jakub}).
For this result we need to assume that the geometric generic fibre is smooth and the pair induced on it is log canonical, whereas in general the statement is known to be false (\cite[Example 3.5]{Jakub}).
If only the generic fibre is log canonical, it is still possible to prove a weaker statement for fibrations of relative dimension one by considering purely inseparable covers of the base (\cite[Theorem 1]{Jakub}, \cite[Theorem 6.2]{CBF_inseparable}).

In the paper \cite{positivity}, a new approach has been taken in characteristic $0$.
The advantage is that this approach does not make use of variation of Hodge structures, but rather techniques coming from the minimal model program (MMP).
In order to have the necessary flexibility, the authors of \cite{positivity} consider more general fibrations $f\colon X \to Z$, not only $K$-trivial ones. In this context, it is still possible to define a discriminant part $B_Z$ on $Z$ measuring the singularities of $f$ and then the moduli part $M_X$ is a $\Q$-divisor defined on $X$.
The problem is again to prove that the moduli part has some geometric meaning linked to the variation in moduli of the fibres. More specifically, the main result of the paper is that, given $f\colon X \to Z$ a fibration between normal projective varieties over an algebraically closed field and $B$ an effective $\Q$-divisor such that the pair $(X, B)$ is log canonical and $K_X+B$ is $f$-nef, the moduli part $M_X$ is nef, up to passing to a different birational model of the fibration (\cite[Theorem 1.1]{positivity}).

Recently, there has been a lot of developments on the MMP in positive characteristic.
In particular, it has been proven that we can run the MMP in dimension at most $3$ over fields of characteristic $p>5$ (\cite{HX,Bir,BW,MMPlc}), and we have partial results in low characteristics (\cite{low_char_MMP2,low_char_MMP}). Moreover, log resolutions have been constructed in dimension at most $3$.
In view of this, it is natural to ask whether the techniques in \cite{positivity} could be adapted to the positive characteristic setting.

One of the main problems to follow the strategy in characteristic $0$ is that, due to the failure of generic smoothness, it is not even possible to define a discriminant part of the fibration as all the fibres may be very singular.
The main result of this paper is that, assuming that the fibration has log canonical general fibres, a canonical bundle formula holds over perfect fields of positive characteristic when the base is a curve, conjecturally to the MMP.

\begin{theorem}[see \autoref{t-goal}] \label{t-goal_intro}
    Assume the log MMP, inversion of adjunction and the existence of log resolutions in dimension up to $n$.
    Let $f\colon X \rightarrow Z$ be a fibration from a normal projective $\Q$-factorial variety $X$ of dimension $n$ onto a normal projective curve $Z$ defined over a perfect field $k$ of characteristic $p>2$. Let $B \geq 0$ be a $\Q$-divisor such that $(X, B)$ is dlt and the pair induced on a general fibre has at most log canonical singularities.
    Suppose that $K_X+B$ is $f$-nef.
    Then, there exist a pair $(Y, C)$ and a commutative diagram
    \[
    \xymatrix{
    Y \ar@{-->}[r]^{b} \ar[dr]_{g} & X \ar[d]^f \\
     & Z,
    }
    \]
    where $b$ is a birational map such that
\begin{itemize}
    \item[(i)] $(X/Z, B)$ and $(Y/Z, C)$ are crepant over the generic point of $Z$ (see \autoref{d-generically crepant});
    \item[(ii)] the moduli part $M_Y$ of $(Y/Z, C)$ is nef.
    \end{itemize}
\end{theorem}

The proof of positivity of the moduli part in \cite{positivity} follows two main steps: first, the statement is proven for a specific class of fibrations with particularly good singularities (satisfying \textit{Property $(\ast)$}, as in \autoref{d-Prop star}), and then the general case is reduced to this class with a birational base change (a \textit{$(\ast)$-modification}, as in \autoref{t-property star modifications v1.2}).
Using Weak Semistable reduction (\cite[Theorem 2.1, Proposition 4.4]{AK}), over fields of characteristic $0$ it is possible to construct $(\ast)$-modifications in any dimension. However, over fields of positive characteristic, this result has not been proven in such generality. Nonetheless, if the base of the fibration is a curve, we construct $(\ast)$-modifications using log resolutions and the MMP.
Another difficulty is that, in the proof of the main result, the authors of \cite{positivity} proceed by induction on the dimension by producing a log canonical centre and performing adjunction on it.
In positive characteristic however, due to the failure of \textit{generic smoothness} (\cite[Tag 056V]{stacks}), this is not enough: the general fibres of the induced fibration could be very singular.
The failure of generic smoothness stems from the fact that the generic fibre of a fibration $f\colon X \rightarrow Z$ over a field of positive characteristic is defined over an \emph{imperfect} field.
However, this is somehow the only obstacle and, after a purely inseparable base change, all the singularities of the general fibres appear also on the total space.
More precisely, we consider the diagram:
\leqnomode
\begin{equation}\tag{\raisebox{-0.5ex}{\Snowflake}} \label{s-intro fibres}
    \begin{tikzcd}
    X^{(e)} \arrow[d,"f_e", swap] \arrow[r] & X \arrow[d, "f"] \\
    Z^e \arrow[r, "F^e"] & Z,
    \end{tikzcd}
\end{equation}
\reqnomode
where $F^e$ is the $e$\textsuperscript{th}-power of the Frobenius morphism and $X^{(e)}$ is the normalisation of the reduction of the fibre product.
The resulting fibration $f_e$ is (universally) homeomorphic to $f$, but its general fibres are normal for $e \gg 0$. Moreover, if $\widebar{\eta}$ is the geometric generic point of $Z$ and $\widebar{W} \subset X_{\widebar{\eta}}$ is a log canonical centre, there exists a log canonical centre $W \subset X^{(e)}$ which reduces to $\widebar{W}$.

\begin{theorem}[{Existence of geometric $(\ast)$-modifications, see \autoref{t-property star modifications v2.0}}]
Assume the log MMP and inversion of adjunction in dimension $n$, and the existence of log resolutions in dimensions $n$ and $n-1$.
Let $f\colon X \rightarrow Z$ be a fibration from a normal projective $\Q$-factorial variety $X$ of dimension $n$ onto a normal projective curve $Z$ defined over a perfect field of characteristic $p>0$, such that $X_{\widebar{\eta}}$ is normal, where $\widebar{\eta}$ is the geometric generic point of $Z$.
Then, for $e \gg 0$, there exists a {$(\ast)$-modification} of $X^{(e)}$.
\end{theorem}

For fibrations satisfying Property $(\ast)$, over fields of characteristic $0$, we can compare the moduli part with the canonical divisor of the \emph{foliation} induced by the fibration.
Over fields of characteristic $0$, to each fibration $f \colon X \to Z$ between normal varieties, we can associate a foliation $\cF$, defined as the saturation of the kernel of the differential map $df \colon T_X \to f^*T_Z$. If $f$ is equidimensional, the canonical divisor of the foliation, $K_{\cF}$, can be explicitly computed as ${K_X-f^*K_Z-R(f)}$, where $R(f)$ is the \emph{ramification divisor}, supported on fibres that have multiplicity at least $2$.
The moduli part, under Property $(\ast)$ assumptions, has a similar description.
On the other hand, over fields of positive characteristic, the foliation $\cF$ could behave very differently. For example, if $f$ is not \emph{separable}, the rank of $\cF$ is bigger than expected.
Moreover, even if we require the fibration $f$ to be separable and to have normal general fibres, the divisor $K_{\cF}$ has a different description due to the presence of possible \emph{wild components}, i.e.\ vertical divisors whose multiplicity is divisible by the characteristic.
All in all, we prove that, over perfect fields of positive characteristic, $$K_{\cF}=K_X-f^*K_Z -R(f)-W(f),$$ where $R(f)$ is the ramification divisor defined as in the characteristic $0$ case, while $W(f)$ is an effective divisor supported on the wild components (see \autoref{t-canonical foliation}).
Therefore, even if the moduli part does not always coincide with the canonical divisor of the foliation, under Property $(\ast)$ assumptions, we have a good control of the difference between the two.

In \cite{positivity}, positivity of the moduli part then follows from the the MMP for \emph{algebraically integrable} foliations, in particular from the Cone theorem \cite[Theorem 3.9]{positivity}.
Over fields of positive characteristic this latter result is known to fail (see \cite{counterexamples_Fabio}).
To overcome this issue, we exploit a correspondence between foliations and purely inseparable maps in positive characteristic (\cite[Proposition 2.9]{LMMP}). In fact, since the differential of the Frobenius morphism vanishes, purely inseparable maps define foliations as the kernel of their differential.
Analysing this correspondence in the diagram \autoref{s-intro fibres}, we relate the canonical divisor of $X^{(e)}$ to the moduli part of $f$ and then use the ``standard'' MMP on $X^{(e)}$ to show positivity of the moduli part.
The next example shows this idea.

\begin{example}
    Let $X$ be the variety defined by $y^2z^2 s-x^4 s-x^2y^2t-z^4 t=0$ inside $\bP^2_{[x:y:z]}\times \bP^1_{[s:t]}$ over a perfect field of characteristic $p\neq 2$ and let $f\colon X \to \bP^1$ be the induced projection.
    Let $\cF$ be the foliation induced by $f$, then
    \[
        K_X= \cO_{\bP^2\times\bP^1}(1,-1)|_X \quad \text{and} \quad K_{\cF}= \cO_{\bP^2\times\bP^1}(1,1)|_X.
    \]
    Now, let us perform a Frobenius base change as in diagram \autoref{construction}: the resulting variety $X^{(e)}$ is defined by $y^2z^2 \sigma^{p^e}-x^4\sigma^{p^e}-x^2y^2\tau^{p^e}-\tau^{p^e}$ inside $\bP^2_{[x:y:z]}\times \bP^1_{[\sigma:\tau]}$, where $t= \tau^{p^e}$ and $s= \sigma^{p^e}$. Therefore,
    \[
    K_{X^{(e)}}= \cO_{\bP^2\times\bP^1}(1, p^e-2)|_X.
    \]
    We observe that the divisor $K_X$ is not nef, while both $K_{\cF}$ and $K_{X^{(e)}}$ are.
\begin{remark}
    In fact, when performing a Frobenius base change, for $e \gg 0$ the positivity of $K_{X^{(e)}}$ reflects the positivity of $K_{\cF}$ (see \autoref{t-formulina_general}).
\end{remark}
    \end{example}

\begin{theorem}[{See \autoref{r-formulina}}]
Assume inversion of adjunction in dimension $n$, and the existence of log resolutions in dimension $d$.
Let $f \colon X \to Z$ be a separable equidimensional fibration between normal varieties such that $\dim(X)-\dim(Z)=d$. Let $B$ be an effective $\Q$-divisor such that $(X/Z, B)$ satisfies Property $(\ast)$ and let $M_X$ be the associated moduli part.
Let $\alpha_e \colon X \to X^{(e)}$ be the purely inseparable morphism such that $X \xrightarrow{\alpha_e} X^{(e)} \to X$ is $F^e$.
Then,
\[
    \alpha_e^*K_{X^{(e)}} = (p^e-1)(M_X-B^h) + K_X - W_e,
\]
where $W_e$ is an effective divisor supported on the wild components of $f$.
\end{theorem}

As a corollary of the main result \autoref{t-goal_intro}, we obtain a canonical bundle formula in the $K$-trivial case in a manner similar to \cite[Theorem 1.3]{positivity}.

\begin{theorem}[see \autoref{t-ftrivial}]
    Assume the log MMP, inversion of adjunction and the existence of log resolutions in dimension up to $n$.
    Let $f\colon X \rightarrow Z$ be a fibration from a normal projective $\Q$-factorial variety $X$ of dimension $n$ onto a normal projective curve $Z$ defined over a perfect field of characteristic $p>2$, and let $(X, B)$ be a dlt pair with $B \geq 0$ such that the pair induced on a general fibre is log canonical.
    Assume that $K_X+B \sim_{\Q} f^*L$ for some $\Q$-Cartier $\Q$-divisor $L$ on $Z$.
    Let $M_Z:= L-(K_Z+B_Z)$, where $B_Z$ is the discriminant part of $(X/Z, B)$.
    Then, $M_Z$ is nef.
\end{theorem}

If $X$ is a threefold over an algebraically closed field of characteristic $p>5$, we know the existence of log resolutions and that we can run the MMP for log canonical pairs, so we obtain the following corollary.

\begin{corollary}[see \autoref{c-3folds}]
    Let $f\colon X \rightarrow Z$ be a fibration from a normal projective $\Q$-factorial variety $X$ of dimension $3$ onto a normal projective curve $Z$ over an algebraically closed field of characteristic $p>5$.
    Let $(X, B)$ be a dlt pair with $B \geq 0$ and such that the pair induced on a general fibre is log canonical.
    Suppose that $K_X+B$ is $f$-nef.
    Then, there exist a pair $(Y, C)$ which satisfies {Property $(\ast)$}, with $C \geq 0$, and a commutative diagram
    \[
    \xymatrix{
    Y \ar@{-->}[r]^{b} \ar[dr]_{g} & X \ar[d]^f \\
     & Z,
    }
    \]
    with $b$ birational and such that
    \begin{itemize}
    \item[(i)] $(X/Z, B)$ and $(Y/Z, C)$ are crepant over the generic point of $Z$ (see \autoref{d-generically crepant});
    \item[(ii)] the moduli part $M_Y$ of $(Y/Z, C)$ is nef.
    \end{itemize}
    Moreover, if $K_X+B \sim_{\Q} f^*L$ for some $\Q$-Cartier $\Q$-divisor $L$ on $Z$, define $M_Z:= L-(K_Z+B_Z)$, where $B_Z$ is the discriminant part of $(X/Z, B)$.
    Then $M_Z$ is nef.
\end{corollary}

In \cite[Proposition 3.2]{Jakub}, the author proves a weak canonical bundle formula for fibrations of relative dimension $1$ with smooth log canonical fibres.
This result, together with the above \autoref{c-3folds}, completes the picture in dimension $3$ for fibrations with log canonical general fibres over algebraically closed fields of characteristic $p>5$.

\subsection*{Acknowledgements}
I would like to thank my PhD advisor Paolo Cascini for suggesting the problem, for his guidance throughout and for his indispensable support.
I would like to thank Jefferson Baudin, Fabio Bernasconi, Aurore Boitrel, Federico Bongiorno, Iacopo Brivio, Ivan Cheltsov, Samuele Ciprietti, Enrica Floris, Przemysław Grabowski, Martin Ortiz, Quentin Posva, Calum Spicer, Stefania Vassiliadis, Pascale Voegtli, Jakub Witaszek and the anonymous referees for very helpful discussions and comments on earlier drafts.
I thank the LSGNT and Imperial College London for providing a great environment for discussions.
This work was supported by the Engineering and Physical Sciences Research Council [EP/S021590/1], the EPSRC Centre for Doctoral Training in Geometry and Number Theory (The London School of Geometry and Number Theory), University College London.
I would also like to thank the NCTS, National Centre for Theoretical Sciences Mathematics division, for their support during my stay in Taipei.

\subsection*{Statements and declarations} 
The corresponding author states that there is no conflict of interest.
Data sharing not applicable to this article as no datasets were generated or analysed during the current study.

\section{Notation and conventions}
\begin{itemize}
    \item We fix $k$ to be a perfect field of characteristic $p>0$, unless otherwise stated.
    \item A variety $X$ over $k$ is an integral and separated scheme of finite type over $k$. We will often omit to mention the base field $k$.
    \item We denote the function field of a variety $X$ by $k(X)$.
    \item We denote by $\nu\colon X^\nu\to X$ the normalisation morphism of a variety $X$.
    \item We denote by $X_{\mathrm{red}}$ the reduced subscheme of a scheme $X$.
    \item If $X$ is a variety, by a \textit{big open subset} $U \subseteq X$, we mean an open dense subset $U \subseteq X$ such that $\codim(X \setminus U) \geq 2$.
    \item Given a $\Q$-divisor $D$ on $X$, its support $\Supp(D)$ is the union of all its prime components.
    \item Given a $\Q$-divisor $D=\sum_{i=1}^Nd_iD_i$, with $d_i \in \Q$ and $D_i$ prime divisors, we define $\lfloor D \rfloor := \sum_{i=1}^N \lfloor d_i \rfloor D_i$.
    \item Let $f\colon X \rightarrow Z$ be a projective morphism between normal varieties over any field.
    It is called a \textit{fibration} if $f_*\cO_X=\cO_Z$.
\end{itemize}

Let $f\colon X \to Z$ be a fibration between normal projective varieties.
\hfill
\begin{itemize}
\item Given a prime divisor $D \subseteq X$, we say $D$ is \textit{horizontal} if $f(D)$ dominates $Z$, we say it is \textit{vertical} otherwise.
\item Given a $\Q$-divisor $D$ on $X$, we can decompose it into its horizontal and vertical components, denoted by $D^h$ and $D^v$, respectively, so that $D= D^h +D^v$.
\item Given a curve $\xi \subseteq X$, we say $\xi$ is \textit{horizontal} if $f(\xi)$ is a curve, we say it is \textit{vertical} otherwise.
\item If $Z$ is a curve and $W \subset X$ is a subvariety, we say $W$ is \textit{horizontal} if $f(W)=Z$, we say it is \textit{vertical} otherwise.
\item If $\eta$ is the generic point of $Z$, we denote by $\widebar{\eta}$ its geometric generic point and by $X_{\eta}$ (resp.\ $X_{\widebar{\eta}}$) the generic (resp.\ geometric generic) fibre of $f$. Note that $X_{\eta}$ is generally defined over an imperfect field.
\item If $z \in Z$ is a point, we denote the fibre over it $X_z:=f^{-1}(z)$.
\item Let $D$ be a $\bQ$-divisor on $X$, and let $z,\eta\in Z$ denote a general point and the generic point, respectively. We denote by $D_{\eta}$ the base change of $D$ to the generic fibre and $D_{\widebar{\eta}}$ the base change to the geometric generic fibre. If $X_z$ is normal, $D$ is $\bQ$-Cartier along any codimension $1$ point of $X_z$, hence the restricted divisor $D_z:= D|_{X_z}$ is well-defined.
\item If $\delta$ is a prime divisor on $Z$ whose preimage under $f$ is of pure codimension $1$ in $X$, we denote by $f^{-1}(\delta)$ the induced divisor on $X$ with its reduced structure.
\end{itemize}

We recall that, when $X$ is a normal projective $\bk$-variety and $L$ is a Weil divisor, $H^0(X,L)$ is a finite dimensional $\bk$-vector space. 
\begin{itemize}
\item We denote by $|L|$ the associated projective space. We can naturally identify $|L|$ as the set of effective Weil divisors $L'$ which are linearly equivalent to $L$, and we refer to $|L|$ as the \textit{linear system associated to $L$}.
\item If $V$ is a subspace of $H^0(X,L)$, we denote by $|V|\subseteq |L|$ the natural associated projective subspace, which is called the \textit{linear subsystem associated to $V$}. 
\item This notation extends naturally to $\bQ$-divisors: given a $\bQ$-divisor $L$ we denote by $|L|_{\bQ}$ the set of all effective $\bQ$-divisors $L'$ such that $L'\sim_{\bQ}L$.
\item We say that a collection of subspaces $V_{\bullet}:=(V_m \subseteq H^0(X,mL))_{m \in \N}$ forms a \textbf{$\Q$-linear subsystem} if $V_m \cdot V_{m'} := \{ \sigma \tau \text{ s.t. } \sigma \in V_m \text{ and } \tau \in V_{m'} \}\subseteq V_{mm'}$.
\item We say that a linear subsystem $|V| \subseteq |L|$ is \textit{base point free} if, for every $x \in X$ there is $s \in V$ such that $s(x)\neq 0$. 
\item We define the \textit{base locus} of $|V|$ as the set of points $x \in X$ such that $s(x)=0$ for every $s \in V$ and we denote it by $\Bs(V)$. If $V$ is the complete linear system $H^0(X,L)$, we denote it simply by $\Bs(L)$.
\end{itemize}

\subsection{Singularities in birational geometry}

Singularities arise naturally in birational geometry because the class of smooth varieties is not stable under the birational operations of the Minimal Model Program.
The main classes of singularities studied in birational geometry are classified according to the discrepancy that appears under birational maps. For a more detailed discussion, see \cite[Chapter 2]{KM}.

\begin{definition}
A pair $(X,B)$ is the data of a normal variety $X$ and a $\Q$-divisor $B$ on $X$ such that $K_X+B$ is $\Q$-Cartier.
Let $(X, B)$ be a pair. Given a proper birational morphism from a normal variety $\mu\colon  Y \rightarrow X$, we denote by $\Exc(\mu)$ the union of all exceptional divisors of $\mu$.
We write:
\[
K_Y + \mu^{-1}_*B = \mu^*(K_X+B) + \sum_{i=1}^N a(E_i, X, B) E_i,
\]
where $\mu^{-1}_*B$ is the strict transform of $B$ and the $E_i$'s are all the prime exceptional divisors in $\Exc(\mu)$.
The quantity  $a(E, X, B)$ is called the \textbf{discrepancy} of $E$ with respect to $(X, B)$.
The divisor $E$ is called a \textbf{place} over $X$, while its image in $X$ is called the \textbf{centre} of $E$.
We say $(X, B)$ is:
\begin{itemize}
\item[•] \textbf{terminal} if $a(E, X, B)>0$ for all possible exceptional divisors $E$ over $X$;
\item[•] \textbf{canonical} if $a(E, X, B)\geq 0$ for all possible exceptional divisors $E$ over $X$;
\item[•] \textbf{Kawamata log terminal} or \textbf{klt} if $a(E, X, B)>-1$ for all possible exceptional divisors $E$ over $X$ and $\lfloor B \rfloor \leq 0$;
\item[•] \textbf{divisorial log terminal} or \textbf{dlt} if there exists a dense open subset $U \subseteq X$ such that $(U, B|_U)$ is log smooth, and $a(E, X,B) >-1$ for every $E$ whose centre is not contained in $U$;
\item[•] \textbf{log canonical} or \textbf{lc} if $a(E, X, B)\geq -1$ for all possible exceptional divisors $E$ over $X$.
\end{itemize}
\end{definition}

Sometimes it is useful to consider a weakening of the notion of log canonical singularities for varieties that are not necessarily normal (for a more detailed discussion see \cite[Chapter 5]{singularities}).

\begin{definition}
Let $X$ be a variety, we say it is \textbf{demi-normal} if it is $S_2$ and its codimension one points are either regular or nodal.
Note that we can define the canonical divisor $K_X$ also for demi-normal varieties.
Given an effective $\Q$-divisor $B$ on a demi-normal variety $X$ such that $K_X+B$ is $\Q$-Cartier, we say $(X,B)$ is \textbf{semi-log canonical} or \textbf{slc} if $(X^{\nu}, B^{\nu})$ is log canonical, where $\nu\colon X^{\nu} \rightarrow X$ is the normalisation morphism and $B^{\nu}$ is defined by log pullback, i.e.\ $K_{X^{\nu}}+B^{\nu}= \nu^*(K_X+B)$.
\end{definition}

\subsection{Resolutions of singularities}

The existence of log resolutions is proven in characteristic $0$ in all dimensions by Hironaka, whereas in positive characteristic it is still open in general. 

\begin{conjecture} \label{c-resolutions}
When we say that we assume \emph{the existence of log resolutions in dimension $n$}, we mean that, if $X$ is a quasi-projective variety of dimension $n$ defined over a perfect field $k$ and $B \geq 0$ is a $\Q$-divisor on it, then there exists a sequence of birational maps
    \[
    Y_N \to ...\to Y_1 \to Y_0=X
    \]
such that each $Y_j$ is the blow-up of $Y_{j-1}$ at a smooth centre, $Y_N$ is smooth and the union of the strict transform of $B$ in $Y_N$ and the exceptional divisors of $Y_N \to X$ is simple normal crossing. Moreover, the map $Y_N \to Y_0$ is an isomorphism outside the locus where $(X,B)$ is log smooth.
\end{conjecture}

\begin{theorem}[{\cite{Lipman, resolutionsAbhy, resolutionsCut, resolutions1, resolutions2}}] \label{r-resolutions}
        If $X$ is a variety of dimension at most $3$ over an algebraically closed field of characteristic $p>0$, then \autoref{c-resolutions} holds.
\end{theorem}

\subsection{Minimal Model Program in positive characteristic}

\begin{conjecture} \label{c-MMP}
When we say that we assume \emph{the LMMP in dimensions $n, m$}, we mean that if $(X, B)$ is a quasi-projective log canonical pair of dimension $n$ with $B \geq 0$, and $f \colon X \to Z$ is a projective morphism to a quasi-projective variety $Z$ of dimension $m$, then we can run the $K_X+B$ Minimal Model Program over $Z$.
More precisely, we assume that there exists a finite sequence of divisorial contractions and flips over $Z$ which terminates either with a Mori Fibre Space over $Z$, or with a minimal model, i.e.\ a birational model $(Y, C)$ over $Z$ such that $K_Y+C$ is nef over $Z$.
When we say that we assume \emph{the dlt LMMP in dimensions $n, m$}, we mean that we are assuming the LMMP in dimensions $n, m$, but only starting from pairs $(X,B)$ that are dlt.
When we say that we assume \emph{the birational dlt LMMP in dimension $n$}, we mean that we are assuming the dlt LMMP in dimensions $n, m=n$ and that all $f$-exceptional divisors are contained in $\Supp(B)$.
\end{conjecture}

\begin{theorem} \label{t-MMP}
    Let $k$ be a perfect field of characteristic $p>0$.
    \begin{itemize}
        \item[(i)] If $(X,B)$ is a two-dimensional log canonical pair over $k$, then the LMMP of dimensions $2, m$ holds, where $m \leq 2$ (\cite{MMPsurfaces}).
        \item[(ii)] If $p>5$ and $(X,B)$ is a three-dimensional log canonical pair over $k$, then the LMMP of dimensions $3, m$ holds, where $m \leq 3$ (\cite{HX, Bir, BW, MMPlc}).
        \item[(iii)] If $p>3$ and $(X,B)$ is a three-dimensional dlt pair over $k$, then the LMMP of dimensions $3, m$ holds, where $m \leq 3$ (\cite[Theorem 1.2]{low_char_MMP2}).
        \item[(iv)] For any $p>0$, if $(X,B)$ is a three-dimensional dlt pair, then the birational dlt LMMP holds (\cite[Theorem 1.1]{low_char_MMP}).
        \item[(v)] If $p>5$ and $(X,B)$ is a four-dimensional dlt pair with standard coefficients, then the birational dlt LMMP holds provided that log resolutions exist (\cite[Theorem 1.1]{mixedMMP}).
    \end{itemize}
\end{theorem}

\begin{conjecture} \label{c-inversion}
When we say that we assume \emph{inversion of adjunction in dimension $n$} we mean that, if $(X, B+S)$ is a pair, where $X$ has dimension $n$ and $S$ is a prime divisor, then $(X, B+S)$ is log canonical around $S$ if and only if the pair induced by adjunction $(S^{\nu}, B_{S^{\nu}})$ is log canonical.
\end{conjecture}

\begin{theorem} \cite[Corollary 10.1]{7authors} \label{r-inversion}
If $k$ is a perfect field of characteristic $p>5$ and $(X,B+S)$ is a three-dimensional log canonical pair over $k$, where $S$ is a prime divisor, then \autoref{c-inversion} holds.
\end{theorem}

\begin{remark} 
In \cite[Corollary 1.5]{low_char_MMP} a weaker version of inversion of adjunction (plt) is proven in dimension three also in low characteristics, and in \cite[Corollary 4.8]{mixedMMP} it is proven in dimension four for pairs with standard coefficients and assuming the existence of log resolutions.
However, this is not enough for our purposes. 
The only place where we explicitly use inversion of adjunction is \autoref{p-well-defined}, to prove that the discriminant and the moduli parts are well-defined, there inversion of adjunction is applied to all general fibres.
In fact, even if we start from a pair $(X, B)$ and a fibration $f\colon X \to Z$ such that the geometric generic fibre $X_{\wb{\eta}}$ is normal and the pair induced on it $(X_{\wb{\eta}},B_{\wb{\eta}})$ is klt, in the proof, we need to modify it so that the geometric generic fibre becomes dlt.
Already in this latter case the plt inversion of adjunction is not enough to conclude.
\end{remark}

\section{Singularities of fibrations}

In characteristic $0$, given a fibration, it is automatic that the generic fibre is geometrically reduced.
In positive characteristic, this is no longer true and it is equivalent to asking that the fibration is \emph{separable}.

\begin{definition}
Let $K \subseteq L$ be a field extension. It is called \textbf{separable} if there exists a transcendence basis $t_1, ..., t_{\ell}$ such that $L$ is a finite separable extension of $K(t_1,...,t_{\ell})$.\\
Let $f\colon X \rightarrow Z$ be a morphism between integral varieties.
We say that $f$ is \textbf{separable} if the field extension $k(Z) \subseteq k(X)$ is separable; otherwise $f$ is called \textbf{inseparable}.
\end{definition}

The next lemma is well-known, we state it for completeness.

\begin{lemma} \label{l-split extensions} \label{d-degree}
    \begin{itemize}
        \item[(i)] Let $K \subseteq L$ be a field extension, then there is an intermediate field $K \subseteq L' \subseteq L$ such that $K \subseteq L'$ is a separable extension and $L' \subseteq L$ is a purely inseparable finite extension.
        \item[(ii)] Let $\varphi \colon Y \to X$ be a finite morphism of curves. Then there exists $\sigma\colon Y \to X$ separable and $e \geq 0$ such that $\varphi=F^e \circ \sigma = \sigma \circ F^e$. In this case $p^e$ is said to be the \textbf{purely inseparable degree} of $\varphi$.
        \item[(iii)] Let $X$ be a normal projective variety with function field $k(X)$. Let $k(X) \subseteq L$ be a finite extension. Then, up to considering a further purely inseparable extension $L \subseteq L'$, there exist a normal projective variety $Y$, a generically finite separable map $\sigma \colon Y \to X$ and $e \geq 0$ such that the extension induced by $\varphi := \sigma \circ F^e = F^e \circ \sigma$ is exactly $k(X) \subseteq L'$.
    \end{itemize}
\end{lemma}

\begin{proof}
Point (i) follows from \cite[Tag 030K]{stacks}.
Point (ii) follows from \cite[Tag 0CD2]{stacks} and \cite[Tag 0CC7]{stacks}.
As for point (iii): by (i) we can decompose the extension $k(X) \subseteq L$ as a finite separable extension $k(X) \subseteq K'$ followed by a purely inseparable extension $K' \subseteq L$. Up to considering a further purely inseparable extension $L \subseteq L'$, we can assume $L'=K'^{\frac{1}{p^e}}$ for some $e \geq 0$.
The field extension $k(X) \subseteq K'$ induces a separable finite map $U \rightarrow X$. 
By taking a projective completion of $U$, resolving the resulting rational map and taking the normalisation, we obtain a generically finite separable map $\sigma\colon Y \to X$ from a normal projective variety $Y$. The composition $\varphi:= \sigma \circ F^e$ corresponds to the field extension $k(X) \subseteq L'$.
Note that, by \cite[Tag 0CC7]{stacks}, $\varphi= \sigma \circ F^e = F^e \circ \sigma$.
\end{proof}

\begin{example} \label{e-inseparable}
    In general, the condition $f_*\cO_{X} = \cO_{Z}$ is not enough to ensure separability.
    Consider, for example, the threefold $X = V(sx^p+ty^p+z^p) \subset \Pn^{2}_{[x:y:z]} \times \A^{2}_{(s,t)}$.
    Let $f$ be the fibration induced by the natural projection onto $Z:= \A^{2}_{(s,t)}$.
    Then $f$ satisfies $f_*\cO_{X} = \cO_{Z}$, but it is not separable.
\end{example}

\begin{proposition}[{\cite[Proposition 2.15, Chapter 3]{Liu}}] \label{p-reducedness}
Let $K$ be a field. A variety $X$ over $\Spec(K)$ is geometrically reduced if and only if ${f\colon X \to \Spec(K)}$ is separable.
\end{proposition}

\begin{remark}
In particular, a fibration $f\colon X \to Z$ is separable if and only if the geometric generic fibre $X_{\Bar{\eta}}$ is reduced.
\end{remark}

When $Z$ is a curve, a theorem of MacLane \cite{MacLane} allows us to compare the notion of separability of a surjective morphism with its Stein factorisation.
In this case it is therefore easier to check this condition.
We write here a version of it restated in geometric terms.

\begin{theorem}[{\cite[Corollary 2.5]{non-reduced}}] \label{t-curve separable fibration}
Let $f\colon X \rightarrow Z$ be a fibration onto a curve $Z$.
Then $f$ is separable.
\end{theorem}

\begin{remark} \label{r-generic smoothness fails}
    Even if a fibration is separable, its fibres may be highly singular.
    For example, in characteristics $2$ and $3$, there exist smooth surfaces equipped with a fibration whose general fibre is a cuspidal curve (quasi-elliptic fibrations).
    Separable {Calabi--Yau} fibrations of relative dimension $1$ have normal general fibre for $p>3$ (\cite[Corollary 1.8]{LMMP}), whereas in higher dimension it is not even known if the same holds for $p \gg 0$.
\end{remark}

\subsection{Log canonical singularities in families} \label{s-MMP_sing_families}

Over fields of positive characteristic, heuristically, the generic fibre of a fibration $f\colon X \rightarrow Z$ reflects properties of $X$, while the geometric generic fibre is strictly related to the general fibres of $f$.
In this section, we make this philosophy more precise, in particular studying the property of being log canonical.

\begin{definition} \label{d-GGLC}
We denote by $(X/Z, B)$ the data of a a pair $(X, B)$ and a fibration between normal varieties $f\colon X \rightarrow Z$.

We say $(X/Z, B)$ is \textbf{generically log canonical} or \textbf{GLC} if the pair $(X_{\eta}, B_{\eta})$ is log canonical, where $\eta$ is the generic point of $Z$. 
We say $(X/Z, B)$ is \textbf{geometrically generically log canonical} or \textbf{GGLC} if $X_{\wb{\eta}}$ is normal and the pair $(X_{\widebar{\eta}}, B_{\widebar{\eta}})$ is log canonical.
\end{definition}

\begin{remark}
In particular, given a fibration $f\colon X \rightarrow Z$ between normal varieties, if there exists a $\Q$-divisor $B$ on $X$ such that $(X/Z, B)$ is GGLC, then $f$ is separable.
\end{remark}

\begin{remark}
    Over fields of characteristic $0$, the notions of GLC and GGLC coincide.
\end{remark}

\begin{definition}
    Let $K$ be a field and $\widebar{K}$ its algebraic closure. Let $K \subseteq L \subseteq \widebar{K}$ be a finite field extension of $K$.
    Let $X$ be a variety over $\widebar{K}$. We say that $X$ is \textbf{defined over $L$} if there exists a variety $X_L$ over $L$ such that $X_L \times_{L} \widebar{K} = X$.
\end{definition}

\begin{lemma} \label{l-GGLC normal fibre}
Let $f\colon X \rightarrow Z$ be a fibration between normal varieties over $k$.
Assume that $B$ is a $\Q$-divisor on $X$ such that $B^h \geq 0$ and $K_X+B$ is $\Q$-Cartier.
Suppose that the geometric generic fibre $X_{\widebar{\eta}}$ is integral. Let ${B_{\widebar{\eta}}}^{\nu}$ be the boundary divisor on ${X_{\widebar{\eta}}}^{\nu}$ defined by restriction.
Assume that $({X_{\widebar{\eta}}}^{\nu}, {B_{\widebar{\eta}}}^{\nu})$ is log canonical, and that one of the following conditions holds:
\begin{itemize}
    \item[(a)] $\mathrm{char}(k)=p>2$, or
    \item[(b)] $\mathrm{char}(k)=2$ and ${X_{\widebar{\eta}}}^{\nu} \to X_{\wb{\eta}}$ is separable.
\end{itemize}
Then $X_{\widebar{\eta}}$ is normal. In particular, $(X/Z, B)$ is GGLC.
\end{lemma}

\begin{proof}
The pair $({X_{\widebar{\eta}}}, {B_{\widebar{\eta}}})$ is slc.
In fact, the $S_2$ property is invariant by flat base change. 
Moreover, if $X_{\widebar{\eta}}$ had singularities worse than nodal in codimension $1$, ${B_{\widebar{\eta}}}^{\nu}$ would have coefficients strictly bigger than $1$ coming from the conductor over those singularities, contradicting the log canonical assumption.

Furthermore, the normalisation of $X_{\wb{\eta}}$ is a universal homeomorphism by \cite[Lemma 2.2]{pi_basechange}, therefore, if we are in cases (a) or (b), by \cite[Proposition 3.1.12]{nodes} nodal singularities cannot appear, whence $X_{\wb{\eta}}$ must be already normal.
\end{proof}


\begin{example}
In characteristic $2$, there exist nodal surface singularities whose normalisation is purely inseparable. An example is $X= V(zx^2+y^2) \subseteq \A^3_{(x,y,z)}$ (see \cite[Example 3.1.9]{nodes}).
\end{example}

\begin{proposition}[{\cite[Proposition 2.1, Lemma 2.2]{LMMP}}] \label{p-geometric vs general}
Let $f\colon X \rightarrow Z$ be a morphism of varieties.
Then the geometric generic fibre is normal (resp.\ regular, reduced) if and only if a general fibre is normal (resp.\ regular, reduced).
Let $Y \rightarrow X$ be the normalisation of $X$. 
If for a general point $z \in Z$, $Y_z$ is normal, then $Y_z$ is the normalisation of $X_z$.
\end{proposition}

\begin{lemma} \label{l-conductor is vertical}
Let $f\colon X \rightarrow Z$ be a separable fibration between normal varieties and let $\varphi\colon Z' \rightarrow Z$ be a generically finite map.
Let $Y'$ be the normalisation of the main component of the fibre product $X':= X \times_Z Z'$.
If the geometric generic fibre $X_{\widebar{\eta}}$ is normal, the conductor of $Y' \rightarrow X'$ is vertical.
In particular, $Y'_{\widebar{\eta}}= X_{\widebar{\eta}}$.
\end{lemma}

\begin{proof}
Note that $X_{\widebar{\eta}}=X'_{\widebar{\eta}}$, therefore $X'_{\eta}$ is geometrically normal. This implies that there exists an open dense subset $U \subseteq Z'$ such that $V:=f'^{-1}(U)$ is normal, whence $Y' \to X'$ is an isomorphism over $V$.
\end{proof}

\begin{lemma} \label{l-spread-out}
Let $f\colon X \rightarrow Z$ be a separable fibration between normal varieties.
Assume that the geometric generic fibre $X_{\widebar{\eta}}$ is normal and let $\widebar{\sigma} \colon \widebar{Y} \rightarrow X_{\widebar{\eta}}$ be a proper birational morphism between normal varieties.
Then, there exist a generically finite map $\varphi\colon Z' \rightarrow Z$ and a proper birational morphism $\sigma\colon Y \rightarrow Y'$, where:
\begin{itemize}
    \item[(i)] $Y'$ is the normalisation of the main component of $X \times_Z Z'$;
    \item[(ii)] $Y_{\widebar{\eta}}= \widebar{Y}$. 
\end{itemize}
If $X$ and $Z$ are projective, we can choose $Z'$ and $Y$ projective.
\end{lemma}

\begin{proof}
There exists $L$, finite extension of $k(Z)$ such that $\widebar{Y}$ and $\widebar{\sigma}$ are defined over $L$.
By ``spreading out techniques'' (see \cite[Proof of Corollary 1.10]{DW} and \cite[Lemma 2.25]{invariance}), there exist $U \subseteq Z$ dense open subset and a finite map $\varphi\colon U' \rightarrow U$ such that, if $W'$ is the normalisation of the main component of $f^{-1}(U) \times_U U'$, there is a proper birational map $s \colon W \rightarrow W'$ with $W_{\widebar{\eta}}= \widebar{Y}$.
If $Z$ is projective, let $\widebar{U}'$ be a projective closure of $U'$ and define $\varphi\colon Z' \to Z$ generically finite, as a resolution of the indeterminacies of $\widebar{U}' \dashrightarrow Z$.
Let $Y'$ be the normalisation of the main component of $X \times_Z Z'$ and $Y''$ a projective closure of $W$; there is an induced rational map $\tau \colon Y'' \dashrightarrow Y'$, which is well-defined over $U'$. We take $\sigma\colon Y \rightarrow Y'$ to be a resolution of indeterminacies of $\tau$ that is an isomorphism over $U'$. By possibly substituting $Y$ with its normalisation, we can also assume it is normal.
\end{proof}

\begin{remark} \label{c-pairs over the ggfibre}
    Let $f\colon X \rightarrow Z$ be a separable fibration between normal varieties.
    Assume that the geometric generic fibre $X_{\widebar{\eta}}$ is normal.
    Let $B \geq 0$ be a $\Q$-divisor on $X$ such that $K_X+B$ is $\Q$-Cartier and let $B_{\wb{\eta}}$ be the induced $\Q$-divisor on $X_{\wb{\eta}}$.
    Let $\wb{Y} \to X_{\wb{\eta}}$ be a birational map, together with a $\Q$-divisor $\wb{C} \geq 0$.
    By \autoref{l-spread-out} and \autoref{l-conductor is vertical}, up to shrinking $Z$, we can construct $Z' \to Z$ and $Y \to X'$ such that $\wb{Y}=Y_{\wb{\eta}'}$. Moreover, up to choosing a further finite cover of $Z$, we can assume that $\wb{C}$ is also defined over $k(Z')$. Therefore, there exists a $\Q$-divisor $C$ on $Y$ such that $(Y_{\wb{\eta}'}, C_{\wb{\eta}'})= (\wb{Y}, \wb{C})$, where $\eta'$ is the generic point of $Z'$.
\end{remark}

\begin{proposition} \label{p-general vs geometric generic lc}
Assume the existence of log resolutions of singularities in dimension $d$.
Let $f\colon X \rightarrow Z$ be a fibration between normal varieties such that $\dim(X) - \dim(Z)=d$.
Let $B \geq 0$ be a $\Q$-divisor such that $K_X+B$ is $\Q$-Cartier. The pair $(X/Z, B)$ is GGLC if and only if:
\begin{itemize}
\item[(i)] a general fibre $X_z$ of $f$ is normal, and
\item[(ii)] the pair $({X_z}, B_{z})$ is log canonical. 
\end{itemize}
\end{proposition}

\begin{proof}
Assume $(X/Z, B)$ is GGLC. Since $X_{\wb{\eta}}$ is normal, by \autoref{p-geometric vs general} this implies that $X_z$ is normal for a general $z \in Z$.
Let $(\wb{Y}, \wb{C})$ be a log resolution of $(X_{\wb{\eta}}, B_{\wb{\eta}})$. By \autoref{c-pairs over the ggfibre}, up to shrinking $Z$, there exists a finite morphism $Z' \to Z$ and a birational map $\sigma\colon Y \to X':= X \times_Z Z'$ such that $Y_{\wb{\eta}'}=\wb{Y}$, where $\eta'$ is the generic point of $Z'$. Moreover, there exists a $\Q$-divisor $C$ on $Y$ such that $C_{\wb{\eta}'}=\wb{C}$.
Define a $\Q$-divisor $B'$ on $X'$ by $K_{X'}+B'= \Phi^*(K_X+B)$. 
Note that the map $\Phi\colon X' \to X$ is an isomorphism when restricted to $X'_{\wb{\eta}'}$ and to a general fibre $X'_{z'}$, so, by \autoref{p-geometric vs general}, $(Y_{z'}, C_{z'})$ is a log resolution of $(X_{z}, B_z)$.
Let $E$ be a horizontal exceptional divisor of $\sigma$.
By construction, $E_{\wb{\eta}}$ is reduced, therefore, by \autoref{p-geometric vs general}, $E_{z'}$ is reduced for $z' \in Z'$ general.
Then, by adjunction, we see that the discrepancies of $E_{z'}$ and $E_{\widebar{\eta}}$ coincide.

As for the converse, by \autoref{p-geometric vs general}, since $X_z$ is normal for a general $z \in Z$, $X_{\wb{\eta}}$ is also normal.
Therefore, we can proceed as above to construct a finite map $Z' \to Z$, a birational morphism $Y \to X':= X \times_Z Z'$ and a $\Q$-divisor $C$ on $Y$ such that $(Y_{\wb{\eta}'}, C_{\wb{\eta}'})$ and $(Y_{z'}, C_{z'})$ are log resolutions of $(X_{\wb{\eta}}, B_{\wb{\eta}})$ and $(X_z, B_z)$, respectively, where $\eta'$ is the generic point of $Z'$, $z' \in Z'$ is a general point and $z \in Z$ is its image.
We then conclude in the same way, by comparing the discrepancies.
\end{proof}

\begin{lemma} \label{l-normal_afterMMP}
    Let $f\colon X \to Z$ be a fibration of normal projective varieties and $B\geq 0$ a $\Q$-divisor on $X$ such that $(X,B)$ is log canonical. Let $\wb{\eta}$ be the geometric generic point of $Z$ and assume that $X_{\wb{\eta}}$ is normal.
    Let $\varphi \colon X \dashrightarrow Y$ be a step of the $(K_X+B)$-MMP over $Z$.
    Then, $Y_{\wb{\eta}}$ is normal.
\end{lemma}
    
\begin{proof}
    Note that $Y_{\wb{\eta}}$ satisfies the $S_2$ property since $Y$ is normal.
    Therefore, we only need to check that $Y_{\wb{\eta}}$ is regular in codimension $1$.
    The locus where $\varphi$ is not an isomorphism has codimension $\geq 2$ in $Y$, and therefore in $Y_{\wb{\eta}}$. In particular, $\varphi$ is an isomorphism around points of codimension $1$ of $Y_{\wb{\eta}}$.
\end{proof}

\section{Foliations and purely inseparable morphisms} \label{s-foliations}

\subsection{Foliations} \label{s-f1}

Every separable fibration $f\colon X \to Z$ defines a \emph{foliation} $\cF$ whose general leaves consist of the fibres of $f$ and of some subvarieties on which $f$ is inseparable.
Here, we give a formula for the canonical divisor of $\cF$ assuming that the general fibres are normal.

\begin{definition}
Let $X$ be a normal variety.
A subsheaf $\cF \subseteq T_X$ is said to be \textbf{saturated} if the quotient $T_X/\cF$ is torsion free.
In characteristic $p>0$, if $\partial$ is a derivation, then the composition of $p$ copies of $\partial$ as a function is also a derivation and it is called the \textbf{$p$-power} of $\partial$.
A \textbf{foliation} on $X$ is a subsheaf of the tangent sheaf, $\mathcal{F} \subseteq T_X$, which is saturated, closed under Lie brackets and under $p$-powers.\\
Let $\omega_{\mathcal{F}}:= \det( \mathcal{F})^*$ be the dual of the reflexified top exterior power of $\mathcal{F}$.
The canonical divisor of a foliation $\cF$ is any Weil divisor $K_{\cF}$ such that $\cO_{X}(K_{\cF}) = \omega_{\cF}$.
\end{definition}


Separable fibrations naturally induce foliations by considering the relative tangent bundle.

\begin{definition} 
    Let $f\colon X \rightarrow Z$ be a separable fibration between normal varieties. It induces a foliation $\cF$ as the saturation of the kernel of $df \colon T_X \to f^*T_Z$.
\end{definition}

\begin{lemma} \label{l-saturated} 
    Let $f\colon X \to Z$ be a separable flat fibration between normal varieties. Then, the kernel of ${df\colon T_X \to f^*T_Z}$ is saturated in $T_X$.
\end{lemma}

\begin{proof}
    Note that $T_X/ \ker(df) \subseteq f^*T_Z$. Since $Z$ is normal, $T_Z$ is torsion free and since $f$ is flat, by \cite[Tag 0AXV]{stacks} $f^*T_Z$ is torsion free. Therefore, $T_X/ \ker(df)$ is torsion free as well.
\end{proof}

\begin{definition} \label{d-tame}
    Let $f\colon X \rightarrow Z$ be a separable equidimensional fibration between normal varieties.
    Let $\delta \subset Z$ be a prime divisor. If $f^*\delta = \sum \ell_D D$, with the sum running over the prime divisors $D$ over $\delta$, then $\ell_D$ is called the \textbf{multiplicity of $D$ with respect to $f$}.
    We define the \textbf{ramification divisor of $f$} to be
    \[
    R(f) := \sum (f^*\delta -f^{-1}(\delta)) = \sum (\ell_D -1)D,
    \]
where the first sum is taken over all prime divisors $\delta$ of $Z$, while the second sum is taken over all vertical prime divisors $D$ on $X$ and $\ell_D$ is their multiplicity with respect to $f$.
If $p$ divides $\ell_D$, we call $D$ a \textbf{wild component}.
If $f\colon X \to Z$ is an equidimensional separable fibration without wild components, we call it a \textbf{tame fibration}.
\end{definition}

\begin{remark}
    If $f\colon X \to Z$ is an equidimensional morphism between normal varieties and $\delta \subseteq Z$ is a Weil divisor, we can define the pull-back $f^*\delta$ even if $\delta$ is not $\Q$-Cartier.
    Indeed, let $U \subseteq Z$ be a big open subset that is regular, then $f^*\delta$ is defined as the closure of $f|_U^*(\delta|_U)$.
\end{remark}

\begin{theorem} \label{t-canonical foliation}
    Let $f\colon X \to Z$ be an equidimensional separable fibration between normal varieties and let $\cF$ be the 
    foliation induced by $f$. Assume that the geometric generic fibre $X_{\widebar{\eta}}$ is normal.
    Then
    \[
K_{\cF}= K_X -f^*K_Z - R(f) - W(f),
    \]
    where $R(f)$ is the ramification divisor and $W(f) \geq 0$ is supported on the wild components.
    More precisely, for every wild component $D$, there exists an integer $a_D \geq 0$ such that $$W(f)= \sum_{D \text{ wild}} (a_D+1)D.$$
\end{theorem}

\begin{proof}
\noindent \textbf{Step 1.}
Since $f$ is equidimensional and $X$ and $Z$ are normal, by \cite[Lemma 2.19]{Lena-Joe} (as a consequence of the flattening lemma \cite[\S 3.3]{flattening}, \cite[Théorème 5.2.2]{RG}), by restricting to big open subsets, we can assume $f$ is a flat morphism between smooth varieties. Therefore, by \autoref{l-saturated}, we can assume $\cF = \ker(df)$.

\noindent \textbf{Step 2.}
    We claim that there exists a dense open subset $U \subseteq X$ such that $U_{\eta}$ is a big open subset of $X_{\eta}$ and the sequence
    \[
    0 \to \cF|_{U} \to T_X|_{U} \to f^*T_Z|_{U} \to 0
    \]
    is exact.
    The sequence $0 \to \cF \to T_X \to f^*T_Z$ is always exact, therefore we only need to show surjectivity at the generic fibre.
    Since $X_{\widebar{\eta}}$ is normal, its singular locus $\Sigma$ has codimension $\geq 2$.
    Let $\Sigma'\subseteq X_{\eta}$ be a subvariety of codimension $\geq 2$ such that $\Sigma'_{\wb{\eta},\mathrm{red}} \supseteq \Sigma$. Let $T$ be the Zariski closure of $\Sigma'$ in $X$, $U:= X \setminus \supp(T)$ and $V:=f(U)$.
    By \cite[Tag 01V8]{stacks}, up to possibly restricting $V$ further (and restricting $U$ accordingly), $f|_U \colon U \to V$ is a flat smooth fibration.
    By \cite[Tag 02G1]{stacks}, the sheaf $T_{U/V}$ is locally free and by \cite[Tag 02K4]{stacks}, the sequence
    \[
        0 \to T_{U/V} \to T_{U} \to f^*T_V \to 0
    \]
    is exact, whence the claim.

\noindent \textbf{Step 3.}
By Step 2, the difference between $K_{\cF}$ and $K_X-f^*K_Z$ is supported on vertical divisors.
Therefore, to conclude, we compute explicitly the image $I$ of $df \colon T_X \to f^*T_Z$ around codimension $1$ points corresponding to prime vertical divisors.
Note that $I$ is locally free around points of codimension $1$.

Let $D$ be a prime vertical divisor. 
By inductively cutting with general hyperplanes of $Z$ which intersect $f(D)$ transversally, we can assume $f\colon X \to Z$ is a fibration onto a curve.
Around a general point of $D$, we can further assume that the fibres of $f$ (with their reduced structure) are smooth. Therefore, we can choose local étale coordinates $x, y_1, ..., y_d$ around $D$ such that $D = V(x)$ and $f$ is the map $(x, y_1, ..., y_d) \mapsto x^{\ell_D}u$, where $u \in k[x,y_1,...,y_d]$ is a unit around $D$ and $\ell_D$ is the multiplicity of $D$ with respect to $f$.
    Let $\partial_x, \partial_{y_1},..., \partial_{y_d}$ be the derivations in $T_X$ corresponding to the coordinate directions. Let $t:= x^{\ell_D}u$ be a local coordinate of $Z$ and $\partial_t$ the corresponding derivation which generates $T_Z$.

    \underline{Tame case:} assume that $p$ does not divide $\ell_D$.
    Then,
    \[
        \begin{cases}
            df(\partial_x) = x^{\ell_D-1}(v_x \partial_t);\\
            df(\partial_{y_i}) = x^{\ell_D}(v_{y_i} \partial_t), 
        \end{cases}
    \]
    where $v_x$ is a unit around $D$ and $v_{y_i}$ is a function.
    In particular, in this étale neighbourhood, $I= x^{\ell_D-1}f^*T_Z$.
    Therefore, at the generic point of $D$, the sequence
    \[
        0 \to \cF \to T_X \to x^{\ell_D-1}f^*T_Z \to 0
    \]
    is exact. By taking determinants, we get that $K_{\cF}= K_X -f^*K_Z-(\ell_D-1)D$ around $D$.

    \underline{Wild case:} assume that $p$ divides $\ell_D$. Then, for $i=0,...,d$, we define $a_i$ and $v_i$ so that
    \[
        \begin{cases}
            df(\partial_x) = x^{\ell_D}(x^{a_0}v_0 \partial_t) \quad \text{if } df(\partial_x) \text{ is not identically } 0;\\
            df(\partial_{y_i}) = x^{\ell_D}(x^{a_i}v_i \partial_t) \quad \text{if } df(\partial_x) \text{ is not identically } 0.
        \end{cases}
    \]
    Therefore, for all $i \in \{ 0,...,d \}$ such that the expression is not identically $0$, $a_i \geq 0$. Moreover, since $f$ is separable, there exists an index $i$ such that $0 \neq v_i$ is a unit around $D$.
    In particular, in this étale neighbourhood, $I = x^{\ell_D+a_D}f^*T_Z$, for some $a_D \geq 0$, and the sequence
    \[
        0 \to \cF \to T_X \to x^{\ell_D+a_D}f^*T_Z \to 0
    \]
    is exact. By taking determinants, we obtain $K_{\cF}= K_X -f^*K_Z-(\ell_D+a_D)D$ around $D$.
\end{proof}

\begin{remark}
The formula in \autoref{t-canonical foliation} is well-known over fields of characteristic $0$, where all fibrations are tame. See for example \cite[\S 2.9]{Druel}.
\end{remark}

\begin{example}
Let $g \colon \A^2 \to \A^1$ be the fibration defined by $(x, y) \mapsto t:=xy$.
Let $X$ be the blow-up at the origin of $\A^2$ and call $E$ the exceptional divisor. Let $f \colon X \to \A^1$ be the induced fibration. 
Then, $f$ is tame if and only if the characteristic of the base field is $\neq 2$, in which case $K_{\cF}= K_X -f^*K_{\A^1}-E$. If the characteristic is $2$, the wild component is exactly $E$ and $K_{\cF}=K_X -f^*K_{\A^1}-2E$.
\end{example}

\begin{example}
    The assumption on the normality of the geometric generic fibre in \autoref{t-canonical foliation} is essential.
    Indeed, let $k$ be a field of characteristic $3$ and $S := V(y^2-x^3-t) \subseteq \A^3_{(x,y,t)}$. Let $f\colon S \to C:= \A^1_t$ be the projection onto the third coordinate.
    The fibration $f$ is a quasi-elliptic fibration, the geometric generic fibre is reduced, but it has a cusp.
    Let $\cF:= \ker(df)$ and $D:= V(y) \subseteq S$.
    We compute 
    \[
        \begin{cases}
            df(\partial_x) = 0;\\
            df(\partial_y) = 2y \partial_t.
        \end{cases}
    \]
    Therefore, we obtain the short exact sequence
    ${0 \rightarrow \cF \rightarrow T_S \rightarrow 2yf^*T_C \rightarrow 0}$, whence
    \[
        K_{\cF} = K_S -f^*K_C -D.
    \]
\end{example}

\subsection{Frobenius base change} \label{s-purelyinsep_basechange} \label{s-f2}

Note that the differential of the Frobenius morphism is $0$. This easy observation plays a key role in the study of fibrations.
In fact, to each finite purely inseparable morphism corresponds a foliation and vice-versa.
In particular, given a fibration $f$, its induced foliation corresponds to the foliation induced by any power of the relative Frobenius.
Using this correspondence, we provide formulas to relate the canonical divisors obtained from a Frobenius base change of a fibration.
Later, we perform this base change to overcome two issues: the fact that in positive characteristic we do not have a Cone theorem for foliations (\autoref{s-B&B}) and to deal with fibrations with normal, but non-log canonical fibres (\autoref{s-geometric star}).

\begin{definition}
    Let $X$ be a variety over a field of characteristic $p>0$. We will denote by $F\colon X^1\to X$ the \textbf{absolute Frobenius morphism of $X$} or simply \textbf{Frobenius morphism}. It is the identity on the underlying topological space, and it acts on regular functions by raising them to the $p^{\text{th}}$ power. For all $e\geq 1$ we denote by $F^e\colon X^e\to X$ the $e^{\text{th}}$ power of the absolute Frobenius.
\end{definition}

\begin{definition}
    Let $X$ and $X'$ be schemes over a field of characteristic $p>0$.
    A purely inseparable morphism $a \colon X' \rightarrow X$ is called of \textbf{height one} if there exists $\alpha \colon X \rightarrow X'$ such that $a\circ \alpha = F$.
\end{definition}

\begin{proposition}[{\cite[Proposition 2.9]{LMMP}}] \label{p-correspondence_foliations}
    Let $X$ be a normal variety.
    There is a $1$-to-$1$ correspondence
    
    \begin{center}
        \begin{tikzcd}
            \left\{ \begin{array}{c}\text{Height one morphisms}\\ \text{ $X\to X'$ with $X'$ normal}\end{array} \right\} \arrow[rr] & & \left\{ \begin{array}{c}\text{Foliations $\cF\subseteq T_{X}$}\arrow[ll]\\
            \end{array} \right\}
        \end{tikzcd}
    \end{center}
    given by:
    \begin{itemize}
    \item[($\leftarrow$)] $X':= \Spec_{X}(\cO_{X}^{\cF})$, where $\cO_{X}^{\cF}\subseteq \cO_{X}$ is the subsheaf of $\cO_{X}$ that is taken to zero by all the sections of $\cF$;
    \item[($\rightarrow$)] $\cF:=\lbrace \partial\in T_{X} \text{ s.t. } \partial\cO_{X'}=0\rbrace$.
    \end{itemize}
    Moreover, morphisms of degree $p^r$ correspond to foliations of rank $r$.
\end{proposition} 

\begin{samepage}
\leqnomode
\begin{flalign} \label{construction}
    \tag{\raisebox{-0.5ex}{\Snowflake}} \hspace{7mm} \text{\underline{Construction}} &&
\end{flalign}
\reqnomode

Given a fibration $f\colon X \rightarrow Z$ between normal varieties, we consider the following diagram
\end{samepage}

\begin{equation*}
\begin{tikzcd}
X^e \arrow[rr, bend left=35, "F^{e}"] \arrow[r,"\alpha_e"] \arrow[dr, "f", swap] & X^{(e)} \arrow[d,"f_e"] \arrow[r, "\beta_e"] & X \arrow[d, "f"] \\
 & Z^e \arrow[r, "F^e"] & Z,
\end{tikzcd}
\end{equation*}
where:
\begin{itemize}
\item we define $X_{Z^e}$ as the fibre product $X \times_{Z^e} Z$;
\item $X^{(e)}$ is the normalisation of the reduction of $X_{Z^e}$ and $f_e$ is the induced fibration;
\item $\alpha_e$ and $\beta_e$ are the induced maps, so that $\beta_e \circ \alpha_e = F^e$. When $X_{Z^e}$ is reduced, $\alpha_e$ generically coincides with the $e$\textsuperscript{th}-power of the relative Frobenius of $X$ over $Z$;
\item if $D \subseteq X$ is a prime divisor, denote by $D^{(e)} \subseteq X^{(e)}$ its reduced image in $X^{(e)}$;
\item if $D:= \sum_i a_i D_i$ is a $\Q$-divisor on $X$, where the $D_i$'s are prime divisors, by $\beta_e^{-1}D$ we denote the $\Q$-divisor $\sum_i a_i D_i^{(e)}$; 
\item if $W \subseteq X$ is a subvariety, we denote by $W^{(e)}$ its reduced image in $X^{(e)}$;
\item when $f$ is separable, $\cF$ denotes the foliation induced by $f$ and $\cF_e$ the foliation induced by $f_e$, unless otherwise stated;
\item in the sequel, we will often omit the superscript $e$ on the source of $F^e$ if it is clear from the context.
\end{itemize}

Next, we study some properties of the maps and the varieties involved in the diagram \autoref{construction}.
In particular, we determine the relations between the canonical divisors $K_X$, $K_{X^{(e)}}$ and the canonical divisors of the foliations induced by $f$ and $f_e$.

\begin{lemma} \label{l-Yk integral}
If $f\colon X \rightarrow Z$ is a flat separable fibration between normal varieties, $X_{Z^e}$ is integral and $X^{(e)}$ is its normalisation.
\end{lemma}

\begin{proof}
By \cite[Remark 2.5]{Jakub}, if in $X_{Z^e}$ there are some non-reduced components, they must dominate $Z$ since $f$ is flat.
Thus, we can check reducedness at $\eta$, the generic point of $Z$.
By \autoref{p-reducedness}, $X_{\eta}$ is geometrically reduced. 
Moreover, since purely inseparable morphisms are homeomorphisms, $X_{Z^e}$ is irreducible as $X$.
All in all, $X_{Z^e}$ is integral.
\end{proof}

\begin{lemma} \label{l-lemma!} 
    Let $f\colon X \rightarrow Z$ be a separable fibration between normal varieties such that the geometric generic fibre $X_{\widebar{\eta}}$ is normal. Let $\xi$ be a horizontal curve in $X$ and $p^{e_0}$ the purely inseparable degree of $f|_{\xi}$.
    Let $e \in \N$ and assume that $\alpha_e(\xi)$ is not contained in the conductor of $X^{(e)} \to X_{Z^e}$ (note that this is automatically satisfied if $Z$ is a curve).
    Then, $\deg(\alpha_e|_{\xi})= \min\{p^e, p^{e_0}\}$.
\end{lemma}
    
\begin{proof}
    Let $np^{e_0}$ be the degree of $f|_{\xi}$, where $n \in \N$ is coprime with $p$. In particular, $f|_{\xi}$ factors through $F^{e_0}$, but not through $F^{e_0+1}$.
    Let $\xi_e:= \alpha_e(\xi)$ and $\zeta:=f(\xi)$.
    By assumption, around the generic point of $\xi_e$ we can assume that $X^{(e)}\to X_{Z^e}$ is an isomorphism. Note that, when $Z$ is a curve, the normalisation of $X_{Z^e}$ is an isomorphism everywhere but on a vertical subset by \autoref{l-conductor is vertical}, therefore this condition is automatically satisfied.
    By the universal property of the fibre product, the purely inseparable part of $f_e|_{\xi_e}$ has degree $p^{e_0-e}$ if $e < e_0$, while $f_e|_{\xi_e}$ is separable otherwise.
    Consider the diagram:
    \[
    \xymatrix{
    \xi^{\nu} \ar[r]^{\alpha_e|_{\xi^{\nu}}} \ar[dr]_{f|_{\xi^{\nu}}} & \xi_e^{\nu} \ar[d]^{f_e|_{\xi_e^{\nu}}} \\
     & \zeta^{\nu}.
    }
    \]
    The purely inseparable parts of $f|_{\xi^{\nu}}$ and of ${f_e|_{\xi_e^{\nu}} \circ \alpha_e|_{\xi^{\nu}}}$ have the same degree, whence the conclusion.
\end{proof}

\begin{example}
Here, we show an example where $e_0$ in the above \autoref{l-lemma!} is non-zero.

Let $X:= \A^2_{(x,t)}$ and $Z:= \A^1_t$ with $f \colon X \to Z$ being the projection on the second coordinate. Let $e_0$ be a positive integer and let $\xi$ be the curve defined by $x^{p^{e_0}}-t=0$. Note that $X^{(e_0)}= \A^2_{(x, \tau)}$, where $\tau^{p^{e_0}}=t$. Then, the image $\xi_{e_0}$ of $\xi$ in $X^{(e_0)}$ is the zero locus of $x-\tau$, $f|_{\xi}$ and $\alpha_{e_0}|_{\xi}$ coincide with $F^{e_0}$, while $f_{e_0}|_{\xi_{e_0}}$ is separable.
\end{example}

\begin{remark}
    Let $X$ be a normal variety and $\cF$ a foliation on $X$.
    A subvariety $W$ is said to be \textbf{tangent} to $\cF$ if $T_W \subseteq \cF$. We point out that, over fields of positive characteristic, tangent subvarieties behave differently than in the characteristic $0$ case.

    Indeed, let $\cF$ be the foliation induced by a separable fibration $f\colon X \to Z$ with normal general fibres.
    Then, there may be curves that are tangent to $\cF$, but that are horizontal. Moreover, their image in $X^{(e)}$ may be not tangent to $\cF_e$, the foliation induced by $f_e$.
    In fact, if $\xi$ is a curve such that the purely inseparable degree of $f|_{\xi}$ is $p^{e_0}$, for some $e_0>0$, then $T_{\xi} \subseteq \ker(df)$, whereas, for $e \geq e_0$, $T_{\xi_e} \not\subseteq \ker(df_e)$, where $\xi_e$ is the image of $\xi$ in $X^{(e)}$.

   Furthermore, there may be curves that are vertical, but not tangent to $\cF$. For example, if $f\colon \A^2 \to \A^1$ is defined by $(x,y) \mapsto x^p(x+y)$, then the induced foliation is generated by $\partial_x-\partial_y$ and the curve defined by $x=0$ is vertical, but it is not tangent to $\cF$.
\end{remark}

\begin{lemma} \label{l-correspondence of foliations in diagram snowflake}
    If $f\colon X \to Z$ is a separable fibration between normal varieties, the foliations induced by $f$ and $\alpha_e$ coincide for every $e >0$. 
\end{lemma}

\begin{proof}
    Fix $e>0$. By \cite[Lemma 2.2]{FoliationAdj}, we can check whether the foliation induced by $f$ and the one induced by $\alpha_e$ coincide on a dense open subset of $X$.
    Therefore, since $f$ is separable, up to restricting to an open dense subset, we may assume that $X_{Z^e}$ is normal, $\alpha_e$ is the $e^{\text{th}}$ relative Frobenius of $X$ over $Z$, $f$ is smooth, and both $\ker(d\alpha_e)$ and $\ker(df)$ are saturated.

    Let $\xi \subseteq X$ be a curve.
    If $\xi$ is vertical, both $df|_{\xi}$ and $d \alpha_e|_{\xi}$ are $0$.
    On the other hand, if $\xi$ is horizontal, by \autoref{l-lemma!}, $\alpha_e|_{\xi}$ is purely inseparable if and only if the purely inseparable degree of $f|_{\xi}$ is $>1$.
    Therefore, $d\alpha_e|_{\xi}=0$ if and only if $df|_{\xi}=0$.

    Given a general point $x \in X$, if $v \in \ker(df)_x$, there exists a curve $\xi$ passing through $x$ with tangent vector $v$ and the same holds for $\ker(d\alpha_e)$. This implies that $\ker(df)$ and $\ker(d\alpha_e)$ coincide on an open subset $U \subseteq X$.
\end{proof}

\subsection{Wild multiplicities} \label{s-f3}

We now set the ground for the base change formula that we prove in \autoref{s-formulina}.
In particular, we study how the base change \autoref{construction} of a fibration with a power of the Frobenius morphism modifies the ``multiplicities'' of horizontal and vertical divisors.

In this section, we use the notation of \autoref{construction} in \autoref{s-purelyinsep_basechange}.

\begin{lemma} \label{l-coeffsdivisors}
    Let $f\colon X \rightarrow Z$ be a fibration between normal varieties such that the geometric generic fibre $X_{\widebar{\eta}}$ is normal. Let $D \subset X$ be a horizontal prime divisor.
    Let $\eta$ be the generic point of $Z$ and assume that $D_{\widebar{\eta}}= p^{e_0}D_{\widebar{\eta},\mathrm{red}}$ at the generic point of $D_{\widebar{\eta}}$ for some integer $e_0 \geq 0$.
    Then, at the generic point of $D_{\widebar{\eta}}$, we have:
    \[
     D_{\widebar{\eta}} = 
        \begin{cases}
        p^e D^{(e)}_{\widebar{\eta}} \; \text{if} \; e \leq e_0 \\
        p^{e_0} D^{(e)}_{\widebar{\eta}} \; \text{otherwise},
        \end{cases}
    \]
    and
    \[
    \beta_e^*D = 
        \begin{cases}
        p^e D^{(e)} \; \text{if} \; e \leq {e_0} \\
        p^{e_0} D^{(e)} \; \text{otherwise},
        \end{cases}
    \qquad
    \alpha_e^*D^{(e)} =
    \begin{cases}
        D \; \text{if} \; e \leq {e_0} \\
        p^{e-e_0} D \; \text{otherwise}.
    \end{cases}
    \]
    Moreover, $D^{({e_0})}_{\widebar{\eta}}$ is reduced at its generic point.
\end{lemma}
    
\begin{proof}
    First of all, note that, by \autoref{l-conductor is vertical}, around the generic point of $D$, we can assume $X_{Z^e}$ is normal.
    By cutting $Z$ with general hyperplanes, we can assume that $f \colon X \to Z$ is a fibration onto a curve.
    Let $f_e|_D\colon D^{(e)} \rightarrow Z$ be the induced map on $D^{(e)}$.
        By the universal property of the fibre product and since $Z$ is a curve, $f_{e_0}|_{D^{(e_0)}}$ is separable, thus $D^{({e_0})}_{\widebar{\eta}}$ is reduced by \autoref{p-reducedness}.
        We will prove the lemma by induction on ${e_0}$.
        If ${e_0}=0$, $f_e|_{D^{(e)}}$ is separable for each $e \in \N$, so $D^{(e)}_{\widebar{\eta}}$ is reduced and it coincides with $D_{\widebar{\eta}}$.
        In particular, $\beta_e^*D=D^{(e)}$ and $\alpha_e^*D^{(e)}=p^eD$ for all $e>0$.
        If ${e_0} >0$, consider the natural maps $X^{({e_0})} \rightarrow X^{(1)} \rightarrow X$.
        By the universal properties of the fibre product and since $Z$ is a curve, $D$ and $D^{(e)}$ are isomorphic at their generic points for all $e \leq {e_0}$, thus $f_e|_{D^{(e)}}$ has purely inseparable degree $p^{{e_0}-e}$.
        In particular, $f_{e_0}|_{D^{({e_0})}}$ is separable and $f_1|_{D^{(1)}}$ has purely inseparable degree $p^{{e_0}-1}$.
        The map $X^{(1)} \rightarrow X$ is purely inseparable of degree $p$ and $D^{(1)} \rightarrow D$ is an isomorphism. Thus $\beta_1^*D = p D^{(1)}$.
        Then, we conclude by the inductive assumption.
\end{proof}


\begin{remark} \label{r-tameness in diagram snowflake}
    Let $f\colon X \to Z$ be a tame separable equidimensional fibration between normal varieties.
    Then, for all $e> 0$, $f_e$ is tame as well.
    Indeed, if $\delta$ is a prime divisor in $Z$,
    \[
        p^e f_e^*\delta= \beta_e^*f^*\delta = \sum_{D \text{ over } \delta} \ell_D \beta_e^*D= \sum_{D \text{ over } \delta} \ell_D p^{d_D}D^{(e)},
    \]
    where the divisors $D$ and $D^{(e)}$ are prime and reduced, $\ell_D$ is coprime with $p$ by the tameness assumption on $f$ and $d_D \leq e$ since $\beta_e$ is a purely inseparable morphism factorising $F^e$.
    Hence, we have $d_D=e$ and the multiplicity of the vertical divisor $D^{(e)}$ with respect to $f_e$ is again $\ell_D$.
\end{remark}

\begin{lemma} \label{p-multiplicities wild components}
Let $f\colon X \to Z$ be a separable equidimensional fibration between normal varieties and $D$ a vertical prime divisor with multiplicity $\ell = np^{e_0}$, for some $n$ coprime with $p$ and $e_0 \geq 0$.
Then, the multiplicity of $D^{(e)}$ with respect to $f_e$ is
$\ell_e = np^{e_0-h_e}$, for some integer $0\leq h_e\leq \min\{e,e_0\}$.
In particular, $\beta_e^*D = p^{e-h_e}D^{(e)}$ and $\alpha_e^*D^{(e)}=p^{h_e} D$.
\end{lemma}

\begin{proof}
\noindent \textbf{Step 1.}
If $f$ is tame, i.e.\ $e_0=0$, we have $\ell_{e}=n$ for all $e \geq 0$ by \autoref{r-tameness in diagram snowflake}, whence the conclusion.

\noindent \textbf{Step 2.}
Now, suppose $e_0 >0$.
Since $f$ is equidimensional, by \cite[Lemma 2.19]{Lena-Joe} (as a consequence of the flattening lemma \cite[\S 3.3]{flattening}, \cite[Théorème 5.2.2]{RG}), up to restricting the fibration to a big open subset, we can assume $f\colon X \to Z$ is a flat fibration between smooth varieties.
Let $d:= \dim(X) -\dim(Z)$.
By localising around the generic point of $f(D)$ (or by cutting $Z$ with general hyperplanes), we can assume that $Z$ is a curve.
Then, by possibly restricting to open dense subsets around $D$, by \cite[Tag 039P]{stacks} there exist étale morphisms $\varphi$ and $\psi$ and a fibration $f^{\mathrm{et}}$, fitting in the following diagram: 
\[
\begin{tikzcd}
X \arrow[d, "f", swap] \arrow[r, "\psi"] & A:= \Spec(k[x,y_1, ..., y_d]) \arrow[d, "f^{\mathrm{et}}"]\\
Z \arrow[r, "\varphi"] & B:= \Spec(k[t]).
\end{tikzcd}
\]
We can assume $D=V(x) \subseteq A$ and $t= x^{np^{e_0}}u$, where $u \in k[x, y_1, ..., y_d]$ is a unit around $D$ and, since $f$ is separable, its Jacobian has generic rank $1$.

\noindent \textbf{Step 3.} In this step, we prove the proposition for the morphism $f^{\mathrm{et}} \colon A \to B$.

Consider the diagram:
\[
\begin{tikzcd}
A \arrow[dr, "f^{\mathrm{et}}", swap] \arrow[r, "\alpha_1^{\mathrm{et}}"] & A^{(1)} \arrow[d, "f_1^{\mathrm{et}}"] \arrow[r, "\beta^{\mathrm{et}}_1"] & A \arrow[d, "f^{\mathrm{et}}"]\\
& B:= \Spec(k[\tau]) \arrow[r, "F"] & B,
\end{tikzcd}
\]
where $\tau^{p}=t$ and $A^{(1)}$ is the normalisation of ${\Spec(k[x, y_1, ..., y_d, \tau]/(\tau^{p}-x^{np^{e_0}}u))}$.
Note that, since $A^{(1)}$ is normal, it contains an element $z$ such that $z x^{np^{e_0-1}}=\tau$. The map $f_1^{\mathrm{et}}$ is then described as:
\[
    A^{(1)}= \left(\Spec\left(\frac{k[x, y_1,...,y_d,z]}{(z^{p}-u)}\right)\right)^{\nu} \rightarrow B; \qquad \tau= z x^{np^{e_0-1}}.
\]
Since $\beta_1^{\mathrm{et}}$ factorises $F\colon A \to A$, either $(x=0)$ is irreducible or there exists $\xi$ in the section ring of $A^{(1)}$ such that $\xi^p=x$.
In the first case, $D^{(1)}=V(x)$ and its multiplicity with respect to $f_1^{\mathrm{et}}$ is $np^{e_0-1}$.
In the second case, $D^{(1)}=V(\xi)$ and its multiplicity with respect to $f_1^{\mathrm{et}}$ is $np^{e_0}$.
In particular, $\alpha_1^{\mathrm{et}*}D^{(1)}=pD$ and $\beta_1^{\mathrm{et}*}D=D^{(1)}$ in the first case, while $\alpha_1^{\mathrm{et}*}D^{(1)}=D$ and $\beta_1^{\mathrm{et}*}D=pD^{(1)}$ in the second case.

If $e >1$, we decompose the base change as $\beta_e^{\mathrm{et}}=\beta_1^{\mathrm{et}} \circ B_{e-1}^{\mathrm{et}}$, where $B_{e-1}^{\mathrm{et}}$ is induced by the normalised base change of $A^{(1)}$ along $F^{e-1}\colon Z \to Z$, and we prove the statement by induction on $e$.

\noindent \textbf{Step 4.}
We claim that, around $D^{(e)} \subseteq X^{(e)}$, there are étale morphisms $\varphi_e$ and $\psi_e$ and a fibration $f_{\mathrm{et}}^e$, fitting in the following diagram:
\[
\begin{tikzcd}
X^{(e)} \arrow[d, "f_e", swap] \arrow[r, "\psi_e"] & A^{(e)} \arrow[d, "f_e^{\mathrm{et}}"]\\
Z \arrow[r, "\varphi_e"] & B.
\end{tikzcd}
\]
Since multiplicities can be computed étale locally, we conclude by Step 3.

Now, let us show the claim.
Note that, by commutativity of the Frobenius morphism, we can choose $\varphi_e:= \varphi$.
Moreover, by the universal properties of the fibre product and since $X^{(e)}$ is normal, there is a map $X^{(e)} \xrightarrow{\psi_e} A^{(e)}$. We need to show that it is étale.
Let ${C:= X \times_A A^{(e)} \xrightarrow{\chi} A^{(e)}}$.
Since $\psi$ is étale and $A^{(e)}$ is normal, $\chi$ is étale and $C$ is normal as well. By construction, there is a map $X^{(e)} \to C$.
By the universal properties of the normalisation and of the fibre product, we construct also a map $C \to X^{(e)}$. By diagram chasing, we conclude that $C$ and $X^{(e)}$ are isomorphic and $\psi_e=\chi$ is étale.
\end{proof}

\begin{example}
    \begin{itemize}
        \item Let $f\colon \Spec(k[x, y] \to k[t])$ be the fibration induced by $t=x^py$ and let $D=V(x)$.
        Then the normalisation of $k[x, y, \tau]/(\tau^p=x^py)$ is $k[x, z]$ such that $z^p=y$ and $\tau=xz$, $\beta_1^*D$ is reduced and its multiplicity with respect to $f_1$ is $1$.
        \item Let $f\colon \Spec(k[x, y] \to k[t])$ be the fibration induced by $t=x^p(x+y^p)$ and let $D=V(x)$.
        The normalisation of $k[x, y, \tau]/(\tau^p=x^p(x+y^p))$ contains an element $z$ such that $z^p=x+y^p$. In particular $V(x)=V((z-y)^p)$ is not reduced.
    \end{itemize}
\end{example}

\begin{corollary} \label{c-ramifications}
    Let $f \colon X \to Z$ be a separable equidimensional fibration between normal varieties. For every wild component $D$, we denote by $e_D$ the integer such that the multiplicity of $D$ is $\ell_D= n_Dp^{e_D}$, for some $n_D$ coprime with $p$.
    Let $\cW$ be the set of all wild components of $f$.
    Then, for all $e \geq 1$, there exist $\cW_e \subseteq \cW$ and positive integers $1 \leq h_D \leq e_D$, such that
    \[
        \alpha_e^*R(f_e) = R(f) - \sum_{D \in \cW_e} (p^{h_D}-1)D.
    \]
    Moreover, for $e \gg 0$, the term $\sum_{D \in \cW_e} (p^{h_D}-1)D$ is independent of $e$.
\end{corollary}

\begin{proof}
    Let $D \subseteq X$ be a vertical prime divisor and $\delta:= f(D) \subseteq Z$. 
    We do local computations around $D$.
    Denote $f_e^*\delta = \ell_{D^{(e)}} D^{(e)}$.
    By \autoref{p-multiplicities wild components}, there exists $0\leq h_e \leq e_D$ such that $\ell_{D^{(e)}}=p^{e_D-h_e}n_D$.
    Moreover, $\alpha_e^*D^{(e)}=p^{h_e}D$.
    Hence, 
    \[
        \alpha_e^*R(f_{e})= (p^{e_D-h_e}n_D-1)\alpha_e^*D^{(e)}=R(f)-(p^{h_e}-1)D.
    \]
    As for the ``moreover'' part, note that $\cW$ is a finite set and, for each wild component, its multiplicity along the morphisms $f_e$ can drop only a finite number of times. In particular, for $e \gg 0$, the values $h_e$ are independent of $e$.
\end{proof}

\subsection{Base change formula} \label{s-formulina}

In the papers \cite[Corollary 3.4]{Ekedahl}, \cite{LMMP} and \cite{Lena-Joe}, the authors give a very explicit description of the relative canonical bundle of a purely inseparable base change of height $1$.
In the next two sections, we study similar relations for the base change described in \autoref{construction} in \autoref{s-purelyinsep_basechange}.

In this section we use the notation of \autoref{construction} in \autoref{s-purelyinsep_basechange}.

\begin{proposition}[{\cite[Proposition 2.10]{LMMP},  \cite[Proposition 9.1.2.3]{SB}, \cite[{Corollary 3.4}]{Ekedahl}}] \label{p-foliations and relative canonical}
Let ${X \rightarrow X'}$ be a purely inseparable morphism of height one between normal varieties and let $\mathcal{F}$ be the corresponding foliation.
Then
\[
\omega_{X/X'} \simeq (\det \mathcal{F})^{[p-1]},
\]
where by $[p-1]$ we denote that we are taking the reflexified tensor product $p-1$ times.
\end{proposition}


\begin{corollary} \label{c-formulina}
If $\mathcal{F}$ is the foliation induced by a separable equidimensional tame fibration $f\colon X \rightarrow Z$ between normal varieties such that the geometric generic fibre $X_{\widebar{\eta}}$ is normal, 
then:
\[
\alpha_e^*K_{X^{(e)}} = (p^e-1) K_{\mathcal{F}} + K_X \quad \text{and} \quad \alpha_e^*K_{\cF_e} = p^e K_{\mathcal{F}}.
\]
\end{corollary}

\begin{proof}
We prove the statement by induction on $e$.
For $e=1$,
\[
    \alpha_1^*K_{X^{(1)}} = (p-1)K_{\cF} +K_X,
\]
by \autoref{p-foliations and relative canonical} and \autoref{l-correspondence of foliations in diagram snowflake}.
Since $f$ is tame, by \autoref{t-canonical foliation} and \autoref{c-ramifications}, $\alpha_1^*K_{\cF_1}= pK_{\cF}$.

If $e>1$, factorise the diagram in \autoref{construction} in \autoref{s-purelyinsep_basechange}, in the following way:
\[
\xymatrix{
X \ar[rd]_{\alpha_{e-1}} \ar[rrd]^{\alpha_e} \ar@/_3pc/[ddrr]_{f} & & & \\
 		& X^{(e-1)} \ar[dr]_(.4){f_{e-1}} \ar[r]^(.4){\delta} & X^{(e)} \ar[d]^{f_e} \ar[r] & X^{(e-1)}  \ar[d]^{f_{e-1}}\\
 & & 												Z \ar[r]^{F} 			& Z.
}
\]
Let $\cF_{e-1}$ be the foliation induced by $f_{e-1}$.
By \autoref{p-foliations and relative canonical} and \autoref{l-correspondence of foliations in diagram snowflake} applied to the lower part of the diagram above,
\[
\alpha_e^*K_{X^{(e)}} = \alpha_{e-1}^*\delta^*K_{X^{(e)}} = (p-1)\alpha_{e-1}^*K_{\cF_{e-1}} + \alpha_{e-1}^*K_{X^{(e-1)}}.
\]
Since $f_{e-1}$ is tame by \autoref{r-tameness in diagram snowflake}, we get that $K_{\cF_{e-1}} = {K_{X^{(e-1)}}- f_{e-1}^*K_Z-R(f_{e-1})}$ by \autoref{t-canonical foliation}.

By induction, we know that
\[
\alpha_{e-1}^*K_{X^{(e-1)}} = (p^{e-1}-1) K_{\mathcal{F}} + K_X.
\]
Using that $K_{\cF}= K_X -f^*K_Z -R(f)$ by \autoref{t-canonical foliation} and that $\alpha_{e-1}^*R(f_{e-1})=R(f)$ by \autoref{c-ramifications}, we get the result.
\end{proof}

\begin{theorem} \label{t-formulina_general}
    Let $f\colon X \to Z$ be a separable equidimensional fibration between normal varieties such that the geometric generic fibre $X_{\widebar{\eta}}$ is normal.
    Then,
    \[
        \alpha_e^*K_{X^{(e)}} = (p^e-1)(K_X-f^*K_Z-R(f)) +K_X - \sum_{D \text{ wild}}w_{D,e}D,
    \]
    where $w_{D,e} \geq 0$ for all $e$ and all $D$ wild components.

    More precisely, if the multiplicity of $D$ with respect to $f$ is $\ell=np^{e_0}$, with $n$ coprime with $p$, then $w_{D,e} \geq p^{e}-1$ if $e \leq e_0$ and $w_{D,e} \geq p^{e_0}-1$ otherwise.
\end{theorem}

\begin{proof}
    Let $D \subseteq X$ be a vertical prime divisor.
    We prove the statement locally around $D$.
    By \autoref{c-formulina}, if $f$ is tame around $D$, the theorem holds. Therefore, we can suppose $D$ is a wild component of multiplicity $\ell= np^{e_0}$, with $n$ not divisible by $p$.
    
    We prove the statement by induction on $e$.
    If $e=1$, by \autoref{p-foliations and relative canonical}, \autoref{l-correspondence of foliations in diagram snowflake} and \autoref{t-canonical foliation},
    \[
        \alpha_1^*K_{X^{(1)}}= (p-1)(K_X-f^*K_Z-R(f)) +K_X -(p-1)(a_D+1)D,
    \]
    for some integer $a_D\geq 0$.
    Now, assume $e>1$ and assume that $D^{(e-1)}$ is still a wild component for $f_{e-1}$.
    Factorise the diagram in \autoref{construction} in \autoref{s-purelyinsep_basechange}, in the following way:
    \[
    \xymatrix{
    X \ar[rd]_{\alpha_{e-1}} \ar[rrd]^{\alpha_e} \ar@/_3pc/[ddrr]_{f} & & & \\
             & X^{(e-1)} \ar[dr]_(.4){f_{e-1}} \ar[r]^(.4){\delta} & X^{(e)} \ar[d]^{f_e} \ar[r] & X^{(e-1)}  \ar[d]^{f_{e-1}}\\
     & & 												Z \ar[r]^{F} 			& Z.
    }
    \]
   By \autoref{p-multiplicities wild components}, the multiplicity of $D^{(e)}$ is $\ell_e=np^{e_0-h_e}$, for some $0 \leq h_e \leq e_0$.
    Then,
\begin{align*}
&\alpha_e^*K_{X^{(e)}}= \alpha_{e-1}^* \delta^*K_{X^{(e)}}\\
&=\alpha_{e-1}^*((p-1)K_{\cF_{e-1}}+K_{X^{(e-1)}})
    \hspace{44pt}\textup{\footnotesize{by \autoref{p-foliations and relative canonical} and \autoref{l-correspondence of foliations in diagram snowflake}}}\\
&= \alpha_{e-1}^*((p-1)(K_{X^{(e-1)}}-f_{e-1}^*K_Z-(np^{e_0-h_{e-1}}-1)D^{(e-1)})
\\&-(p-1)(b_D+1)D^{(e-1)}+K_{X^{(e-1)}})   \hspace{30pt}\textup{\footnotesize{by \autoref{t-canonical foliation}}}\\
& = p(p^{e-1}-1)(K_X-f^*K_Z-R(f)) +pK_X -(p-1)f^*K_Z-(p-1)(np^{e_0}-1)D \\
& -p w_{e-1}D+(p-1)(p^{h_{e-1}}-1)D - p^{h_{e-1}}(p-1)(b_D+1)D\\
 &    \hspace{212pt}\textup{\footnotesize{by induction and \autoref{p-multiplicities wild components}}}\\
&=(p^e-1)(K_X-f^*K_Z-R(f)) +K_X -w_e D,
\end{align*}
where $b_D \geq 0$ and $w_{e-1} \geq \max\{p^{e-1}-1, p^{e_0}-1\}$ by the inductive step. Therefore, 
\[
    w_e= pw_{e-1}- (p-1)(p^{h_{e-1}}-1)+p^{h_{e-1}}(p-1)(b_D+1) \geq  \max\{p^{e}-1, p^{e_0}-1\}.
\]

Now, let $e'$ be the smallest integer such that $D^{(e')}$ is not a wild component of $f_{e'}$ (if it exists, it is certainly $\geq e_0$) and let $e \geq e'$. Consider the factorisation
\[
\xymatrix{
    X \ar[r]^{\alpha_{e'}} \ar@/^1.7pc/[rr]^{\alpha_e} & X^{(e')} \ar[r]^{A_e} & X^{(e)} \ar[r]^{B_e} \ar@/^1.7pc/[rr]^{\beta_e} & X^{(e')} \ar[r]^{\beta_{e'}} & X.
}
\]
Since $f_e$ is tame around $D^{(e)}$ for $e\geq e'$, we apply \autoref{c-formulina} to $X^{(e')} \xrightarrow{A_e} X^{(e)} \xrightarrow{B_e} X^{(e')}$. After that, we apply the above computation to the morphisms ${X \xrightarrow{\alpha_{e'}} X^{(e')} \xrightarrow{\beta_{e'}} X}$.
All in all,
\begin{align*}
    &\alpha_e^*K_{X^{(e)}}= \alpha_{e'}^*A_e^*K_{X^{(e)}}\\
    & =\alpha_{e'}^*
    ((p^{e-e'}-1)(K_{X^{(e')}}-f_{e'}^*K_Z-(n-1)D^{(e')})+K_{X^{(e')}})\\ 
    & \hspace{190pt}\textup{\footnotesize{by \autoref{c-formulina}}}\\
    & = p^{e-e'}(p^{e'}-1)(K_X-f^*K_Z-R(f)) +p^{e-e'}K_X -(p^{e-e'}-1)f^*K_Z\\ 
    & - (p^{e-e'}-1)R(f) -p^{e-e'}w_{e'}D + (p^{e-e'}-1)(p^{e_0}-1)D \\
    & \hspace{190pt}\textup{\footnotesize{by the above computation and \autoref{p-multiplicities wild components}}}\\
    & = (p^e-1)(K_X-f^*K_Z-R(f)) +K_X -w_e D,
\end{align*}
where
\[
    w_e = p^{e-e'}w_{e'} - (p^{e-e'}-1)(p^{e_0}-1) \geq p^{e_0}-1.
\]
\end{proof}

\begin{remark}
    Let $f\colon X \to Z$ be an equidimensional separable fibration between normal varieties with normal geometric generic fibre and let $\cF$ be the induced foliation.
    For every $e \geq 0$, let $\cF_e$ be the foliation induced by $f_e$.
    The sequence $(\cF, \cF_1, ..., \cF_e,...)$ (resp.\ $(\cF, \cF_1, ..., \cF_e)$) is an \emph{$\infty$-foliation} (resp.\ an \emph{$e$-foliation}) in the sense of \cite[Definition 2.19, Definition 4.6]{Chemy}.
    Combining the above \autoref{t-formulina_general} with \autoref{t-canonical foliation}, we see that
    \[
        \alpha_e^*K_{X^{(e)}} = (p^e-1)K_{\cF} +K_X + \sum_{D \text{ wild}} ((p^e-1)(a_D+1)-w_{D,e})D.
    \]
    We compare this with the results in \cite{Chemy_thesis}, which state that
    \[
        \alpha_e^*K_{X^{(e)}} = (p^e-1)K_{\cF} +K_X + E,
    \]
    for $E$ an effective $\Q$-divisor such that $E=0$ if and only if $(\cF, \cF_1, ..., \cF_e)$ is \emph{Ekedahl} in the sense of \cite{Chemy_thesis} (and \cite{Ekedahl}).
    Therefore, we conclude that $E$ is supported on the wild components, $w_{D,e} \leq (p^e-1)(a_D+1)$ and $(\cF, \cF_1, ..., \cF_e)$ is Ekedahl if and only if $f$ is a tame fibration.
    We expect that, if the geometric generic fibre of $f$ is not normal, the divisor $E$ is supported on the wild components \textit{and} on the (horizontal) singular locus of the fibres.
\end{remark}

\section{Moduli and discriminant parts} \label{s-CBF ACSS}

Classically, the canonical bundle formula has been studied using variations of Hodge structures on \textit{$K$-trivial} fibrations $f\colon X \to Z$, i.e.\ fibrations where $K_X$ is the pull-back of some $\Q$-Cartier $\Q$-divisor on $Z$ (see \cite{Kaw98,FujinoMori,Shokurov,modulibdiv,KollarCBF}).
In \cite{positivity}, the authors study the canonical bundle formula in characteristic $0$ using tools of the Minimal Model Program.
In this setting it is useful to extend our focus to more general fibrations, not only the $K$-trivial ones.
This will give us the necessary flexibility to perform birational transformations.

\begin{definition} \label{lct} \label{d-discriminant}
    Let $f\colon X \rightarrow Z$ be a fibration between normal varieties and $B$ a $\Q$-divisor on $X$ such that $K_X+B$ is $\Q$-Cartier.
    For each prime divisor $\delta \subseteq Z$, define
    \[
    \gamma_{\delta} := \sup \{ t \in \R \, \text{ s.t. } (X, B+ tf^*\delta) \, \text{ is log canonical over the generic point of $\delta$} \}.
    \] 

Let $f\colon X \rightarrow Z$ be a fibration between normal projective varieties and $(X/Z, B)$ a GGLC pair on it (see \autoref{d-GGLC}).
 The \textbf{discriminant part} of $f$ is
    \[
    B_Z:= \sum_{\delta \subseteq Z} (1 - \gamma_{\delta})\delta,
    \]
    where the sum is taken over all prime divisors of $Z$.

Let $f\colon X \rightarrow Z$ be a fibration between normal varieties and $(X/Z, B)$ a GGLC pair on it such that $B^h \geq 0$. 
Let $B_Z$ be the discriminant part of $f$.
Suppose that $f$ is equidimensional or that $K_Z+ B_Z$ is $\Q$-Cartier.
Then, the \textbf{moduli part} of $f$ is
\[
M_X := K_X+B-f^*(K_Z+B_Z).
\]
Note that $M_X$ is defined on the total space $X$ and only up to linear equivalence.
\end{definition}

\begin{proposition} \label{p-well-defined}
Assume inversion of adjunction in dimension $n$ and the existence of log resolutions of singularities in dimension $d$.
Let $X$ be a variety of dimension $n$ and $f\colon X \rightarrow Z$ be an equidimensional fibration between normal varieties such that $\dim(X)-\dim(Z)=d$.
Let $(X/Z, B)$ be a GGLC pair associated with it with $B^h \geq 0$.
Then, the discriminant part $B_Z$ is well-defined. Consequently, the moduli part $M_X$ is also well-defined. 
\end{proposition}

\begin{proof}
By \autoref{p-general vs geometric generic lc}, a general fibre $X_z$ is normal and the pair defined by adjunction $(X_z, B_{z})$ is log canonical.
Let $U \subseteq Z$ be the dense open subset of $Z$ such that $(X_z, B_z)$ is log canonical with $B_z \geq 0$ for all $z \in U$, and $f^*\delta$ is integral for every $\delta$ prime divisor not contained in $Z \setminus U$.
Let $\delta$ be a prime divisor in $Z$ which is not contained in $Z \setminus U$ and let $X_{\delta}$ be the fibre over $\delta$.
We claim that $\gamma_{\delta}=1$.
If this was not the case, there would exist a non-log canonical place $E$ of $(X_{\delta}, B_{\delta})$ such that its centre contains the generic point of $\delta$.
In particular, by adjunction, this would imply that, for $z \in \delta$ general, there exists a non-log canonical place of $(X_z, B_{z})$, which is a contradiction.
By inversion of adjunction, we conclude that $(X,B+f^*\delta)$ is log canonical around the generic point of $\delta$.

In particular, $\gamma_{\delta}=1$ for all, but finitely many prime divisors $\delta \subseteq Z$, therefore $B_Z$ and consequently $M_X$ are well-defined $\Q$-divisors.
\end{proof}

\begin{remark}
    Note that we need to assume inversion of adjunction in dimension $n$ and existence of log resolutions of singularities in dimension $d$ even just to prove that the discriminant and moduli parts are well-defined (\autoref{p-general vs geometric generic lc} and \autoref{p-well-defined}). For this reason, we will make these assumptions in many results later even if we do not explicitly use them.
\end{remark}

\section{Property \texorpdfstring{$(\ast)$}{} in positive characteristic} \label{s-property star}

\subsection{General properties}

The main idea behind the proof of positivity of the moduli part consists of two steps. First, we find a class of fibrations for which it is easier to get the result. Then, we reduce to this case with a birational modification.
The authors of \cite{positivity} introduce the notion of Property $(\ast)$: when a fibration satisfies this property, the moduli part coincides with the canonical divisor of the foliation associated with it.
In this context, we can use the birational geometry of the foliation to get the desired positivity.

\begin{definition}[{\cite[Definition 2.13]{positivity}}] \label{d-Prop star}
Let $f\colon X \rightarrow Z$ be a fibration between normal varieties and $(X/Z, B)$ a GGLC pair on it with $B \geq 0$.
We say it satisfies \textbf{Property $(\ast)$} if:
\begin{itemize}
\item[(a)] there exists a reduced divisor $\Sigma_Z$ on $Z$ such that $(Z, \Sigma_Z)$ is log smooth and $B^v = f^{-1}(\Sigma_Z)$;
\item[(b)] for any closed point $z \in Z$ and any reduced divisor $\Sigma \geq \Sigma_Z$ such that $(Z, \Sigma)$ is log smooth around $z$, then $(X, B+ f^*(\Sigma-\Sigma_Z))$ is log canonical around $f^{-1}(z)$.
\end{itemize}
\end{definition}

In the next proposition, we list some useful features that Property $(\ast)$ pairs enjoy.

\begin{proposition}[{\cite[Lemma 2.14, Proposition 2.18]{positivity}}] \label{p-properties property star}
Assume inversion of adjunction in dimension $n$ and the existence of log resolutions of singularities in dimension $d$.
Let $X$ be a variety of dimension $n$ and $f\colon X \rightarrow Z$ a fibration between normal varieties such that $\dim(X)-\dim(Z)=d$.
Let $(X/Z, B)$ be a GGLC pair on it satisfying {Property $(\ast)$} with $B\geq 0$.
Then the following properties hold.
\begin{itemize}
\item[(i)] The pair $(X, B)$ is log canonical and the discriminant part $B_Z$ coincides with $\Sigma_Z$. Moreover,if $B \geq 0$, the map $f$ is equidimensional outside $\Sigma_Z$.
\item[(ii)] Suppose $B \geq 0$ and let $\varphi\colon X \dashrightarrow Y$ be a sequence of steps of the $(K_X+ B)$-MMP over $Z$. Let $C:= \varphi_*B$ and $g\colon Y \rightarrow Z$ be the induced fibration. Then, $(Y/Z, C)$ is a GGLC pair by \autoref{l-normal_afterMMP}, it satisfies Property $(\ast)$ and for any closed point $z \in Z$, the map $\varphi^{-1}$ is an isomorphism along the generic point of any irreducible component of $g^{-1}(z)$.
In particular, the discriminant part of $(Y/Z, C)$ coincides with the discriminant part of $(X/Z,B)$.
\end{itemize}
\end{proposition}

\subsection{Moduli part}

We use the notation of \autoref{construction} in \autoref{s-purelyinsep_basechange}.

When the fibration satisfies Property $(\ast)$, the moduli part has an easy description.

\begin{proposition} \label{p-M_X for property star}
    Assume inversion of adjunction in dimension $n$ and the existence of log resolutions of singularities in dimension $d$.
Let $X$ be a variety of dimension $n$ and $f\colon X \rightarrow Z$ be an equidimensional fibration between normal varieties such that $\dim(X)-\dim(Z)=d$.
    Let $(X/Z,B)$ be a GGLC pair associated with it which satisfies Property $(\ast)$ with $B \geq 0$.
    Then
    \[
        M_X = K_X + B^h -f^*K_Z - R(f),
    \]
    where $R(f):= \sum_{\delta \subseteq Z} (f^*\delta-f^{-1}(\delta))$ is as in \autoref{d-tame}.
\end{proposition}

\begin{proof}
    The proof follows \cite[Proposition 3.6]{positivity}, which states a similar equality over fields of characteristic $0$.
    We need to show that, if $B_Z$ is the discriminant part, $B^v-f^*B_Z=-R(f)$.
    Since $(X/Z, B)$ has Property $(\ast)$, $B^v$ and $B_Z$ are reduced $\Q$-divisors and $B^v=f^{-1}(B_Z)$.
    Let $D$ be a vertical prime divisor and let $\ell_D$ be its multiplicity with respect to $f$. Then, the coefficient of $D$ in $B^v-f^*B_Z$ is $1-\ell_D$, as in $-R(f)$.
\end{proof}

\begin{corollary} \label{r-formulina}
    Assume inversion of adjunction in dimension $n$ and the existence of log resolutions of singularities in dimension $d$.
Let $X$ be a variety of dimension $n$ and $f\colon X \rightarrow Z$ be an equidimensional fibration between normal varieties such that $\dim(X)-\dim(Z)=d$.
    Let $(X/Z, B)$ be a GGLC pair on it satisfying {Property $(\ast)$} with $B \geq 0$.
    For every $e \geq 0$:
    \[
        \alpha_e^*K_{X^{(e)}} = (p^e-1)(M_X-B^h)+K_X - \sum_{D \text{ wild}} w_{D,e} D,
    \]
    where $w_{D,e} \geq 0$ for every wild component $D$.
    \end{corollary}
  
    \begin{proof}
        Combine \autoref{t-formulina_general} with \autoref{p-M_X for property star}.
    \end{proof}

\subsection{Existence of \texorpdfstring{$(\ast)$}{}-modifications}

Given any GLC pair in characteristic $0$, in \cite[Theorem 2.2, Proposition 2.16]{positivity} the authors construct a birationally equivalent model that satisfies Property $(\ast)$ thanks to the existence of toroidal modifications as proven in \cite{AK}.
In characteristic $p>0$, however, we cannot always use this construction. 
In fact, in one of the steps, the idea of the proof is to inductively apply the Semistable Reduction theorem for curves and consider the quotient by the action induced by the finite cover needed to obtain the semistable model. However, the variety we obtain in this way has bad singularities if the finite cover has order divisible by $p$ (\cite[Remark 0.3.2]{AdJ}).
Nonetheless, when $Z$ is a curve, we are able to construct $(\ast)$-modifications by blowing-up vertical centres.
This is one of the key reasons why we restrict our main results (\autoref{t-goal star}, \autoref{t-goal} and \autoref{t-ftrivial}) to the case of $Z$ being a curve.
We point out that, unfortunately, even assuming the existence of $(\ast)$-modifications in positive characteristic, the proof of \autoref{t-goal star} does not go through completely.
Indeed, when performing adjunction on a log canonical centre of a pair satisfying Property $(\ast)$, we need to control that the induced pair still satisfies Property $(\ast)$ (see \autoref{s-adjunction}) and it is not clear to the author if this property is satisfied in general.
    
\begin{lemma} \label{l-vertical_dlt_modification}
    Assume the birational dlt LMMP and the existence of log resolutions in dimension $n$.
    Let $f\colon X \rightarrow Z$ be a fibration from a normal projective variety $X$ of dimension $n$ onto a normal projective variety $Z$.
    Let $B \geq 0$ be a $\Q$-divisor such that $K_X+B$ is $\Q$-Cartier.
    Assume that $(X_{\eta}, B_{\eta})$ is $\Q$-factorial and dlt, where $\eta$ is the generic point of $Z$.
    Then, there exists a dlt modification $\mu\colon Y \to X$ such that:
    \begin{itemize}
        \item[(i)] the map $\mu$ is an isomorphism over $\eta$;
        \item[(ii)] $Y$ is $\Q$-factorial;
        \item[(iii)] if $C:= \mu_*^{-1}B + \sum_{E \subseteq \Exc(\mu)} E$, then $(Y, C)$ is dlt;
        \item[(iv)] there exists an effective $\mu$-exceptional $\Q$-divisor $R$ on $Y$, supported on non-log canonical places of $(X,B)$, such that
    \[
    K_Y+C+R=\mu^*(K_X+B);
    \]
    \item[(v)] $K_Y+C$ is $\mu$-nef.
    \end{itemize}
\end{lemma}
    
\begin{proof}
        First of all, note that, since $X_{\eta}$ is $\Q$-factorial and $(X_{\eta}, B_{\eta})$ is dlt, there exists a dense open subset $U \subseteq Z$ such that $(X_U, B|_{X_U})$ is dlt and $X_U$ is $\Q$-factorial, where $X_U:= f^{-1}(U)$.
        By the dlt property, there exists $V \subseteq X_U$ open dense subset such that $(V, B|_V)$ is log smooth and all the log canonical centres of $(X_U, B|_{X_U})$ are contained in $V$.
    Let $\sigma\colon X' \to X$ be a log resolution of $(X,B)$ that is an isomorphism on $V$ and write
    \[
    K_{X'}+B'+R'=\sigma^*(K_X+B)+E',
    \]
    where $B':= \sigma_*^{-1}B+ \sum_{E \subseteq \Exc(\sigma)} E$ and $R'$ and $E'$ do not have any common component.
    Note that, if $E$ is an irreducible exceptional divisor with discrepancy $>-1$, then it is in the support of $E'$.
    Now, run a $(K_{X'}+B')$-MMP over $X$ and let $Y$ be the resulting variety, so that it fits in the following diagram:
    \[
    \xymatrix{
    X' \ar@{-->}[r]^{\varphi} \ar[dr]_{h} & Y \ar[r]^{\mu} \ar[d]^{g} & X \ar[dl]^f \\
     & Z,
    }
    \]
    where $\varphi$ is the composition of the steps of the MMP, $\mu$ is proper birational and $g$ and $h$ are the induced fibrations.
    Let $C:= \varphi_*B'$ and $R:=\varphi_*R'$.
    Note that $\varphi_*E'=0$ by the Negativity lemma (see \cite[Lemma 3.39]{KM}).
    We claim that $Y \to X$ is an isomorphism over $X_U$.
    Indeed, it is an isomorphism over $V$ by construction. Furthermore, all the exceptional divisors whose centre is in $X_U \setminus V$ have discrepancy $>-1$, and therefore they are contracted by $\varphi$.
    Thus, $\mu^{-1}(X_U) \to X_U$ is a small birational map between $\Q$-factorial varieties, whence it is an isomorphism.
\end{proof}
    
\begin{theorem} \label{t-property star modifications v1.2}
    Assume the birational dlt LMMP and inversion of adjunction in dimension $n$, and the existence of log resolutions in dimensions $n$ and $n-1$.
    Let $f\colon X \rightarrow Z$ be a fibration from a normal projective variety $X$ of dimension $n$ onto a normal curve $Z$ and $(X/Z, B)$ a GGLC pair associated with it with $B \geq 0$.
    Let $\eta$ be the generic point of $Z$ and assume that $X_{\eta}$ is $\Q$-factorial and $(X_{\eta}, B_{\eta})$ is dlt.
    Then, there exists a GGLC pair $(Y/Z, C)$ satisfying Property $(\ast)$, with $Y$ $\Q$-factorial, $C \geq 0$, $\lfloor C \rfloor - \sum_{E \subseteq \Exc(\mu)}E \geq 0$, and $(Y, C)$ dlt, together with a commutative diagram:
    \[
    \xymatrix{
    Y \ar[r]^{\mu} \ar[dr]_{g} & X \ar[d]^f \\
     & Z,
    }
    \]
    where $\mu$ is a proper birational map which is an isomorphism over an open dense subset of $Z$ and $g$ is a fibration.
    Moreover, there exist a vertical effective exceptional $\Q$-divisor $R$, whose support consists of non-log canonical places of $(X,B)$, and a vertical effective $\Q$-divisor $G$, such that
    \[
    K_Y+ C + R = \mu^*(K_X+B) + G.
    \]
    The pair $(Y,C)$ is called a \textbf{$(\ast)$-modification} of $(X,B)$.
\end{theorem}
    
\begin{proof}
    \noindent \textbf{Step 1.}
    By \autoref{l-vertical_dlt_modification}, there exists a dlt modification $\rho\colon X' \to X$ which is an isomorphism over $\eta$. In particular, $X'$ is $\Q$-factorial and, if $B'$ is the sum of the strict transform of $B$ with all the reduced $\rho$-exceptional divisors:
    \[
    K_{X'}+B'+R'= \rho^*(K_X+B),
    \]
    where $R'$ is a vertical $\Q$-divisor supported on the non-log canonical places of $(X,B)$.
    
    \noindent \textbf{Step 2.}
    Now, we produce a birational model which is GGLC and satisfies Property~$(\ast)$.
    
    Let $h\colon X' \rightarrow Z$ and let 
    \[
    G':= \left( \sum_{z \in \Supp(B_Z)} h^{-1}(z) \right)- B'^v \geq 0,
    \]
    where $B_Z$ is the discriminant part of $(X'/Z,B')$.
    Finally, let $\widebar{B}:= B'+G'$.
    By \autoref{l-vertical_dlt_modification}, there exists $\sigma\colon Y \to X'$, a dlt modification of $(X', \widebar{B})$, which is an isomorphism over $\eta$.
    Define $C$ as the sum of the strict transform of $\widebar{B}$ with all the reduced $\sigma$-exceptional divisors.
    By construction, the discriminant part is $\Sigma_Z = \sum_{z \in \Supp(B_Z)} (z)$.
    Let $g\colon Y \to Z$ be the induced fibration and $\mu\colon Y \to X$ the induced birational map.
    
    We claim that $(Y/Z, C)$ satisfies Property~$(\ast)$.
    \begin{itemize}
    \item[(a)] The pair is GGLC because $\mu$ is an isomorphism over $\eta$;
    \item[(b)] since $Z$ is a curve, $(Z, \Sigma_Z)$ is log smooth; 
    \item[(c)] if $z \in Z \setminus \Supp(\Sigma_Z)$, $(Y, C + Y_z)$ is log canonical around $Y_z= g^{-1}(z)$ because $\sigma$ is an isomorphism outside $g^{-1}(B_Z)$;
    \item[(d)] if $z \in \Supp(\Sigma_Z)$, $(Y, C)$ is log canonical (around $g^{-1}(z)$) by construction.
    \end{itemize}
    Moreover,
    \[
    K_{Y}+C+R^*=\sigma^*(K_{X'}+\widebar{B})= \sigma^*(K_{X'}+B')+\sigma^*G',
    \]
    for an effective $\Q$-divisor $R^*$ supported on the non-log canonical places of $(X', \wb{B})$.
    
    Note that, since $(X',B')$ is log canonical, $\sigma^*G'-R^*$ is effective. 
    To conclude, let $R:=\sigma^*(R')$ and $G:=\sigma^*G'-R^*$,
    then:
    \[
    K_Y+C+R = \mu^*(K_X+B)+G,
    \]
    and $(Y/Z, C)$ satisfies the required properties.
\end{proof}

\begin{remark}
    Let $f \colon X \to Z$ be a fibration between normal projective varieties and let $\eta$ be the generic point of $Z$.
    Note that, if we blow-up a centre $\cC$ such that $\cC_{\wb{\eta}}$ is non-reduced, the resulting fibration may have highly singular fibres (see for example \cite[Proposition 5.3]{Jeff_Albanese}).
    In particular, if $(X/Z,B)$ is a GGLC pair, and $(Y,C)$ is a dlt modification of $(X,B)$, it may happen that $Y_{\wb{\eta}}$ is not normal.
    Therefore, in order to maintain the GGLC property, we have only performed birational operations that are isomorphisms over the generic point of $Z$.
\end{remark}

\subsection{Geometric \texorpdfstring{$(\ast)$}{}-modifications} \label{s-geometric star}

To do induction on the dimension, we will perform adjunction on log canonical centres of the geometric generic fibre. Therefore, we need to extract divisors with discrepancy $-1$ over them.
However, this cannot always be done over the total space $X$.
In the next examples, we see that it may happen that $X$ does not have any log canonical centre, even if the geometric generic fibre does.
To overcome this issue, we base change our fibration with a high enough power of the Frobenius morphism as in \autoref{construction} in \autoref{s-purelyinsep_basechange}: in this way the singularities of the geometric generic fibre appear on the total space.

In this section we use the notation of \autoref{construction} in \autoref{s-purelyinsep_basechange}.

\begin{example} 
Let $X$ be the projective closure of $V(x^3s+ y^3s + xyzs + v^3t)$ inside $\Pn^{3}_{[x:y:z:v]} \times \Pn^{1}_{[s:t]}$ and let $f\colon X \rightarrow \Pn^{1}_{[s:t]}$ be the fibration induced by the projection onto the second factor.
On the open set where $s \neq 0, v \neq 0$, $X$ is regular everywhere.
If the characteristic is $\neq 3$, the geometric generic fibre $X_{\widebar{\eta}}$ is smooth and only the closed fibre over $t=0$ has a singularity at the origin.
On the other hand, if the characteristic of the base field is $3$, a general fibre is a deformation of a cone and it has two singularities that come together over $t=0$.
On the geometric generic fibre they can be described by the equations $W:= V(x,z,y^3+t)$ and $W':= V(y,z,x^3+t)$.
These are canonical centres of the geometric generic fibre. 
Note that $f|_W$ and $f|_{W'}$ coincide with the Frobenius morphism.

Now, consider the base change with the Frobenius $F\colon Z \rightarrow Z$ and let $X^{(1)}$ be the normalisation of the fibre product.
Let $\tau$ be an element of $k(X^{(1)})$ such that $\tau^3=t$.
The total space $X^{(1)}$ is no longer regular, but has singularities at $W^{(1)}:= V(x,z,y+\tau)$ and $W'^{(1)}:= V(y, z, x+\tau)$.
\end{example}

\begin{example}
    Let $X:= \Pn^{2}_{[x:y:z]} \times \Pn^{1}_{[s:t]}$, $D=V(x^ps-y^pt) \subset X$ and $B:= \frac{1}{p}D$ over an algebraically closed field of characteristic $p>0$.
    Let $f \colon X \rightarrow \Pn^{1}_{[s:t]}$ be the natural projection.
    Consider the base change with $F \colon \Pn^{1} \rightarrow \Pn^{1}$, let $X^{(1)}$ be the normalisation of the fibre product and $\beta_1\colon X^{(1)} \rightarrow X$.
    Then $(X, B)$ is klt, while the geometric generic fibre $(X_{\wb{\eta}}, B_{\wb{\eta}})$ and $(X^{(1)}, \beta_1^*B)$ are strictly log canonical. 
\end{example}


\begin{lemma} \label{l-lc centres under finite maps}
Let $f\colon X \rightarrow Z$ be a separable fibration between normal varieties such that the geometric generic fibre $X_{\widebar{\eta}}$ is normal.
Let $B$ be a $\Q$-divisor on $X$ such that $K_X+B$ is $\Q$-Cartier.
Let $\sigma\colon Z' \rightarrow Z$ be a generically finite separable map and consider the diagram:
\[
\xymatrix{
X' \ar[r]^s \ar[d]_{f'} & X \ar[d]^f \\
Z' \ar[r]^{\sigma} & Z,
}
\]
where $X'$ is the normalisation of the main component of the fibre product $X \times_Z Z'$.
Define a $\Q$-divisor $B'$ on $X'$ by log pullback, so that $K_{X'}+ B'= s^*(K_X+B)$.
Let $\eta$ be the generic point of $Z$.
Then a horizontal subvariety $W \subseteq X$ is a log canonical centre (resp.\ non-log canonical centre) of $(X,B)$ if and only if there exists $W' \subseteq X'$, irreducible component of $s^{-1}(W)$, which is a log canonical centre (resp.\ non-log canonical centre) of $(X', B')$.
\end{lemma}

\begin{proof}
    By \autoref{l-conductor is vertical}, the conductor of $X'$ is vertical, so, up to shrinking $Z$, we can assume the fibre product is already normal.
    Moreover, up to shrinking $Z$ further, we can assume both $\sigma$ and its Galois closure are étale, so $s$ and its Galois closure are étale as well. In particular, they do not have any (wild) ramification.
    Thus, the discrepancies over $(X, B)$ and $(X', B')$ coincide (\cite[Proposition 5.20]{KM}). 
\end{proof}

\begin{proposition} \label{p-g lc centres}
Let $f\colon X \rightarrow Z$ be a separable fibration between normal varieties such that the geometric generic fibre $X_{\widebar{\eta}}$ is normal.
Let $B$ be a $\Q$-divisor on $X$ such that $K_X+B$, $B^v$ and all the vertical divisors $D$ whose multiplicity is $\ell_D>1$ are $\Q$-Cartier.
Let $\widebar{W}$ be a log canonical (resp.\ non-log canonical) centre of $(X_{\widebar{\eta}}, B_{\widebar{\eta}})$ and let $W$ be a subvariety of $X$ such that an irreducible component of $W_{\widebar{\eta},\red}$ is exactly $\widebar{W}$.
Then, there exists $e \gg 0$ such that $W^{(e)}$ is a log canonical (resp.\ non-log canonical) centre of $(X^{(e)}, B_e)$, where $B_e:=\beta_e^*B^h +\beta_e^{-1}B^v$.
\end{proposition}

\begin{proof}
Since $\widebar{W}$ is a log canonical (resp.\ non-log canonical) centre of $(X_{\widebar{\eta}}, B_{\widebar{\eta}})$, there exists $\widebar{Y} \rightarrow X_{\widebar{\eta}}$ proper birational map that extracts an exceptional divisor $\widebar{E}$ over $\widebar{W}$ with discrepancy $-1$ (resp.\ $<-1$).

By \autoref{l-spread-out}, there is a generically finite map $\varphi\colon Z' \rightarrow Z$ such that, if $X'$ is the normalisation of the main component of $X \times_Z Z'$, there exists $\sigma \colon Y \rightarrow X'$ birational with $Y_{\widebar{\eta}}= \widebar{Y}$.
By possibly taking a further purely inseparable base change and shrinking $Z$ if necessary, by \autoref{l-split extensions} the map $\varphi$ can be decomposed as $F^e \circ \psi$, with $\psi$ separable and $e \in \N$, so that we have the following diagram:
\[
\xymatrix{
Y \ar[r]^{\sigma} & X' \ar[r]^{\gamma} \ar[d] & X^{(e)} \ar[d] \ar[r]^{\beta_e} & X \ar[d]^f \\
& Z' \ar[r]^{\psi} & Z \ar[r]^{F^e} & Z.
}
\]
Let $B_e:= \beta_e^*B^h + \beta_e^{-1}B^v$ and define $B'$ by log pullback from $X^{(e)}$, so ${K_{X'}+B'}= \gamma^*(K_{X^{(e)}}+ B_e)$.
Note that, by \autoref{t-formulina_general} and the $\Q$-Cartier assumptions in the statement, $K_{X^{(e)}}+B_e$ is $\Q$-Cartier.
Since $X_{\widebar{\eta}}$ is normal, by \autoref{l-conductor is vertical}, $X'_{\widebar{\eta}}= X_{\widebar{\eta}} = X^{(e)}_{\widebar{\eta}}$, and $\beta_e$ and $\gamma$ are isomorphisms on the geometric generic fibre, therefore $B_{\widebar{\eta}}=B_{e,\widebar{\eta}}= B'_{\widebar{\eta}}$.

Let $E \subset Y$ be an exceptional divisor on $Y$ such that $E_{\widebar{\eta}}= \widebar{E}$. The discrepancy of $E$ with respect to the pair $(X',B')$ can be computed at $\widebar{\eta}$, therefore it is $-1$ (resp.\ $<-1$).

In the non-log canonical case, this proves that $\sigma(E)$ is a non-log canonical centre of $(X', B')$ and $(\beta_e \circ \gamma \circ \sigma)(E)=W$, therefore we conclude by \autoref{l-lc centres under finite maps}.

In the log canonical case, we have verified that $E$ is a place over $\sigma(E)$ with discrepancy $-1$ with respect to $(X',B')$ and such that $(\beta_e \circ \gamma \circ \sigma)(E)=W$.
Now, let $b\colon Y' \to X'$ be a proper birational map with an exceptional divisor $E' \subseteq Y'$ such that $(\beta_e \circ \gamma \circ b)(E')=W$. We prove that the discrepancy of $E'$ over $(X', B')$ is $\geq -1$.
Indeed, let $C'$ be the strict transform of $B'$ in $Y'$ and let $a$ be the discrepancy of $E'$. Then,
\[
K_{Y'}+C'=b^*(K_{X'}+B')+aE'+R,
\]
where $R$ is an exceptional $\Q$-divisor whose support does not contain $E'$.
We now consider the restriction of the equality to $(Y'_{\wb{\eta}})^{\nu}$.
Let $\wb{b} \colon (Y'_{\wb{\eta}})^{\nu} \to X'_{\wb{\eta}}=X_{\wb{\eta}}$
By \cite[Theorem 1.1]{LMMP} (see also \cite[Theorem 4.2]{pi_basechange}), there exists an effective divisor $D$ such that $K_{(Y'_{\wb{\eta}})^{\nu}}+D+C'_{\wb{\eta}}=\wb{b}^*(K_{X_{\wb{\eta}}}+B_{\wb{\eta}})+aE'_{\wb{\eta}}+R_{\wb{\eta}}$.
Let $d \geq 0$ be the coefficient of $E'_{\wb{\eta},\red}$ in $D$.
Since $\wb{W}$ is a log canonical centre of $(X_{\wb{\eta}}, B_{\wb{\eta}})$, we have that $ac-d \geq -1$ for some $c >0$, whence $a \geq -1$.
We conclude that $W^{(e)}$ is a log canonical centre of $(X^{(e)}, B_e)$ by \autoref{l-lc centres under finite maps}.
\end{proof}

\begin{definition} \label{d-cycleG}
    Let $X$ be a variety defined over a perfect field $k$ and let $\widebar{k}$ be the algebraic closure of $k$. Define $\widebar{X}:= X \times_k \widebar{k}$.
    If $W \subseteq \widebar{X}$ is a subvariety, let $k \subseteq k' \subseteq \wb{k}$ be a (separable) Galois finite extension of $k$ over which $W$ is defined.
    Let $G:= \mathrm{Gal}(\widebar{k}/k)$ and $G':= \mathrm{Gal}(k'/k)$.
    Then
    $
    \sum_{g \in G'} g(W)
    $
    descends to a cycle $\mathcal{C}$ on $X$ defined over $k$.
    Let $\mathcal{C}:= \sum a_i C_i$ be the decomposition of $\mathcal{C}$ into irreducible components defined over $k$, then we define $W^G:= \sum C_i$.
    Note that an integral component of $W^G \times_k \widebar{k}$ is exactly $W$.
\end{definition}

\begin{proposition} \label{l-geometric blowup}
    Assume the existence of log resolutions in dimension up to $n-1$. 
    Let $f\colon X \rightarrow Z$ be a separable fibration between normal projective varieties, where $\dim(X)=n$ and let $B \geq 0$ be a $\Q$-divisor on $X$ such that $K_X+B$ is $\Q$-Cartier.
    Let $\eta$ be the generic point of $Z$.
    Then, for $e \gg 0$, there exists a birational map $Y \to X^{(e)}$ such that $Y_{\widebar{\eta}}$ is a log resolution of $(X_{\wb{\eta}}, B_{\wb{\eta}})$.
    In particular, the strict transform of $B_{\wb{\eta}}$ together with the exceptional divisors is simple normal crossing.
\end{proposition}
    
\begin{proof}
    Since we assume the existence of log resolutions, there exists a sequence of blow-ups $$\wb{X}_N \to ... \to \wb{X}_1 \to \wb{X}_0=X_{\wb{\eta}}$$ at smooth centres $Z_j \subseteq \wb{X}_j$ such that $\wb{X}_N$ is a log resolution of $(X_{\wb{\eta}}, B_{\wb{\eta}})$. 
    Let $L$ be a finite extension of $k(Z)$ over which all the $Z_j$'s are defined.
    By \autoref{l-split extensions}, up to considering a further purely inseparable extension, we can suppose there exists $e \geq 0$ such that $L$ is a separable Galois extension of $K_e := k(Z)^{\frac{1}{p^e}}$. Let $G$ be the Galois group of $L$ over $K_e$.
    
    Let $Z_0^G$ as in \autoref{d-cycleG}. In particular, we can view $Z_0^G$ as a geometrically reduced cycle inside $X^{(e)}_{\eta}$ and an irreducible component of $Z_0^G \times_{\eta} \wb{\eta}$ is exactly $Z_0$.
    By ``spreading-out techniques'' (see \cite[Proof of Corollary 1.10]{DW} and \cite[Lemma 2.25]{invariance}), there exists $W_0 \subseteq X^{(e)}$ such that $W_{0, \eta}=Z_0^G$.
    Let $X_1$ be the blow-up of $X^{(e)}$ at $W_0$.
    Note that there is a map $X_{1, \wb{\eta}} \to \wb{X}_1$ that can be decomposed over $\wb{\eta}$ as a sequence of blow-ups at smooth centres.
    
    Let $Z'_1$ be the strict transform of $Z_1$ inside $X_{1, \wb{\eta}}$. We repeat the above process: we define $Z_1^{'G} \subseteq X_{1, \eta}$ and $W_1 \subseteq X_1$ a spread-out of $Z_1^{'G}$ and we define $X_2$ as the blow-up of $X_1$ at $W_1$.
    We proceed in the same way for all $j=0,...,N-1$.
    Let $Y:= X_N \to X^{(e)}$ be the resulting variety. Note that $Y_{\wb{\eta}} \to \wb{X}_N$ can be decomposed as a sequence of blow-ups at smooth centres, whence the conclusion.
\end{proof}

\begin{theorem} \label{t-property star modifications v2.0}
    Assume the birational dlt LMMP and inversion of adjunction in dimension $n$, and the existence of log resolutions in dimensions $n$ and $n-1$.
    Let $f\colon X \rightarrow Z$ be a fibration from a normal variety $X$ of dimension $n$ onto a normal curve $Z$, such that the geometric generic fibre $X_{\widebar{\eta}}$ is normal.
    Let $B \geq 0$ be a $\Q$-divisor on $X$ such that $K_X+B$, $B^v$ and all the vertical divisors $D$ whose multiplicity is $\ell_D>1$ are $\Q$-Cartier. Assume that the coefficients of $B$ and $B_{\widebar{\eta}}$ are $\leq 1$. 
    Let $\widebar{W}$ be a log canonical centre of $(X_{\widebar{\eta}}, B_{\widebar{\eta}})$ and let $W$ be a subvariety of $X$ such that an irreducible component of $W_{\widebar{\eta},\red}$ is exactly $\widebar{W}$.
    For $e \geq 0$, let $B_e:= \beta_e^*B^h+\beta_e^{-1}B^v$.
    Then, for all $e \gg 0$, there exists a $(\ast)$-modification of $(X^{(e)}, B_e)$ which extracts an exceptional divisor $\tilde{E}$ over $W^{(e)}$ with discrepancy $-1$.
    More precisely, there exist a dlt GGLC pair $(Y/Z, C+\tilde{E})$ with $Y$ $\Q$-factorial, $C \geq 0$, $\lfloor C+\tilde{E} \rfloor - \sum_{E \subseteq \Exc(\mu)} E \geq 0$, and a diagram:
    \[
    \xymatrix{
    Y \ar[r]^{\mu} \ar[dr]_{g} & X^{(e)} \ar[d]^{f_e} \\
     & Z,
    }
    \]
    where $\mu$ is a proper birational map which extracts $\tilde{E}$ over $W^{(e)}$ and $g$ is a fibration.
    Moreover, there exist an effective exceptional $\Q$-divisor $R$ on $Y$ supported on non-log canonical places of $(X^{(e)},B_e)$, and a vertical effective $\Q$-divisor $G$ on $Y$, such that
    \[
    K_Y+ C +\tilde{E} + R = \mu^*(K_{X^{(e)}}+B_e) + G.
    \]
    Additionally, $(K_Y+C+\tilde{E})$ is $\mu$-nef.
\end{theorem}

\begin{proof}
    \noindent \textbf{Step 1.} First of all note that $(X_{\widebar{\eta}}, B_{\widebar{\eta}})$ is not necessarily log canonical. In this first step we find a model over $(X^{(e)}, B_e)$ that satisfies the GGLC property.

    By \autoref{p-g lc centres}, for $e \gg 0$, $W^{(e)}$ is a log canonical centre of $(X^{(e)}, B_e)$.
    By \autoref{l-geometric blowup}, up to choosing a bigger $e \geq 0$, there exists $\sigma\colon X_1 \rightarrow X^{(e)}$ which is a log resolution at the geometric generic point of $Z$.
    By possibly considering a higher model obtained as a sequence of blow-ups at vertical smooth centres, we can suppose $X_1$ is a log resolution of $(X^{(e)}, B_e)$.
    Let $B_1 := \sigma^{-1}_*B_e+ \sum_{E \subseteq \Exc(\sigma)} E$ and write
    \[
    K_{X_1}+B_1+R_1= \sigma^*(K_{X^{(e)}}+B_e)+T_1,
    \]
    where both $R_1$ and $T_1$ are effective, $\sigma$-exceptional and they do not have any common components.    
    Note that $R_1$ is supported on non-log canonical places of $(X^{(e)}, B_e)$. Moreover, $(X_1, B_1)$ and $(X_{1,\widebar{\eta}}, B_{1,\widebar{\eta}})$ are log smooth.
    
    Now, we run a $(K_{X_1}+B_1)$-MMP over $X^{(e)}$. Note that this is also an MMP over $Z$. Let $\varphi\colon X_1 \dashrightarrow X'_1$ and $\tau\colon X'_1 \rightarrow X^{(e)}$ be the resulting morphisms, $B'_1:= \varphi_*B_1$ and $R'_1:=\varphi_*R_1$. Since $T_1$ and $R_1$ do not share any component, by the Negativity lemma (see \cite[Lemma 3.39]{KM}), $\varphi_*T_1=0$, so that:
    \[
    K_{X'_1}+B'_1+R'_1 = \tau^*(K_{X^{(e)}}+B_e),
    \]
    where $R'_1$ is $\tau$-exceptional and supported on non-log canonical places of $(X^{(e)}, B_e)$.
    By \autoref{l-normal_afterMMP} $X'_{1, \wb{\eta}}$ is normal.
    Moreover, both $(X'_1, B'_1)$ and $(X'_{1,\widebar{\eta}}, B'_{1,\widebar{\eta}})$ are dlt and $\Q$-factorial by \cite[proof of Corollary 3.44]{KM}.
    
    \noindent \textbf{Step 2.} In this step, we extract a log canonical place over $W^{(e)}$.

 Note that the strict transform of $W^{(e)}$ in $X'_1$ must be a stratum of the boundary. Since $(X'_{1,\widebar{\eta}}, B'_{1,\widebar{\eta}})$ is dlt, all the horizontal strata are geometrically reduced over $\eta$.
    Therefore, it is possible to extract a log canonical place $E_2$ as a sequence of blow-ups at smooth centres which are also geometrically smooth over $\eta$.
    
    Let $s\colon X_2 \rightarrow X'_1$ be the resulting birational morphism extracting a log canonical place $E_2$ over $W^{(e)}$.
    Define $B_2:=s^{-1}_*B'_1 +\sum_{E \subseteq \Exc(s)}E-E_2$ and $T_2$ effective $\Q$-divisors so that $$K_{X_2}+B_2+E_2=s^*(K_{X'_1}+B'_1)+T_2.$$
    Note that $X_2$ is $\Q$-factorial and, by construction, $(X_{2, \wb{\eta}}, B_{2, \wb{\eta}}+E_{2, \wb{\eta}})$ and $(X_2, B_2+E_2)$ are dlt.
    Run a $(K_{X_2}+B_2+E_2)$-MMP over $X^{(e)}$. Let $\psi\colon X_2 \dashrightarrow X'_2$ and $t \colon X'_2 \rightarrow X^{(e)}$ be the resulting morphisms, $B'_2:= \psi_*B_2$ and $E'_2:=\psi_*E_2$. Note that 
    all the contracted (or flipped) curves are in the support of $T_2$ because $K_{X_1'}+B_1'$ is $\tau$-nef; in particular, $E_2$ is not contracted.
    By the Negativity lemma (see \cite[Lemma 3.39]{KM}), $\psi_*T_2=0$, so that:
    \[
    K_{X'_2}+B'_2+E'_2 + R'_2 = t^*(K_{X^{(e)}}+B_e),
    \]
    where $R'_2:=\psi_*s^*R'_1$ is $t$-exceptional and supported on non-log canonical places of $(X^{(e)}, B_e)$.
    By \autoref{l-normal_afterMMP}, $X'_{2, \wb{\eta}}$ is normal.
    Moreover, both $(X'_2, B'_2+E'_2)$ and $(X'_{2,\widebar{\eta}}, B'_{2,\widebar{\eta}}+E'_{2,\widebar{\eta}})$ are dlt 
    and $\Q$-factorial and $K_{X'_2}+B'_2+E'_2$ is $t$-nef.
    
    \noindent \textbf{Step 3.} In this step, we find a $(\ast)$-modification of $(X'_2, B'_2+E'_2)$.
    
    We apply \autoref{t-property star modifications v1.2} to $(X'_2, B'_2+E'_2)$ to find a $(\ast)$-modification ${(Y'/Z, C'+E')}$ with $\lambda\colon Y' \rightarrow X'_2$ birational, where $E'$ is the strict transform of $E'_2$. Let $g'\colon Y' \to Z$ be the induced fibration.
    In particular,
    \begin{enumerate}
        \item $\lambda$ is an isomorphism over an open subset of $Z$;
        \item $(Y'/Z, C'+E')$ is GGLC and satisfies Property $(\ast)$;
        \item $Y'$ is $\Q$-factorial, $(Y', C'+E')$ is dlt, and $\lfloor C'+E'\rfloor - \sum_{E \subseteq \Exc(\mu')} E$;
        \item there exists an effective $\Q$-divisor $G'$, such that $K_{Y'}+ C'+E' = {\lambda^*(K_{X'_2}+B'_2+E'_2) + G'}$ and $G'$ is $g'$-vertical;
        \item if $U:=Y' \setminus \supp(G')$, $(K_{Y'}+C'+E')|_{U}$ is $\mu'$-nef, where $\mu'\colon Y \rightarrow X^{(e)}$;
        \item let $R':=\lambda^*R'_2$, then $K_{Y'}+C'+E'+R'=\mu'^*(K_{X^{(e)}}+B_e)+G'$.
    \end{enumerate}
    
\noindent\textbf{Step 4.} As last step, we run a $(K_{Y'}+C'+E')$-MMP over $X^{(e)}$.
    
    By point (e) in Step 3, this MMP is an isomorphism on $U$, whence it does not contract $E'$.
    Let $Y$ be the resulting variety, $\mu\colon Y \to X^{(e)}$ the induced morphism, and $C, \tilde{E}, R$ and $G$ the push-forward on $Y$ of $C',E', R'$ and $G'$, respectively. 
    Let $g \colon Y \to Z$ be the induced morphism.
    Then:
    \begin{enumerate}
        \item $(Y/Z, C+\tilde{E})$ is GGLC and satisfies Property $(\ast)$ by \autoref{p-properties property star};
        \item $Y$ is $\Q$-factorial, $(Y, C+\tilde{E})$ is dlt, and $\lfloor C+\tilde{E}\rfloor - \sum_{E \subseteq \Exc(\mu)} E$;
        \item $K_{Y}+ C+\tilde{E}+R = \mu^*(K_{X^{(e)}}+B_e)+G$;
        \item $K_{Y}+C+\tilde{E}$ is $\mu$-nef.
    \end{enumerate}
\end{proof}

\subsection{Adjunction of the moduli part} \label{s-adjunction}

Now, we study the restriction of the moduli part to geometric log canonical centres for fibrations onto curves.

\begin{lemma} \label{l-Snormal}
    Let $(X, B+S)$ be a dlt pair over $k$ with $B \geq 0$, where $S$ is a prime divisor and $X$ is a normal $\Q$-factorial variety.
    Assume $\mathrm{char}(k)= p>2$.
    Then, the normalisation morphism $S^{\nu} \rightarrow S$ is an isomorphism in {codimension $1$}.
    Moreover, if either
    \begin{itemize}
        \item[(a)] $X$ has dimension $\leq 3$ and $\mathrm{char}(k)=p>5$, or
        \item[(b)] $S$ satisfies the $S_2$ property,
    \end{itemize}
    then $S$ is normal.
\end{lemma}

\begin{proof}
    The pair $(S^{\nu}, B_{S^{\nu}})$ induced on $S^{\nu}$ by adjunction is log canonical.
    Moreover, by \cite[Lemma 2.1]{mixedMMP}, $S^{\nu} \rightarrow S$ is a universal homeomorphism, therefore the singularities in codimension $1$ are worse than nodal for $p>2$ by \cite[Proposition 3.1.12]{nodes}.
    Thus, if $S^{\nu} \rightarrow S$ was not an isomorphism in codimension $1$, the conductor would have coefficients $>1$, leading to a contradiction.

    If $X$ is a threefold over an algebraically closed field of characteristic $p>5$, the claim is proven in \cite[Lemma 5.2]{Bir}. If $X$ is defined over a perfect field $k$ of characteristic $p>5$, by the same argument, we conclude that $\widebar{S} := S \times_{k} \widebar{k}$ is normal, which implies that $S$ is normal by \cite[Tag 0C3M]{stacks}. 
    On the other hand, if $S$ satisfies the $S_2$ property, the pair $(S, B_S)$ induced on $S$ by adjunction is slc and regular in codimension $1$, whence the conclusion. 
\end{proof}

\begin{remark}
If $X$ has dimension $>3$, or the characteristic of the base field is $p \leq 5$, then it is not true in general that plt centres are normal (see \cite{plt_centers_Fabio, plt_centers_PaoloTanaka}).
\end{remark}

We now set some notation that we will use in the following.
\leqnomode
\begin{flalign} \label{setup}
    \tag{\raisebox{-0.5ex}{\FourStarOpen}} \hspace{7mm} \text{\underline{Set-up}} &&
\end{flalign}
\reqnomode

\begin{itemize}
\item[(o)] Assume inversion of adjunction in dimensions $n$ and $n-1$ and the existence of log resolutions in dimensions $n-1$ and $n-2$. Assume that $\mathrm{char}(k)=p>2$.
\item[(i)] Let $f\colon X \rightarrow Z$ be an equidimensional fibration between normal varieties, where $n:=\dim(X)$ and $Z$ is a curve. Let $(X/Z, B+S)$ be a GGLC pair associated with it, where $S$ is a prime horizontal divisor and $B \geq 0$.
\item[(ii)] Assume that $X$ is $\Q$-factorial and $(X, B+S)$ is dlt. Then, by \autoref{l-Snormal}, $S$ is normal in codimension $1$.
\item[(iii)] Let $(S^{\nu}, B_{S^{\nu}})$ be the pair induced on $S^{\nu}$ by adjunction.
\item[(iv)] Suppose that $(X/Z, B+S)$ satisfies Property $(\ast)$.
\item[(v)] Let $f|_{S^{\nu}}= g \circ \varphi$ be the Stein factorisation of $f|_{S^{\nu}}$; in particular $g$ is a fibration and $\varphi\colon Z' \rightarrow Z$ is a finite morphism.
\end{itemize}

\begin{remark} \label{r-multiplicities_adjunction}
    In the above Set-up \autoref{setup}, since $(X/Z, B+S)$ is GGLC, if $\widebar{\eta}$ is the geometric generic point of $Z$, $S_{\widebar{\eta}}$ is reduced, therefore, by \autoref{p-reducedness}, $\varphi$ is separable.
Moreover, we have some control over the ramification of $\varphi$.
Indeed, by point (ii) in Set-up \autoref{setup}, without loss of generality, we can suppose $S=S^{\nu}$. Let $\delta \subseteq Z$ be a prime divisor. Since $(X/Z, B+S)$ has Property $(\ast)$, $(X, S+ f^{-1}(\delta))$ is log canonical around $f^{-1}(\delta)$, whence the induced pair on $S$ has coefficients $\leq 1$ around $f^{-1}(\delta)|_{S}$ by adjunction. In particular, $f^{-1}(\delta)|_{S}$ is reduced.
Let $D$ be a prime vertical divisor in $X$ and $D_S$ a prime component of $D|_S$.
If we work around a general point of $D_S$ and we let $\delta':= g(D_S)$ and $\delta:= f(D)$, we can write $\varphi^*\delta= m \delta'$, $g^*\delta'=n D_S$ and $f^*\delta= \ell D$. Then, $\ell = mn$.

Recall that, if $\delta' \subseteq Z'$ is a prime divisor such that its multiplicity with respect to $\varphi$ is divisible by $p$, $\delta'$ is called a divisor of \textbf{wild ramification}. If the multiplicity is $\geq 2$ and coprime with $p$, $\delta'$ is a divisor of \textbf{tame ramification}.
If $D$ is a wild component for $f$, then $\delta':= g(D)$ may be a divisor of wild ramification for $\varphi$. On the other hand, if $D$ is tame, then $\delta'$ is either unramified or of tame ramification.
\end{remark}

\begin{proposition} \label{p-property star adjunction}
In Set-up \autoref{setup}, the pair $(S^{\nu}/Z', B_{S^{\nu}})$ associated to the fibration $g\colon S^{\nu} \to Z'$ is GGLC and satisfies Property $(\ast)$.
Moreover, the vertical parts satisfy $(B_{S^{\nu}})^v= B^v|_{S^{\nu}}$.
\end{proposition}

\begin{proof}
First of all, note that the geometric generic points of $Z$ and $Z'$ coincide and the pair $(S_{\widebar{\eta}}^{\nu}, B_{S_{\widebar{\eta}}^{\nu}})$ is log canonical by adjunction.
By the universal property of the normalisation $(S^{\nu})_{\widebar{\eta}}^{\nu} = S_{\widebar{\eta}}^{\nu}$, and by \autoref{l-GGLC normal fibre} $S_{\widebar{\eta}}^{\nu}= (S^{\nu})_{\wb{\eta}}$ is normal.
Thus, $(S^{\nu}/Z', B_{S^{\nu}})$ is GGLC.

Since $S^{\nu} \rightarrow S$ is an isomorphism in codimension $1$, for the purpose of the proof, we can assume $S$ is normal.
Let us verify that $(S/Z', B_{S})$ satisfies Property $(\ast)$. Let $\Sigma_Z$ be the discriminant part of $(X/Z, B+S)$.
Since $X$ is $\Q$-factorial, we define $B^h_S$ by adjunction, so that $(K_X+B^h+S)|_S=K_S+B^h_S$ and $K_S+B_S=K_S+B^h_S+ B^v|_S$.

\begin{itemize}
    \item[(a)] We claim that $B^h_S$ does not contain any vertical component. In particular, $(B_S)^v=B^v|_S$ and $(B_S)^h=B^h_S$.
    Indeed, if this was not the case, let $D_S$ be a vertical divisor contained in $\Supp(B^h_S)$ and $z := f|_S(D_S)$.
    Since $(X, B^h+S+f^{-1}(z))$ is log canonical around $z$, $(S, B^h_S+D_S)$ is log canonical around $z$ as well. Therefore, the coefficient of $D_S$ in $B^h_S$ must be $0$, which is a contradiction.
    To conclude, note that $(B^v)|_S \leq (B_S)^v$ since $S$ is horizontal.
    
    \item[(b)] Let $\Sigma_{Z'}:= \varphi^{-1}(\Sigma_Z)$. Then $B_S^v = f|_S^{-1}(\Sigma_{Z})= g^{-1}(\Sigma_{Z'})$ and, since $Z'$ is a normal curve, $(Z', \Sigma_{Z'})$ is log smooth.
    
    \item[(c)] Let $z' \in Z' \setminus \Supp(\Sigma_{Z'})$ be a closed point and $z:= \varphi(z') \in Z \setminus \Supp(\Sigma_Z)$. In particular, $f$ is unramified at $z$, therefore, by the computations in \autoref{r-multiplicities_adjunction}, both $\varphi$ and $g$ are unramified at $z'$.
    Since $(X, B +S +f^*z)$ is log canonical around $f^{-1}(z)$, $(S, B_S+f|_S^*z)= (S, B_S +g^*z' +D)$ is log canonical around $g^{-1}(z')$ by adjunction, where $f|_S^*z = {g^*z'+D}$. \qedhere
\end{itemize}
\end{proof}

\begin{proposition} \label{p-adjunction moduli}
In Set-up \autoref{setup}, define $M_X$ and $M_{S^{\nu}}$ to be the moduli parts of $(X/Z, B+S)$ and $(S^{\nu}/Z', B_{S^{\nu}})$, respectively. Then,
\[
    M_{S^{\nu}}= M_X|_{S^{\nu}} - \sum_{D_S} v_{D_S} D_S,
\]
where the sum is taken over the divisors $D_S$ that lie over points of wild ramification of $\varphi$ and the coefficients $v_{D_S}$ are strictly positive integers.
\end{proposition}

\begin{proof}
For the sake of the proof, we assume $S=S^{\nu}$.
By \autoref{p-property star adjunction}, $(S/Z', B_S)$ satisfies {Property $(\ast)$} and $(B_S)^v=B^v|_S$, therefore, by \autoref{p-M_X for property star},
\[
M_X = K_X +B^h -f^*K_Z - R(f) \quad \text{and} \quad
M_S = K_S + B_S^h -g^*K_{Z'} -R(g),
\]
where $B^h_S$ satisfies $(K_X+B^h)|_S=K_S+B_S^h$.
By the Hurwitz formula (see \cite[Proposition 2.3, Chapter IV]{H}),
\[
    K_{Z'}= \varphi^*K_Z + \sum_i (m_i -1)z'_i + \sum_j (a_j+1)z'_j,
\]
where the first sum is taken over all ramified point of $\varphi$, with $m_i$ being their multiplicity, and the second sum only over the wildly ramified points of $\varphi$, with $a_j$ being a non-negative integer.
Now, let us do the computations locally around a vertical prime divisor $D_S \subseteq S$.
Let $D \subseteq X$ be a vertical prime divisor such that $D_S \leq D|_S$, $z:=f(D)$ and $z':=g(D_S)$. We can suppose $f^*z= \ell D$, $\varphi^*z=mz'$, $g^*z'=nD_S$ and $\ell=mn$.
If $\varphi$ is tamely ramified at $z'$, locally we have:
\begin{align*}
 M_S & = K_S + B_S^h -g^*K_{Z'} -(n-1)D_S \\
 & = K_S+B_S^h-f|_S^*K_Z-(m-1)g^*z'-(n-1)D_S\\
 & = (K_X+B^h-f^*K_Z -(mn-1)D)|_S= M_X|_S.
\end{align*}
If $\varphi$ is wildly ramified at $z'$, locally we have:
\begin{align*}
 M_S & = K_S + B_S^h -g^*K_{Z'} -(n-1)D_S\\
 & = K_S+B_S^h-f|_S^*K_Z-g^*(m+a)z'-(n-1)D_S\\
 & = (K_X-f^*K_Z -(mn-1)D)|_S -(an+n)D_S= M_X|_S - (an+n)D_S,
\end{align*}
for some $a\geq 0$. Set $v_{D_S}:= an+n$ to conclude.
\end{proof}

\begin{remark}
    In Set-up \autoref{setup}, let $\Sigma_Z$ be the discriminant part of $(X/Z, B+S)$.
    Define $\Sigma_{Z'}:= \varphi^{-1}(\Sigma_Z)$.
    If $\dim(Z) \geq 2$, then $(Z', \Sigma_{Z'})$ may not even be log canonical, due to the possible presence of wild ramification in $\varphi$. Therefore, 
    we cannot conclude a formula similar to the one in \autoref{p-adjunction moduli} for adjunction of the moduli part.
\end{remark}

\section{The Bend and Break theorem for the moduli divisor} \label{s-B&B}

A crucial step in the proof of positivity of the moduli part in characteristic $0$ is the Cone theorem for algebraically integrable foliations \cite[Theorem 3.9]{positivity}: if the canonical bundle of a foliation is not nef, we can construct rational curves that are \emph{tangent to the foliation}.
The idea is to find them by applying {Miyaoka--Mori's} Bend and Break theorem \cite{MM} to the variety reduced modulo a big enough prime (see \cite{SB}). 
If our setting is in positive characteristic to start with, we do not have the possibility of changing the prime.
However, for fibrations under Property $(\ast)$ conditions, we can relate the moduli part to the canonical divisor of the variety obtained after a Frobenius base change using \autoref{t-formulina_general}.
This is a key ingredient for our strategy and it constitutes one of the main differences between the situation in characteristic $0$ and over fields of positive characteristic.
Note that our result \autoref{p-non-big B vertical} is not a Cone theorem for foliations in positive characteristic. In fact, this latter is known not to hold (\cite{counterexamples_Fabio}).

In this section we use the notation of \autoref{construction} in \autoref{s-purelyinsep_basechange}.

We will use the following version of the Bend and Break theorem.

\begin{theorem}[{Bend and Break theorem, \cite[Theorem II.5.8]{B&B}}] \label{t-classicB&B}
Let $X$ be a normal projective variety of dimension $n$ over an algebraically closed field of any characteristic and let $\xi$ be a smooth curve such that $X$ is smooth around $\xi$. Let $x \in \xi$ be a general point and $G$ a nef divisor on $X$.
If $K_X \cdot \xi <0$, there exists a rational curve $\zeta_x$ such that $x \in \zeta_x$ and
\[
    G \cdot \zeta_x \leq 2n \frac{G\cdot \xi}{-K_X \cdot \xi}.
\]
\end{theorem}

The next result is a direct consequence of the Bend and Break theorem; the proof follows \cite[Theorem 6.1]{KMM94} and \cite[Corollary 2.28]{Foliations-Calum}.

\begin{corollary} \label{t-B&B}
Let $X$ be a normal projective variety of dimension $n$ over an algebraically closed field of any characteristic and let $D_1, ..., D_n, G$ be nef divisors on $X$.
Assume that 
\begin{itemize}
    \item[(a)] $D_1 \cdot ... \cdot D_n=0$;
    \item[(b)] $K_X \cdot D_2 \cdot ... \cdot D_n <0$.
\end{itemize}
Then, for $x \in X$ general point, there is a rational curve $\zeta_{x} \subseteq X$ containing $x$ such that
\begin{itemize}
    \item[(i)]$ G \cdot \zeta_x \leq 2n \frac{G \cdot D_2 \cdot ... \cdot D_n}{-K_X \cdot D_2 \cdot ... \cdot D_n}$;
    \item[(ii)] $D_1 \cdot \zeta_x =0$.
\end{itemize}
\end{corollary}

\begin{proof} 
Let $H$ be an ample divisor and $0 \leq \varepsilon \ll 1$ small enough such that, if $H_i:=D_i + \varepsilon H$ for $i=2, ..., n$, we have that $K_X \cdot H_2 \cdot ... \cdot H_n <0$.
Now, choose $m_i \in \N$ such that $m_iH_i$ are very ample for $i=2, ..., n$ and let $\xi$ be a curve in the intersection of the linear systems $|m_iH_i|$, passing through a general point of $X$ and such that $\xi$ is smooth and contained in the regular locus of $X$.
Let $G_{\varepsilon, k}:= kD_1+G+\varepsilon H$.
By \autoref{t-classicB&B}, there exists $\zeta_{\varepsilon ,k}$ through a general point of $\xi$ such that
\begin{equation} \leqnomode \tag{\raisebox{-0.5ex}{\EightStarTaper}} \label{e-equation B&B}
G_{\varepsilon, k} \cdot \zeta_{\varepsilon,k} \leq 2n \frac{G_{\varepsilon, k} \cdot \xi}{-K_X \cdot \xi}
= 2n \frac{G_{\varepsilon,k}\cdot H_2 \cdot ... \cdot H_n}{-K_X \cdot H_2 \cdot ... \cdot H_n}.
\end{equation}

Since $D_1 \cdot D_2 \cdot ... \cdot D_n=0$, there exists a constant $c \geq 0$ with the following property. For every $k>0$, there exists $\varepsilon_k >0$ such that the RHS of the equation \autoref{e-equation B&B} is $\leq c$ for every $\varepsilon \leq \varepsilon_k$.
Therefore the family $\{ \zeta_{\varepsilon_k, k} \}_{k}$ is bounded.
Since bounded integral points in the cone of curves are finitely many, up to passing to a sub-sequence, we can assume $\zeta_{\varepsilon_k, k} = \zeta$ is constant.

Now, letting $k$ approach $+ \infty$, we get that $(kD_1 +G+\varepsilon_k H) \cdot \zeta$ is bounded, whence $D_1 \cdot \zeta=0$.
Therefore, \autoref{e-equation B&B} becomes
\[
(G+\varepsilon H) \cdot \zeta \leq 2n \frac{(kD_1+ G +\varepsilon H) \cdot H_2 \cdot ... \cdot H_n}{-K_X \cdot H_2 \cdot ...\cdot H_n}.
\]
We conclude by letting $\varepsilon$ go to $0$.
\end{proof}

The next proposition is one of the key steps for the proof of the main result, \autoref{t-goal}. 

\begin{proposition} \label{p-non-big B vertical}
    Assume inversion of adjunction in dimension $n$ and the existence of log resolutions in dimension $n-1$.
    Let $f\colon X \rightarrow Z$ be a fibration between normal projective varieties and $(X/Z, B)$ a GGLC pair of dimension $n$ associated with it over an algebraically closed field of characteristic $p>0$, with $B \geq 0$.
    Assume that $(X/Z, B)$ satisfies Property $(\ast)$ with $X$ $\Q$-factorial and let $M_X$ be the moduli part.
    Assume that there exist $D_2, ..., D_n$ nef $\Q$-divisors, $k \in \Q_{>0}$ and $A$ ample $\Q$-divisor such that:
    \begin{itemize}
    \item[(a)] $(M_X+kA) \cdot D_2 \cdot ... \cdot D_n =0$;
    \item[(b)] $M_X \cdot D_2 \cdot ... \cdot D_n <0$; 
    \item[(c)] $(M_X+kA)$ is nef.
    \end{itemize}
    Then, through a general point $x \in X$, there exists a vertical rational curve $\zeta$ such that $M_X \cdot \zeta<0$.
\end{proposition}
    
\begin{proof}
    By \autoref{r-formulina},
    \[
    \alpha_e^*K_{X^{(e)}} \cdot D_2 \cdot ...\cdot D_n = (p^e-1) (M_X-B^h) \cdot D_2 \cdot ...\cdot D_n + K_X \cdot D_2 \cdot ...\cdot D_n - W_e \cdot D_2 \cdot ... \cdot D_n,
    \]
    where $W_e:=\sum_{D \text{ wild}} w_{D,e}D$ with $w_{D,e} \geq 0$.

    
    Note that $B^h\cdot D_2 \cdot ... \cdot D_n \geq 0$ and $D\cdot D_2 \cdot ... \cdot D_n \geq 0$ for every wild component $D$.
    All in all, for $e \gg 0$, we have $\alpha_e^*K_{X^{(e)}} \cdot D_2 \cdot ...\cdot D_n <0$.

    Let $D_1:= M_X +kA$ and $D_{i,e}:= \beta_e^*D_i$ for $i= 1, ..., n$.
    Note that $\alpha_e^*D_{i,e}=p^e D_i$.
    Then,
    \[
    \alpha_e^*K_{X^{(e)}} \cdot p^e D_2 \cdot ... \cdot p^e D_n = \deg(\alpha_e) K_{X^{(e)}} \cdot D_{2,e} \cdot ... \cdot D_{n,e} <0.
    \]
    Therefore, for $e \gg 0$:
    \begin{itemize}
    \item[(a)] $D_{1,e} \cdot D_{2,e} \cdot ...\cdot D_{n,e}= \deg(\beta_e)D_1 \cdot ... \cdot D_n = 0$;
    \item[(b)] $K_{X^{(e)}} \cdot D_{2,e} \cdot ...\cdot D_{n,e} <0$.
    \end{itemize}
    Let $G_e := \beta_e^* G$, for some ample Cartier divisor $G$ on $X$.
    We apply \autoref{t-B&B} on $X^{(e)}$, $D_{1, e}, ..., D_{n,e}, G_e$ to find, through a general point of $X^{(e)}$, a rational curve $\zeta_e$ such that:
    \begin{itemize}
    \item[(i)] $
    G_e \cdot \zeta_e \leq 2 n \frac{G_e \cdot D_{2,e} \cdot ...\cdot D_{n,e}}{-K_{X^{(e)}} \cdot D_{2,e} \cdot ...\cdot D_{n,e}};
    $
    \item[(ii)] $D_{1,e} \cdot \zeta_e =0$. 
    \end{itemize}
    Let $\zeta_e^X$ be the image of $\zeta_e$ in $X$. By point (ii) and since $A$ is ample, $$M_X \cdot \zeta_e^X <0.$$
    
    To conclude, we show that there exists $e$ such that $\zeta_e^X$ is vertical.
    Assume, for the sake of a contradiction, that this was not the case.
    Let $p^{\varphi(e)}$ be the minimum between the purely inseparable degree of $f|_{\zeta_e^X}$ and $p^e$. Remark that, by \autoref{l-lemma!}, the degree of $\beta_e|_{\zeta_e}$ is $p^{e-\varphi(e)}$.
    We claim that $\varphi(e) = e - O(1)$, where $O(1) \geq 0$ is a bounded function.
    Indeed,
    \begin{align*}
    & p^e G \cdot p^e D_2 \cdot ... \cdot p^e D_n = \alpha_e^*G_e \cdot \alpha_e^*D_{2,e} \cdot ... \cdot \alpha_e^*D_{n,e} \\
    & = \deg(\alpha_e) G_e \cdot D_{2,e} \cdot ... \cdot D_{n,e}.
    \end{align*}
    Therefore:
    \begin{align*}
    & p^{e-\varphi(e)} \leq p^{e- \varphi(e)} G \cdot \zeta_e^X = G_e \cdot \zeta_e \leq 2 n \frac{G_e \cdot D_{2,e} \cdot ...\cdot D_{n,e}}{-K_{X^{(e)}} \cdot D_{2,e} \cdot ...\cdot D_{n,e}} \\
    & = 2 n \frac{p^e G \cdot D_2 \cdot ...\cdot D_n}{-(p^e -1)(M_X-B^h) \cdot D_2 \cdot ...\cdot D_n - K_X \cdot D_2 \cdot ...\cdot D_n + W_e \cdot D_2 \cdot ... \cdot D_n}\\
    & \leq 2 n \frac{p^e G \cdot D_2 \cdot ...\cdot D_n}{-(p^e -1)(M_X-B^h) \cdot D_2 \cdot ...\cdot D_n - K_X \cdot D_2 \cdot ...\cdot D_n},
    \end{align*}
    where the equality on the second line is given by \autoref{r-formulina}.
    The last line is bounded as $e$ grows, whence the claim.
    
    From the above chain of inequalities, we also infer that $G \cdot \zeta_e^X \leq p^{e- \varphi (e)} G \cdot \zeta_e^X$ is bounded. Therefore, $\{ \zeta_e^X \}_e$ is a bounded family of curves,
    hence, up to considering a sub-sequence of indices, it is eventually constant (up to numerical equivalence).
    Now, let $H$ be an ample Cartier divisor on $Z$, by the projection formula, $f^*H \cdot \zeta_e^X = d_e H \cdot f(\zeta_e^X)$, where $d_e$ is the degree of $f|_{\zeta_e^X}$. Since we can assume that $f^*H \cdot \zeta_e^X$ is constant, $d_e$ is bounded.
    However this contradicts the fact that $\varphi(e)=e-O(1)$.
\end{proof}

\begin{remark}
    The Bend and Break theorem needs the base field to be algebraically closed.
    If the base field $k$ is only perfect, the rational curves we find may be defined over an extension of $k$.
    Note that in this case, their Galois orbits descend to curves over $k$ which satisfy similar inequalities.
\end{remark}

\begin{remark}
    If $f\colon X \to Z$ is a separable equidimensional tame fibration between normal projective varieties, a similar proof shows a ``generic'' Bend and Break theorem for $K_{\cF}$, the canonical divisor of the foliation induced by $f$.
    However, if $f$ is not tame, the proof does not go through for $K_{\cF}$ due to the correction term given by the wild components.
\end{remark}

\section{The canonical bundle formula}

In this section, we prove positivity of the moduli part.
First, we prove that, for pairs which satisfy {Property $(\ast)$}, the moduli part is nef. Then, we use $(\ast)$-modifications to recover this situation.

In this section we use the notation of \autoref{construction} in \autoref{s-purelyinsep_basechange}.

\subsection{Base change}

Before stating the theorem, we prove that we can assume the field $k$ to be algebraically closed.

\begin{lemma} \label{l-perfect nefness}
Let $f\colon X \rightarrow Z$ be a fibration between normal projective varieties over a perfect field $k$ and $D$ a $\Q$-Cartier $\Q$-divisor on it.
Let $\widebar{k}$ be the algebraic closure of $k$, $\widebar{f} \colon \widebar{X} \rightarrow \widebar{Z}$ the base change of $f$ with $\widebar{k}$, and $\widebar{D}:= D \times_k \widebar{k}$.
Then, $D$ is $f$-nef if and only if $\widebar{D}$ is $\widebar{f}$-nef.
\end{lemma}

\begin{proof}
Let $G:= \mathrm{Gal}(\widebar{k}/k)$ and $\xi \subset \widebar{X}$ a curve. 
Note that $\widebar{f}(\xi)$ has dimension $0$ if and only if $f(\xi^G)$ has dimension $0$.
Since $D \cdot \xi^G$ is a positive multiple of $\widebar{D} \cdot \xi$, if $D$ is $f$-nef, then $\widebar{D}$ is $\widebar{f}$-nef as well.
The converse is trivial.
\end{proof}

\begin{lemma} \label{l-log canonical base change}
    Let $(X,B)$ be a pair over a perfect field $k$ and $(\widebar{X}, \widebar{B})$ its base change to the algebraic closure $\widebar{k}$.
    Then $(X,B)$ is log canonical if and only if $(\widebar{X}, \widebar{B})$ is.
\end{lemma}

\begin{proof}
    Note that, since $k$ is perfect, $\widebar{X}$ is normal by \cite[Tag 0C3M]{stacks}. Moreover, by \cite[Tag 01V0]{stacks}, $K_{\widebar{X}} = K_X \times_k \widebar{k}$.

    If $(\widebar{X}, \widebar{B})$ is log canonical, then $(X, B)$ is log canonical.
Indeed, let $Y \rightarrow X$ be a birational map over $X$ and $E$ an exceptional divisor.
Consider the base change with $\widebar{k}$, $\widebar{Y} \rightarrow \widebar{X}$, $\widebar{E} \subseteq \widebar{Y}$.
Since $k$ is perfect, $\widebar{E}$ is reduced by \cite[Tag 020I]{stacks}, thus the discrepancy of $E$ over $X$ coincides with the discrepancy of $\widebar{E}$ over $\widebar{X}$.

Conversely, assume for the sake of a contradiction that there is a non-log canonical place ${E' \subseteq Y' \xrightarrow{\sigma} \widebar{X}}$ over $\widebar{X}$ with discrepancy $a' <-1$.
Without loss of generality, we can assume $E', Y'$ and $\sigma$ are defined over $k'$, a finite Galois extension of $k$.
Let $R:=\cO_{Y', \eta_{E'}}$, where $\eta_{E'}$ is the generic point of $E'$, let $\mathfrak{m}_R:=(t)$, where $t \in k'(Y')$ is a local equation for $E'$, and let $v'$ be the valuation induced by $E'$ on $k'(Y')$.
Let $t_0=t, t_1, ..., t_d$ be the Galois conjugates of $t$.
Note that, since $k'(Y')$ is a separable extension of $k(X)$, $t_i \neq t_j$ for every $i \neq j$.
Then $v:=v'|_{k(X)}$ is a valuation on $k(X)$ whose associated DVR is $(S:=R \cap k(X), \mathfrak{m}_S:= (\prod_{i=1,...,d} t_i))$.
By \cite[Lemma 2.45]{KM}, there exists a proper birational map $\widetilde{X} \to X$ extracting a divisor $\widetilde{E}$ with induced valuation $v$. Let $a$ be its discrepancy and $\widetilde{X}' \to X'$, $\widetilde{E}'\subseteq \widetilde{X}'$ be the base change to $k'$. Then, $\widetilde{E}'$ is locally the zero locus of $\prod_{i=1,...,d} t_i$ and its discrepancy $\widetilde{a}$ coincides with $a$. 
Let $E_0$ be the zero locus of $t_0$. Since the local ring around $E_0$ is exactly $(R, \mathfrak{m}_R)$, by \cite[Remark 2.23]{KM}, its discrepancy coincides with $a'$.
Therefore, $a=\widetilde{a}=a'<-1$, which is a contradiction.
\end{proof}

\begin{corollary} \label{l-perfect moduli part ecc}
Assume inversion of adjunction in dimension $n$ and the existence of log resolutions in dimension $n-1$. Let $f\colon X \rightarrow Z$ be a fibration between normal varieties over a perfect field $k$, where $\dim(X)=n$. Let $(X/Z, B)$ be a GGLC pair associated with it with $B \geq 0$.
Suppose that $(X/Z, B)$ satisfies Property $(\ast)$.
Let $\widebar{k}$ be the algebraic closure of $k$, $\widebar{f} \colon \widebar{X} \rightarrow \widebar{Z}$ the base change of $f$ with $\widebar{k}$, $\widebar{B}:= B \times_k \widebar{k}$.
Then $(\widebar{X}/\widebar{Z}, \widebar{B})$ is GGLC, it satisfies Property $(\ast)$, $B_{\widebar{Z}} = B_Z \times_k \widebar{k}$ and $M_{\widebar{X}}= M_X \times_k \widebar{k}$.
\end{corollary}

\begin{proof}
If $Y$ is a variety over $k$, we denote by $\widebar{Y}:= Y \times_k \widebar{k}$. Similarly, if $D$ is a divisor on $Y$, we denote by $\widebar{D}:= D \times_k \widebar{k}$.
Let $G:= \mathrm{Gal}(\widebar{k}/k)$.

By \autoref{l-log canonical base change}, given $t \in \R_{\geq 0}$ and $\delta \subseteq Z$ prime divisor, if $(X, B +t f^*\delta)$ is log canonical around $\delta$, then $(\widebar{X}, \widebar{B}+ t\widebar{f}^*\widebar{\delta})$ is log canonical around $\widebar{\delta}$.

Conversely, given $t \in \R_{\geq 0}$, $\delta \subseteq \widebar{Z}$ prime divisor, if $(\widebar{X}, \widebar{B}+ t\widebar{f}^*\delta)$ is log canonical around $\delta$, then $(X, B+ t f^*\delta^G)$ is log canonical around $\delta^G$.
Indeed, let $Y \rightarrow X$ be a birational map over $X$ and $E$ a place over $f^*\delta^G$.
Consider the base change with $\widebar{k}$, $\widebar{Y} \rightarrow \widebar{X}$, $\widebar{E} \subseteq \widebar{Y}$.
Since $k$ is perfect, $\widebar{E}$ is reduced by \cite[Tag 020I]{stacks}, thus the discrepancy of $E$ over $X$ coincides with the discrepancy of $\widebar{E}$ over $\widebar{X}$.

To conclude the proof, note that it is enough to show that $B_{\widebar{Z}}= B_{Z} \times_k \widebar{k}$, which follows from the above discussion.
\end{proof}

\subsection{Semiample perturbations}

Bertini theorems are well-known over fields of characteristic $0$. In \cite{semiample}, the author proves that we can perturb log canonical and klt pairs over fields of positive characteristic with divisors coming from a semiample linear system without changing the singularities.

\begin{theorem}[{\cite[Theorem 1]{semiample}}] \label{t-Theorem 1-Tanaka}
    Let $k$ be an $F$-finite field containing an infinite perfect field $k_0$ of characteristic $p>0$.
    Let $X$ be a projective variety over $k$ and $(X, B)$ a log canonical (resp.\ klt) pair, with $B \geq 0$.
    Assume that there exists a log resolution of $(X,B)$.
    Let $M$ be a semiample $\bQ$-Cartier $\Q$-divisor on $X$.
    Then, for $m\gg 0$ and sufficiently divisible, there exists an effective $\Q$-divisor $\Gamma_m \sim mM$ such that $\left( X, B+ \frac{1}{m}\Gamma_m \right)$ is log canonical (resp.\ klt).
\end{theorem}

\begin{corollary} \label{c-corollary Tanaka Theorem 1}
    Let $k$ be an infinite perfect field of characteristic $p>0$ and let $k_0 \subseteq k$ be an infinite perfect subfield.
    Assume the existence of log resolutions in dimension $n$.
    Let $X$ be a projective variety over $k$ of dimension $n$ and $(X, B)$ a log canonical (resp.\ klt) pair, with $B \geq 0$.
    Let $M$ be a $\bQ$-Cartier $\Q$-divisor on $X$ and $|V_{\bullet}|:= (|V_m| \subseteq |mM|)_{m \in \N}$ a $\Q$-sub-linear system.
    Fix an integer $\widebar{m} \geq 1$, choose a basis $\cB:= \{ s_1,..., s_{\ell}\}$ of $V_{\widebar{m}}$ and define $V_{\widebar{m}}(k_0):= \left\{ \sum_{i=1}^{\ell}a_i s_i \; \text{s.t. } a_i \in k_0 \right\}$.
\begin{itemize}
    \item[(i)] Suppose that $V_{\widebar{m}}$ is base point free. Then, for $m\gg 0$ and sufficiently divisible, there exists an effective $\Q$-divisor $\Gamma_m \in |V_{\widebar{m}m}|$ such that $\Gamma_m$ can be decomposed as $\sum_{i=1}^m D_i$ with $D_i \in |V_{\widebar{m}}(k_0)|$ and $\left( X, B+ \frac{1}{\widebar{m}m}\Gamma_m \right)$ is log canonical (resp.\ klt).
    \item[(ii)] Suppose that $\Bs(V_{\widebar{m}})=W \subseteq X$ and let $W' \subseteq W$ be a subvariety.
    Then, for $m\gg 0$ and sufficiently divisible, there exists an effective $\Q$-divisor $\Gamma_m \in |V_{\widebar{m}m}|$ such that $\Gamma_m$ can be decomposed as $\sum_{i=1}^m D_i$ with $D_i \in |V_{\widebar{m}}(k_0)|$, and $\left( X, B+ \frac{1}{m}\Gamma_m \right)$ is log canonical (resp.\ klt) outside $W$ and has a non-klt centre at $W'$.
\end{itemize}
\end{corollary}

\begin{proof}
    Point (i) follows directly from the proof of \autoref{t-Theorem 1-Tanaka} noting the following.
    In the notation of the proof of \cite[Proposition 2, Proposition 3]{semiample}, we consider $T_1:= \A^{\ell}$ with basis $\cB$ and then we choose $\cB^{\dim(X)}$ as basis of $T:= T_1^{\dim(X)}$.
    With this choice, the statement follows.

    As for point (ii), let $\mu \colon Y \to X$ be a birational model such that $\mu^*|V_{\widebar{m}}|= |M_{\widebar{m}}|+\Phi$, where $|M_{\widebar{m}}|$ is base point free and $\Phi$ is the fixed divisor.
    By possibly passing to a higher model, we can assume $\mu$ is a log resolution of $(X,B)$ and that $\Phi$ contains a place over $W'$.
    Define an effective $\Q$-divisor $B_Y$ as
    \[
    K_Y+B_Y=\mu^*(K_X+B)+E,
    \]
    where $E$ is an effective $\mu$-exceptional divisor with no components in common with $B_Y$.
Let $\varphi$ be the rational function defining $\Phi$, then, for all $i=1,...,\ell$, there is a rational function $t_i$ such that $\mu^*s_i=\varphi t_i$.
Define $\cB':= \left\{ t_1,..., t_{\ell}\right\}$,
$W_1(k_0):= \left\{ \sum_{i=1}^{\ell}a_i t_i \; \text{s.t. } a_i \in k_0 \right\}$ and, for all $m \in \N$, let $W_{m}:=\mathrm{Sym}^m M_{\widebar{m}}$.
Note that $|W_1(k_0)|+\Phi = \mu^*|V_{\widebar{m}}(k_0)|$.
By point (i), for some integer $m \geq 1$, we find an effective $\Q$-divisor $\Gamma_{Y,m}$ that admits a decomposition as $\sum_{i=1}^m D_i$ with $D_i \in W_1(k_0)$ and such that $\left(Y, B_Y + \frac{1}{m}\Gamma_{Y,m}\right)$ is log canonical.
Therefore, $\left(Y, B_Y + \frac{1}{m}\Gamma_{Y,m}+ \Phi\right)$ is log canonical outside $\supp(\Phi)$ and has a non-klt place over $W'$. Since $D_i+\Phi \in \mu^*|V_{\widebar{m}}(k_0)|$, by the projection formula, there exists an effective $\Q$-divisor $\Gamma_m \in |V_{\widebar{m}m}|$ on $X$ which pulls-back to $\Gamma_{Y, m}+ m\Phi$ and satisfies the required properties.
\end{proof}

\subsection{Property \texorpdfstring{$(\ast)$}{} case}

We are now ready to prove that the moduli part is nef for fibrations satisfying Property $(\ast)$.

\begin{theorem} \label{t-goal star}
Assume the birational dlt LMMP, inversion of adjunction and the existence of log resolutions in dimension up to $n$.
Assume the dlt LMMP in dimensions $n',1$, for all $n' \leq n$.
Assume that $\mathrm{char}(k)= p>2$.
Let $f\colon X \rightarrow Z$ be a fibration from a normal projective variety $X$ of {dimension $n$} onto a normal projective curve $Z$ and $(X/Z, B)$ a GGLC pair associated with it with $B \geq 0$.
Suppose that $(X/Z, B)$ satisfies {Property $(\ast)$} and $X$ is $\Q$-factorial.
Assume that ${K_X+B}$ is $f$-nef, then the moduli part $M_X$ is nef.
\end{theorem}

\begin{proof}
    The strategy to prove \autoref{t-goal star} is inspired by the proof of \cite[Lemma 3.12]{positivity}.
    First of all, if $k$ is not algebraically closed, let $\widebar{k}$ be its algebraic closure, $\widebar{f} \colon \widebar{X} \rightarrow \widebar{Z}$ the base change of $f$ with $\widebar{k}$ and $\widebar{B}:= B \times_k \widebar{k}$.
    By \autoref{l-perfect nefness} and \autoref{l-perfect moduli part ecc}, $(\widebar{X}/\widebar{Z}, \widebar{B})$ is GGLC, it satisfies Property $(\ast)$ and $M_{\widebar{X}}= M_X \times_k \widebar{k}$ is $\widebar{f}$-nef.
    If we prove nefness of $M_{\widebar{X}}$, by \autoref{l-perfect nefness}, this implies that $M_X$ is nef as well.
    Therefore, we can assume that $k$ is algebraically closed.
    
    If $M_X$ was not nef, there would exist $\rho$, extremal ray, such that $M_X \cdot \rho <0$.
    Let $A$ be an ample divisor on $X$ such that $H_{\rho} := M_X+A$ is a supporting hyperplane for $\rho$.
    In particular, $H_{\rho}$ is nef.
    
    \noindent \textbf{Outline of the proof}
    
    \noindent \underline{Non big case.} Assume that $H_{\rho}$ is not big. Then we can find a negative curve that is general enough on which $M_X$ is negative. 
    We conclude by \autoref{p-non-big B vertical}.
    
    \noindent \underline{Big case.} If $H_{\rho}$ is big, we produce a geometric log canonical centre containing $\rho$.
    Thanks to \autoref{t-property star modifications v2.0}, we extract this centre over $X^{(e)}$, i.e.\ after a base change of $f$ with a high enough power of the Frobenius morphism.
    The aim is to prove nefness of the moduli part by induction on the dimension, thus we perform adjunction using \autoref{p-adjunction moduli}.
    Note that the case $\dim(X)=1$ is trivial.\\
    
    \noindent \textbf{Non big case.}
    Let $d$ be the numerical dimension of $H_{\rho}$ and assume that $d < n$.
    Define $D_i:= H_{\rho}$ for $2 \leq i \leq d +1$ and $D_i:= A$ for $d +1 < i \leq n$.
    Since $H_{\rho}^{d +1} \cdot A^{n - d -1}=0$, $M_X \cdot H_{\rho}^{d} \cdot A^{n-d -1} = (H_{\rho}-A) \cdot H_{\rho}^{d} \cdot A^{n-d -1} < 0$.
    Then,
    \begin{itemize}
    \item[(a)] $(M_X+A) \cdot D_2 \cdot ... \cdot D_n=0$;
    \item[(b)] $M_X \cdot D_2 \cdot ... \cdot D_n<0$;
    \item[(c)] $M_X+A$ is nef.
    \end{itemize}
    Therefore, by \autoref{p-non-big B vertical}, there exists a vertical curve $\zeta$ such that $M_X \cdot \zeta <0$, which is a contradiction.
    
    \noindent \textbf{Big case.}

    Now, assume that $H_{\rho}$ is big.
    
    \noindent \textbf{Step 1.} In this step, we produce a log canonical centre on the geometric generic fibre of $f$, containing a curve that is negative on the moduli part. More precisely, we find a horizontal curve $\xi \subseteq X$ and an effective $\Q$-divisor $\Gamma$ on $X$ such that:
    \begin{itemize}
        \item[(a)] $M_X \cdot \xi <0$;
        \item[(b)] $\Gamma \cdot \xi \leq 0$; 
        \item[(c)] $\xi|_{X_{\widebar{\eta}}}$ is a log canonical centre of $(X_{\widebar{\eta}}, B_{\widebar{\eta}}+\Gamma_{\widebar{\eta}})$.
    \end{itemize}
    
    By \cite[Theorem 1.1]{augmented_BS}, if $\varepsilon >0$ is small enough and $m\gg 0$ is divisible enough, ${\Bs(m(H_{\rho}-\varepsilon A))}$ is supported on those subvarieties $W \subseteq X$ such that $H_{\rho}|_W$ is not big. 
    Let $\sum_{i=1}^{\infty} a_i \zeta_i$ be a representative of the ray $\rho$, where $a_i >0$ and $\zeta_i$ are irreducible curves on $X$.
    Since $H_{\rho}$ is nef, $H_{\rho} \cdot \zeta_i=0$, therefore $\bigcup_{i=1}^{\infty} \zeta_i \subseteq {\Bs(m(H_{\rho}-\varepsilon A))}$.
    Note that, since $H_{\rho}$ is big, ${\Bs(m(H_{\rho}-\varepsilon A))}$ is a closed proper subset of $X$.
    Consider the restriction maps
    \[
    r_m \colon H^0(X, m(H_{\rho}-\varepsilon A)) \to H^0(X_{\widebar{\eta}}, m(H_{\rho}-\varepsilon A)|_{X_{\widebar{\eta}}}).
    \]
    Define the sub-linear system $|V_{\bullet}|:=(|m(H_{\rho}-\varepsilon A)|_{X_{\widebar{\eta}}})_{m \in \N}$ on $X_{\widebar{\eta}}$, then, for some $\widebar{m}>0$ sufficiently divisible, $\bigcup_{i=1}^{\infty} \zeta_i|_{X_{\widebar{\eta}}} \subseteq \Bs(V_{\widebar{m}})$.
    Let $V_{\widebar{m}}(k)$ be the image of $r_{\widebar{m}}$ and let $\{s_1, ..., s_{\ell}\}$ be a choice of basis of $V_{\widebar{m}}(k)$, considered as a $k$-vector space.
    Since $M_X$ is $f$-nef and $H_{\rho}-\varepsilon A \sim_{\Q} M_X+(1-\varepsilon)A$, there exists an index $i_0$ such that $\xi:=\zeta_{i_0}$ is a horizontal curve with $(H_{\rho}-\varepsilon A) \cdot \xi<0$ and $M_X \cdot \xi <0$.
    Let $\rho_{\widebar{\eta}}$ be the ray corresponding to $\rho$ in $X_{\widebar{\eta}}$ and $\xi_{\widebar{\eta}}:=\xi|_{X_{\widebar{\eta}}}$. 
    By \autoref{c-corollary Tanaka Theorem 1}, there exists an effective $\Q$-divisor $\widebar{\Gamma} \in |V_{\widebar{m}m}|$ such that $(X_{\widebar{\eta}}, B_{\widebar{\eta}}+\frac{1}{m}\widebar{\Gamma})$ is log canonical outside $\Bs(V_{\widebar{m}})$ and has a non-klt centre at $\xi_{\widebar{\eta}}$.
    Moreover, $\widebar{\Gamma}$ has a decomposition as $\sum_{i=1}^m D_i$ with $D_i \in |V_{\widebar{m}}(k)|$; in particular, $\widebar{\Gamma}$ belongs to the image of $r_{\widebar{m}m}$.
    Let $0\leq \lambda\leq \frac{1}{m}$ such that $(X_{\widebar{\eta}}, B_{\widebar{\eta}}+\lambda\widebar{\Gamma})$ has a log canonical centre at $\xi_{\widebar{\eta}}$ and is log canonical outside $\Bs(V_{\widebar{m}})$.
    Let $\Gamma' \in |\widebar{m}m(H_{\rho}-\varepsilon A)|$ be a lift of $\widebar{\Gamma}$ to $X$ and let $\Gamma:= \lambda \Gamma'^h$, we claim that it satisfies the required properties.
    Indeed, points (a) and (c) follow by construction and point (b) because $\Gamma \cdot \xi = {\lambda \widebar{m}m(H_{\rho}-\varepsilon A)\cdot \xi - \lambda\Gamma'^v\cdot \xi \leq 0}$. 
    
    \noindent \textbf{Step 2.} In this step, we perform a Frobenius base change on the base of the fibration in order to find a curve $\xi_e \subseteq X^{(e)}$ that is a log canonical centre of $(X^{(e)}, B_e+\Gamma_e)$, where $B_e:=\beta_e^*B^h+ \beta_e^{-1}B^v$ and $\Gamma_e := \beta_e^*\Gamma$.
    
    Let $\Delta:=B^h+\Gamma$ and $\Delta_e:= \beta_e^*\Delta$.
    Now, we want to compare the moduli part $M_X$ of $(X/Z, B)$ to $K_{X^{(e)}}$.
    By \autoref{t-formulina_general}, we have:
    \[
    \alpha_e^*K_{X^{(e)}} = (p^e-1)(K_X-f^*K_Z-R(f)) +K_X - \sum_{D \text{ wild}}w_{D,e}D,
    \]
    where $w_{D,e} \geq 0$ for every $D$ wild component.
    By \autoref{t-canonical foliation}:
    \begin{align*}
        &\alpha_e^*(K_{\cF_e}+ \Delta_e+W(f_e))=
     \alpha_e^*(K_{X^{(e)}}-f_e^*K_Z-R(f_e) + \Delta_e)\\
        &=(p^e-1)(K_X-f^*K_Z-R(f)) +K_X - \sum_{D \text{ wild}}w_{D,e}D - f^*K_Z - \alpha_e^*R(f_e) + p^e\Delta.
    \end{align*}
    Now, if $D$ is a vertical prime divisor in $X$, let $\ell_D=n_D p^{e_D}$ be its multiplicity with respect to $f$, where $e_D \geq 0$ and $n_D$ is coprime with $p$. By \autoref{c-ramifications}, for $e\gg 0$, there exist integers $1\leq h_D\leq e_D$ independent of $e$ and $\cW$ a subset of the set of wild components independent of $e$, such that:
    \[
        \alpha_e^*R(f_e) = R(f) - \sum_{D \in \cW} (p^{h_D}-1)D.
    \]
    Bringing everything together: 
    \begin{align*}
        &\alpha_e^*(K_{\cF_e}+ \Delta_e + W(f_e))= \\
        &(p^e-1)(K_X-f^*K_Z-R(f)) +K_X - f^*K_Z - R(f) + p^e \Delta\\ &- \sum_{D \text{ wild}}w_{D,e}D + \sum_{D \in \cW} (p^{h_D}-1)D=\\
        & p^e (M_X+ \Gamma) - \sum_{D \text{ wild}}w_{D,e}D + \sum_{D \in \cW} (p^{h_D}-1)D.
    \end{align*}
    Let $\xi_e$ be a curve corresponding to $\xi$ in $X^{(e)}$. Since ${\sum_{D \in \cW} (p^{h_D}-1)D \cdot \xi}$ is independent of $e$ for $e \gg 0$, up to choosing a bigger $e \geq 0$:
    \[
    (K_{\cF_e}+\Delta_e+W(f_e)) \cdot \xi_e <0.
    \]
    
    \noindent \textbf{Step 3}. In this step, we construct a geometric {$(\ast)$-modification} that extracts a log canonical place over $\xi_e$ and we compute its moduli part.
    
    Up to possibly choosing an even bigger $e \geq 0$, by \autoref{t-property star modifications v2.0}, there exists a dlt GGLC pair $(Y/Z, C+ E)$ with $Y$ $\Q$-factorial, and a diagram:
    \[
    \xymatrix{
    Y \ar[r]^{\mu} \ar[dr]_{g} & X^{(e)} \ar[d]^{f_e} \\
     & Z,
    }
    \]
    where $\mu$ is a birational map and the centre of $E$ is $\xi_e$.
    The induced fibration $E \rightarrow Z$ is separable since by construction $E_{\widebar{\eta}}$ is reduced.
    Moreover, there exist an effective exceptional $\Q$-divisor $R$, supported on non-log canonical places of $(X^{(e)},\Delta_e+ B_e^v)$, and a vertical effective $\Q$-divisor $G$, such that
    \[
    K_Y+ C +E + R = \mu^*(K_{X^{(e)}}+\Delta_e + B_e^v) + G,
    \]
    where $K_Y+C+E$ is $\mu$-nef.
    
    Now, we want to compare the moduli part $M_X$ of $(X/Z, B)$ to the moduli part $M_Y$ of ${(Y/Z, C+E)}$.
    Recall that all the vertical divisors with multiplicity $>1$ lie over the discriminant part $B_Z$ of $(X/Z, B)$, which is reduced by Property $(\ast)$. Thus, by \autoref{t-canonical foliation}:
    \[
    K_{\cF_e} + \Delta_e +W(f_e) = K_{X^{(e)}} + \Delta_e+ B_e^v - f_e^*(K_Z+B_Z).
    \]
    Therefore:
    \begin{align*}
    & \mu^*(K_{\cF_e}+ \Delta_e+W(f_e))=\\
    & K_Y+C+E-g^*(K_Z+C_Z)+R+ g^*(C_Z-B_Z)-G=\\
    & M_Y+R+g^*(C_Z-B_Z)-G,
    \end{align*}
    where $C_Z$ is the discriminant part of $(Y/Z,C+E)$.
    
    Now, we study more closely the $\Q$-divisor $g^*(C_Z-B_Z)-G$.
    Let $z \in Z$ and $b_z$ and $c_z$ be its coefficients in $B_Z$ and $C_Z$, respectively.
    If $b_z=1$, then the fibre over $z$ is contained in $B^v$, so $c_z=1$ as well by construction.
    If $b_z=0$, it may happen that $c_z=1$.
    Let $D$ be a non $\mu$-exceptional vertical prime divisor contained in $g^{-1}(z)$.
    By abuse of notation, call $D$ also $\beta_e(\mu (D))$.
    Then, by construction, the coefficient of $D$ in $G$ is $1$ and, since $(X/Z,B)$ has {Property $(\ast)$} and $b_z=0$, the coefficient of $D$ in $f^*(z)$ must be $1$, hence its coefficient in $g^*(z)$ is $1$ as well ($\mu$ is an isomorphism at the generic point of $D$).
    All in all, we get that $g^*(C_Z-B_Z)-G$ is $\mu$-exceptional.
    Then, applying the Negativity lemma (\cite[Lemma 3.39]{KM}), we conclude that $R+g^*(C_Z-B_Z)-G$ is effective.
    By abuse of notation, let $\xi \subseteq \supp(E)$ be a curve mapping to $\xi_e$ and such that $\xi \not\subseteq \Supp(R)$.
    Since $\xi$ is horizontal, $(R^v+g^*(C_Z-B_Z)-G) \cdot \xi \geq 0$.
    Moreover, $R^h \cdot \xi \geq 0$,
    thus:
    \[
    M_Y \cdot \xi <0.
    \]
    
    \noindent \textbf{Step 4.} In this step we do adjunction on $E$.
    
    Since $(Y,C+E)$ is $\Q$-factorial and dlt, the normalisation morphism $\nu\colon E^{\nu} \rightarrow E$ is an isomorphism in codimension $1$ by \autoref{l-Snormal}.
    Define $C_{E^{\nu}}$ on $E^{\nu}$ by adjunction, so that we have $(K_Y+C+E)|_{E^{\nu}}=K_{E^{\nu}}+C_{E^{\nu}}$, and let $E^{\nu} \xrightarrow{g_E} Z' \xrightarrow{\varphi} Z$ be the Stein factorisation of $g|_{E^{\nu}}$.
    By \autoref{p-property star adjunction}, $(E^{\nu}/Z',C_{E^{\nu}})$ is GGLC and satisfies Property $(\ast)$. Moreover, by \autoref{p-adjunction moduli}, there exists a vertical effective divisor $V$ on $E^{\nu}$ such that $M_Y|_{E^{\nu}} - V =M_{E^{\nu}}$, the moduli part of $(E^{\nu}/Z', C_{E^{\nu}})$.
    By abuse of notation, let $\xi$ be a curve in $E^{\nu}$ mapping onto $\xi$.
    Then,
    $
    M_{E^{\nu}}\cdot \xi <0.
    $
    
    If $\dim(X)=2$, $E^{\nu} \to Z'$ is the identity, whence $M_{E^{\nu}}=0$, giving a contradiction.
    Assume that $\dim(X) \geq 3$.
    Note that, if $\zeta \subseteq \supp(E^{\nu})$ is vertical over $Z'$, $\nu(\mu(\zeta))$ is a point because $\mu(E)$ is supported on a horizontal curve.
    Therefore, ${(K_Y+C+E)|_{E^{\nu}}\cdot \zeta \geq 0}$.

    \noindent \textbf{Step 5.} In this step we conclude by induction on the dimension.

    If $E^{\nu}$ is not $\Q$-factorial, consider a geometric $(\ast)$-modification $(E'/Z', C')$ constructed with \autoref{t-property star modifications v2.0}.
    Let $e \gg 0$, $\mu'\colon E' \to E^{\nu, (e)}$ be the resulting map, $g_e \colon E^{\nu, (e)} \to Z'$ and $g'\colon E' \to Z'$ the induced fibration.

    As in Step 2, if $\cG_e$ is the foliation induced by $E^{\nu, (e)} \to Z'$, $M_{E^{\nu}}$ is the moduli part of $(E^{\nu}/Z', C_{E^{\nu}})$ and $C_e:=\beta_e^*C_{E^{\nu}}^h+\beta_e^{-1}C_{E^{\nu}}^v$, then: 
    \[
    \alpha_e^*(K_{\cG_e}+C_e^h+W(g_e)) = p^e M_{E^{\nu}} - \sum_{D \text{ wild}} w_{D,e}D + \sum_{D \in \cW'}(p^{h'_D}-1)D,
    \]
    where $w_{D,e} \geq 0$ for every $D$ wild component and $\ell'_D=n'_Dp^{e'_D}$ is their multiplicity with respect to $g$, $\cW'$ is a subset of the set of wild components and $1\leq h'_D \leq e'_D$ for all $D \in \cW'$.
    Moreover, for $e\gg 0$, $\cW'$ and $h'_D$ are independent of $e$.
    Therefore, if $\xi_e$ is a curve mapping onto $\xi$, for $e \gg 0$,
    \[
        (K_{\cG_e}+C_e^h+W(g_e)) \cdot \xi_e <0.
    \]
    Moreover, note that $K_{\cG_e}+C_e^h$ is $g_e$-nef over a dense open subset of $Z'$ since $M_{E^{\nu}}$ is $g_E$-nef. 

    As in Step 3,
\[
{\mu'}^*(K_{\cG_e}+C_e^h+W(g_e))=M_{E'}+Q,
\]
where $M_{E'}$ is the moduli part of $(E'/Z', C')$ and $Q$ is a vertical effective $\Q$-divisor.
In particular, if $\xi'$ is a curve mapping onto $\xi$, $M_{E'} \cdot \xi'<0$.
Moreover, $M_{E'}$ is $g'$-nef over a dense open subset $U$ of $Z'$. 

  
    If $M_{E'}$ is not $g'$-nef, run a $(K_{E'}+ C')$-MMP over $Z'$ and let $(E''/Z', C'')$ be the resulting pair with associated fibration $g''\colon E'' \rightarrow Z'$.
    Since $K_{E'}+C'$ is $g'$-nef over $U$, $E'$ and $E''$ are isomorphic on $g'^{-1}(U)$.
    In particular, $\xi'$ is not contracted and the MMP does not terminate with a Mori fibre space.
    Let $\xi''$ be the image of $\xi'$ in $E''$.
    Let $p_1\colon \widetilde{E} \to E'$ and $p_2\colon \widetilde{E} \to E''$ be the projections from a resolution of $E' \dashrightarrow E''$.
    Then, $p_2^*(K_{E''}+C'')=p_1^*(K_{E'}+C')-D''$, where $D''$ is an effective {$p_2$-exceptional} $\Q$-divisor.
    We can choose $p_2$ to be an isomorphism on $g'^{-1}(U)$, therefore $D''$ is vertical over $Z'$.
    By \autoref{l-normal_afterMMP} and \autoref{p-properties property star}, $(E''/Z', C'')$ is GGLC, it satisfies Property $(\ast)$ and its discriminant part coincides with the discriminant part of $(E'/Z', C')$, whence:
    \[
        M_{E''} \cdot \xi'' =M_{E'}\cdot \xi' - D'' \cdot \widetilde{\xi} <0,
    \]
    where $M_{E''}$ is the moduli part of $(E''/Z', C'')$ and $\widetilde{\xi}$ is a lift of $\xi'$ to $\widetilde{E}$.
    
    To sum up, $(E''/Z', C'')$ satisfies Property $(\ast)$, $E''$ is $\bQ$-factorial, $M_{E''}$ is $g''$-nef and there is a horizontal curve $\xi'' \subseteq E''$ such that $M_{E''} \cdot \xi'' <0$.
    By the inductive assumption $M_{E''}$ is nef, which is a contradiction.
\end{proof}

\begin{remark}
    In the statement of \autoref{t-goal star}, we are assuming that $p>2$ to avoid the presence of purely inseparable nodes. Note that, if $\dim(X) \leq 3$ and $X_{\wb{\eta}}$ is normal, then we do not encounter this problem. Indeed, when doing adjunction, the dimension of the geometric generic fibre becomes at most $1$, therefore, even in characteristic $2$ purely inseparable nodes do not appear.
    In particular, the conclusions of \autoref{t-goal star} hold also if $p=2$ and $\dim(X) \leq 3$.
    Moreover, we expect the same statement to hold in general, even if the characteristic of the base field is $2$ and if we only assume that $(X_{\wb{\eta}}^{\nu}, B_{\wb{\eta}}^{\nu})$ is log canonical. However, we would need to include a conductor divisor supported on the inseparable nodes in all our computations and we do not work out the details here.
    The same remark applies for \autoref{t-goal} and \autoref{t-ftrivial}.
\end{remark}

\subsection{General case}

In the general case, we may need to go to a higher model of $X$ to achieve positivity of the moduli part.

\begin{definition} \label{d-generically crepant}
    Let $(X/Z, B)$ and $(X'/Z', B')$ be GGLC pairs over a perfect field of any characteristic such that the associated fibrations $f\colon X \to Z$ and $f'\colon X' \to Z'$ are birationally equivalent. We say that $(X/Z,B)$ and $(X'/Z',B')$ are \textbf{crepant over the generic point of $Z$} if, given birational morphisms $p_1\colon \widetilde{X} \to X$ and $p_2 \colon \widetilde{X} \to X'$ which resolve the indeterminacies of $X \dashrightarrow X'$, we have that
    \[
        p_1^*(K_X+B)-p_2^*(K_{X'}+B')
    \]
    is vertical with respect to the induced fibration $\widetilde{X} \to Z$.
\end{definition}

\begin{theorem} \label{t-goal}
Assume the birational dlt LMMP, inversion of adjunction and the existence of log resolutions in dimension up to $n$.
Assume the dlt LMMP in dimensions $n',1$, for all $n' \leq n$.
Assume that $\mathrm{char}(k)= p>2$.
Let $f\colon X \rightarrow Z$ be a fibration from a normal projective variety $X$ of dimension $n$ onto a normal projective curve $Z$ and $(X/Z, B)$ a GGLC pair associated with it such that $B \geq 0$ and $(X,B)$ is log canonical.
Moreover, if $\eta$ is the generic point of $Z$, assume that $X_{\eta}$ is $\Q$-factorial and $(X_{\eta}, B_{\eta})$ is dlt.
Suppose that $K_X+B$ is $f$-nef.
Then, there exist a projective pair $(Y, C)$ satisfying {Property $(\ast)$}, with $C \geq 0$, and a commutative diagram
\[
\xymatrix{
Y \ar@{-->}[r]^{b} \ar[dr]_{g} & X \ar[d]^f \\
 & Z,
}
\]
with $b$ birational and such that:
\begin{itemize}
\item[(i)] the pairs $(X/Z, B)$ and $(Y/Z, C)$ are crepant over the generic point of $Z$;
\item[(ii)] the moduli part $M_Y$ of $(Y/Z, C)$ is nef.
\end{itemize}
\end{theorem}

\begin{proof}
By \autoref{t-property star modifications v1.2}, there exists $(X', B')$ a $\Q$-factorial dlt pair satisfying Property $(\ast)$ with $B' \geq 0$ and a commutative diagram
\[
\xymatrix{
X' \ar[r]^{\mu} \ar[dr]_{f'} & X \ar[d]^f \\
& Z,
}
\]
where $\mu$ is a projective birational which is an isomorphism over a dense open subset $U$ of $Z$ and $K_{X'}+ B'= \mu^*(K_X+B) +G$ with $G$ vertical effective $\Q$-divisor.
Since $\mu$ is an isomorphism over $U$, the geometric generic fibre $(X'_{\widebar{\eta}}, B'_{\widebar{\eta}})$ is log canonical.
The divisor $K_{X'}+ B'$ may no longer be $f'$-nef.
In this case, we run a $(K_{X'}+ B')$-MMP over $Z$.
Note that this is an isomorphism on ${f'}^{-1}(U) \setminus \Supp(G)$. In particular, this MMP cannot end with a Mori fibre space and it is an isomorphism on the generic fibre.
Let $(Y, C)$ be the resulting pair, $b\colon Y \dashrightarrow X$ the induced birational map and $g \colon Y \to Z$ the induced fibration.
Since $(K_{X'}+B')|_{X'_{\eta}} = (K_Y+C)|_{X'_{\eta}}$, $(Y/Z, C)$ and $(X/Z, B)$ are crepant over the generic point of $Z$.
Moreover, $(Y/Z,C)$ satisfies Property $(\ast)$ by \autoref{p-properties property star}, $Y$ is $\Q$-factorial and $K_Y+C$ is $g$-nef.
Thus, we conclude by applying \autoref{t-goal star}.
\end{proof}

\subsection{The \texorpdfstring{$K$}{}-trivial case}

As a corollary of the previous results we get the canonical bundle formula in the classical setting, i.e.\ when the fibration is $K$-trivial.
The proof is very similar to \cite[Theorem 1.3]{positivity}.

\begin{theorem} \label{t-ftrivial}
Assume the birational dlt LMMP, inversion of adjunction and the existence of log resolutions in dimension up to $n$.
Assume the dlt LMMP in dimensions $n',1$, for all $n' \leq n$.
Assume that $\mathrm{char}(k)= p>2$.
Let $f\colon X \rightarrow Z$ be a fibration from a normal projective variety $X$ of dimension $n$ onto a normal projective curve $Z$ and $(X/Z, B)$ a GGLC pair associated with it such that $B \geq 0$ and $(X,B)$ is log canonical.
Moreover, if $\eta$ is the generic point of $Z$, assume that $X_{\eta}$ is $\Q$-factorial and $(X_{\eta}, B_{\eta})$ is dlt.
Assume that $K_X+B \sim_{\Q} f^*L$ for some $\Q$-Cartier $\Q$-divisor $L$ on $Z$.
Let $M_Z:= L-(K_Z+B_Z)$, where $B_Z$ is the discriminant part of $(X/Z, B)$.
Then, $M_Z$ is nef.
\end{theorem}

\begin{proof}
Let $M_X=f^*M_Z$ be the moduli part of $(X/Z,B)$.
Let $\mu\colon X' \rightarrow X$ be a $(\ast)$-modification as in \autoref{t-property star modifications v1.2}.
Let ${f'\colon X' \rightarrow Z}$ be the induced morphism and $B' \geq 0$ defined so that $K_{X'}+ B'= \mu^*(K_X+B)$.
Let $B'_Z$ and $M'$ be respectively the discriminant and the moduli parts of $(X'/Z, B')$.
Then $M'=\mu^*M_X\sim_{\Q} f'^*L'$, for some $\Q$-Cartier $\Q$-divisor $L'$ on $Z$, so it is enough to show that $M'$ is nef.
There exists an effective vertical $\Q$-divisor $G$ such that, if $B^* := B'+G$, $(X'/Z, B^*)$ satisfies Property $(\ast)$. Let $\Sigma_Z$ be the discriminant part of $(X'/Z, B^*)$ and $M^*$ its moduli part.
Let $\gamma^*_z$ be the log canonical threshold of $f'^*(z)$ with respect to $(X'/Z, B^*)$ and $\gamma'_z$ the one with respect to $(X'/Z, B')$, for $z \in Z$.
Note that $\gamma^*_z \leq \gamma'_z$. In particular, $\Sigma_Z \geq B'_Z$.

Now, define $B'':= B'+ f'^*(\Sigma_Z-B'_Z)$, so that $K_{X'}+B'' \sim_{\Q} f'^*L''$, for some $\Q$-Cartier $\Q$-divisor $L''$ on $Z$. The discriminant part of $(X'/Z, B'')$ is $\Sigma_Z$ and its moduli part is $M'$.
Note that $B'' \leq B^*$ and the difference is vertical.

It is possible that $K_{X'}+B^*$ is not $f'$-nef, in which case we perform a $(K_{X'}+ B^*)$-MMP over $Z$.
Let $\varphi\colon X' \dashrightarrow Y$ be the result of this MMP and $(Y, C)$ the resulting pair, where $C:= \varphi_*B^*$. Note that this MMP cannot end with a Mori fibre space since $K_{X'}+B'$ is $f'$-nef.
Call $g\colon Y \rightarrow Z$ the induced fibration and $M_Y$ the moduli part of $(Y/Z, C)$.
Since $(X, B^*)$ is log canonical, so is $(Y, C)$ and, by \autoref{p-properties property star}, $(Y/Z,C)$ satisfies Property $(\ast)$ and its discriminant part is $\Sigma_Z$.
Let $C'':= \varphi_*B''$. Since $K_{X'}+B'' \sim_{\Q} f'^*L''$, the divisor $(K_Y+C)-(K_Y+C'')$ is $g$-nef. Moreover, by construction, it is effective and supported on a vertical $\Q$-divisor $\Phi$ which does not contain any fibre. 
Let $\Psi$ be an effective vertical $\Q$-divisor such that $\Phi + \Psi \sim_{\Q} g^*D$, for some $\Q$-divisor $D$ on $Z$.
Then, $(K_Y+C)-(K_Y+C'') \sim_{\Q} g^*D -\Psi$ is $g$-nef, whence $\Phi=0$ and $C=C''$.
The pair $(Y/Z, C)$ satisfies Property $(\ast)$ and $M_Y$ is $g$-nef by construction.
Therefore, $M_Y$ is nef by \autoref{t-goal star}.

Now, we want to compare $M_Y$ and $M'$.
Let $\widetilde{X}$ be a common resolution of $X' \dashrightarrow Y$ with $p\colon \widetilde{X} \rightarrow X'$, $q\colon \widetilde{X} \rightarrow Y$ the induced projections.
We have that both $p^*M'-q^*M_Y$ and $q^*M_Y-p^*M'$ are $q$-nef since $M'$ is $f'$-trivial.
Moreover, $q_*(p^*M'-q^*M_Y)= \varphi_*B''-C=0=q_*(q^*M_Y-p^*M')$.
Therefore, by the Negativity lemma (see \cite[Lemma 3.39]{KM}), $p^*M' = q^*M_Y$, and it is nef, whence the conclusion.
\end{proof}

\subsection{The canonical bundle formula in dimension \texorpdfstring{$3$}{}}

 For threefolds over algebraically closed fields of characteristic $p>5$, the LMMP, inversion of adjunction and existence of log resolutions are known to hold, so our results hold unconditionally.

\begin{corollary} \label{c-3folds}
Let $f\colon X \rightarrow Z$ be a fibration from a normal projective variety $X$ of dimension $3$ onto a normal projective curve $Z$ over an algebraically closed field of characteristic $p>5$, and $(X/Z, B)$ a GGLC pair associated with it such that $B \geq 0$ and $(X,B)$ is log canonical.
Moreover, if $\eta$ is the generic point of $Z$, assume that $X_{\eta}$ is $\Q$-factorial and $(X_{\eta}, B_{\eta})$ is dlt.
Let $B$ be an effective $\Q$-divisor on $X$ such that $(X,B)$ is a log canonical pair.
Suppose that $K_X+B$ is $f$-nef.
Then, there exist a pair $(Y, C)$ which satisfies {Property $(\ast)$}, with $C \geq 0$, and a commutative diagram
\[
\xymatrix{
Y \ar@{-->}[r]^{b} \ar[dr]_{g} & X \ar[d]^f \\
 & Z,
}
\]
with $b$ birational such that
\begin{itemize}
\item[(i)] $(X/Z, B)$ and $(Y/Z, C)$ are crepant over the generic point of $Z$;
\item[(ii)] the moduli part $M_Y$ of $(Y/Z, C)$ is nef.
\end{itemize}
Moreover, if $K_X+B \sim_{\Q} f^*L$ for some $\Q$-Cartier $\Q$-divisor $L$ on $Z$, define $M_Z:= L-(K_Z+B_Z)$, where $B_Z$ is the discriminant part of $(X/Z, B)$.
Then $M_Z$ is nef.
\end{corollary}

\begin{proof}
By \autoref{t-MMP}, \autoref{r-inversion} and \autoref{r-resolutions}, we conclude that \autoref{t-goal star}, \autoref{t-goal} and \autoref{t-ftrivial} hold unconditionally for threefolds over a perfect field of characteristic $p>5$.
\end{proof}

\bibliographystyle{alpha}	
\bibliography{biblio}

\end{document}